%% file: main.tex
\documentclass[12pt]{amsart}
\input{preamble}

\title[
Cobordism maps on instanton cube complexes
]{
    Cobordism maps in Khovanov homology and singular instanton homology I}
\author{Hayato Imori, Taketo Sano, Kouki Sato, Masaki Taniguchi}

\newcommand{\printaddresses}{{\parindent0pt
  \bigskip
  Hayato Imori, \textsc{Department of Mathematical Sciences, KAIST, Daejeon, 34141, Republic of Korea}\\
  \textit{E-mail address}: \url{himori@kaist.ac.kr}\par
  \medskip
  Taketo Sano, \textsc{RIKEN Center for Interdisciplinary Theoretical and Mathematical Sciences(iTHEMS), RIKEN, Wako 351-0198, Japan}\\
  \textit{E-mail address}: \url{taketo.sano@riken.jp}\par
  \medskip
  Kouki Sato, \textsc{Meijo University,Tempaku, Nagoya 468-8502, Japan}\\
  \textit{E-mail address}: \url{ satokou@meijo-u.ac.jp}\par
  \medskip
  Masaki Taniguchi, \textsc{Department of Mathematics, Graduate School of Science, Kyoto University, Kitashirakawa Oiwake-cho,
Sakyo-ku, Kyoto 606-8502, Japan}\par
  \textit{E-mail address}: \url{taniguchi.masaki.7m@kyoto-u.ac.jp}
}}

\begin{document}

\begin{abstract}
    Khovanov homology and singular instanton Floer homology are both functorial with respect to link cobordisms. Although the two homology groups are related by a spectral sequence, direct correspondence between the cobordism maps has not been rigorously established. In this paper, we define a cobordism map on the instanton cube complex as a filtered chain map, and prove that it recovers the cobordism maps both in Khovanov homology and singular instanton theory. 
    In a sequel paper, we further extend this cobordism map to immersed cobordisms. 
\end{abstract}

\maketitle

\input{1}
\input{2}
\input{3}

\input{4}
\input{5}

\input{A}

\printbibliography
\printaddresses

\end{document}

%% file: preamble.tex
\usepackage{graphicx} 
\usepackage{amsmath, amssymb}
\usepackage{amsthm}
\usepackage{amscd}
\usepackage{color}
\usepackage[margin=3cm, marginpar=3cm]{geometry}
\usepackage{comment} 
\usepackage{appendix}
\usepackage{tikz-cd}

\usepackage{caption} \usepackage{subcaption}
\theoremstyle{plain}
\usepackage{cleveref}
 \usepackage[all]{xy}
 \usepackage{amscd}
 \usepackage{color}
 \usepackage{enumitem}
\newlist{steps}{enumerate}{1}
\setlist[steps, 1]{label = Step \arabic*:}
\usepackage[fontsize=10]{scrextend}
\usepackage{bm}
\usepackage{afterpage}

\usepackage[
    backend=biber,
    style=alphabetic,
    sorting=nyt,
    doi=false,
    url=false,
    isbn=false,
    maxbibnames=10,
]{biblatex}

\DeclareFieldFormat[misc]{title}{``#1"}

\makeatletter
\DeclareRobustCommand\widecheck[1]{{\mathpalette\@widecheck{#1}}}
\def\@widecheck#1#2{%
   \setbox\z@\hbox{\m@th$#1#2$}%
   \setbox\tw@\hbox{\m@th$#1%
      \widehat{%
         \vrule\@width\z@\@height\ht\z@
         \vrule\@height\z@\@width\wd\z@}$}%
   \dp\tw@-\ht\z@
   \@tempdima\ht\z@ \advance\@tempdima2\ht\tw@ \divide\@tempdima\thr@@
   \setbox\tw@\hbox{%
      \raise\@tempdima\hbox{\scalebox{1}[-1]{\lower\@tempdima\box\tw@}}}%
   {\ooalign{\box\tw@ \cr \box\z@}}}
\makeatother

\theoremstyle{plain}
\newtheorem{thm}{Theorem}[section]
\crefname{thm}{Theorem}{Theorems}
\Crefname{thm}{Theorem}{Theorems}
\newtheorem{prop}[thm]{Proposition}
\crefname{prop}{Proposition}{Propositions}
\Crefname{prop}{Proposition}{Propositions}
\newtheorem{lem}[thm]{Lemma}
\crefname{lem}{Lemma}{Lemmas}
\Crefname{lem}{Lemma}{Lemmas}
\newtheorem{cor}[thm]{Corollary}
\crefname{cor}{Corollary}{Corollaries}
\Crefname{cor}{Corollary}{Corollaries}
\newtheorem {claim}[thm]{Claim}
\crefname{claim}{Claim}{Claims}
\Crefname{claim}{Claim}{Claims}

\crefname{property}{Property}{Properties}
\Crefname{property}{Property}{Properties}

\crefname{problem}{Problem}{Problems}
\Crefname{problem}{Problem}{Problems}
\newtheorem{ques}[thm]{Question}
\crefname{ques}{Question}{Questions}
\Crefname{ques}{Question}{Questions}

\theoremstyle{definition}
\newtheorem{defn}[thm]{Definition}
\crefname{defn}{Definition}{Definitions}
\Crefname{defn}{Definition}{Definitions}

\crefname{notation}{Notation}{Notations}
\Crefname{notation}{Notation}{Notations}

\crefname{convention}{Convention}{Conventions}
\Crefname{convention}{Convention}{Conventions}

\crefname{cond}{Condition}{Conditions}
\Crefname{cond}{Condition}{Conditions}

\crefname{assum}{Assumption}{Assumptions}
\Crefname{assum}{Assumption}{Assumptions}

\crefname{conj}{Conjecture}{Conjectures}
\Crefname{conj}{Conjecture}{Conjectures}

\theoremstyle{remark}
\newtheorem{rem}[thm]{Remark}
\crefname{rem}{Remark}{Remarks}
\Crefname{rem}{Remark}{Remarks}

\crefname{ex}{Example}{Examples}
\Crefname{ex}{Example}{Examples}

\crefname{section}{Section}{Sections}
\Crefname{section}{Section}{Sections}
\crefname{subsection}{Subsection}{Subsections}
\Crefname{subsection}{Subsection}{Subsections}
\crefname{figure}{Figure}{Figures}
\Crefname{figure}{Figure}{Figures}

\newcommand{\Z}{\mathbb{Z}}
\newcommand{\F}{\mathbb{F}}

\newcommand{\Q}{\mathbb{Q}}

\newcommand{\Int}{\mathrm{Int}}

\newcommand{\rank}{\mathop{\mathrm{rank}}\nolimits}
\newcommand{\Hom}{\mathop{\mathrm{Hom}}\nolimits}

\newcommand{\id}{\mathrm{id}}

\def\om{\omega}

\newcommand{\R}{\mathbb R}

\newcommand{\ctext}[1]{\raise0.2ex\hbox{\textcircled{\scriptsize{#1}}}}

\def\dim{\operatorname{dim}}
\def\rank{\operatorname{rank}}

\def\Hom{\operatorname{Hom}}

\def\id{\operatorname{Id}}

\newcommand{\mbar}[1]{{\ooalign{\hfil#1\hfil\crcr\raise.167ex\hbox{--}}}}

\DeclareMathOperator{\gr}{gr}

\def\u{\mathfrak{u}}

\def\bF{\mathbf{F}}

\DeclareMathOperator{\lk}{lk}

\newcommand{\CKh}{\mathit{CKh}}
\newcommand{\Kh}{\mathit{Kh}}

\newcommand{\Cube}{\operatorname{Cube}}

\def\bC{\mathbf{C}}
\def\bv{\mathbf{v}}

\DeclareMathOperator{\ord}{ord}
\def\mF{\mathcal{F}}

%% file: 1.tex
\section{Introduction}

\textit{Khovanov homology} is a link homology theory introduced by Khovanov in \cite{Khovanov:2000} as a categorification of the Jones polynomial. In \cite{KM11u}, Kronheimer and Mrowka introduced  \textit{framed singular instanton Floer homology} and constructed a spectral sequence%
\footnote{
    For conventional reasons, the input link $L$ for Khovanov homology must be mirrored as $L^*$.
}
having Khovanov homology as its $E_2$ term and abutting to singular instanton homology,
\[
    \Kh (L^*)  \Rightarrow I^\sharp (L)
\]
which led to the proof that \textit{Khovanov homology detects the unknot}%
\footnote{
    In \cite{KM11u}, Kronheimer and Mrowka also construct a spectral sequence for the reduced versions $\widetilde{\Kh}(L^*) \Rightarrow I^\natural (L)$, and the main theorem is deduced from the fact that $\rank \widetilde{\Kh}(K^*) = 1$ implies $\rank I^\natural (K) = 1$ for a knot $K$.
}. 
The two theories are contrasting: Khovanov homology is defined combinatorially and admits direct computations, whereas singular instanton homology is defined analytically and is strongly tied with the geometry of the knot. 
Thus connecting the two theories enhances the strengths of both. For example, the above-stated result implies that the unknottedness of any given knot can be determined algorithmically. \footnote{See \cite{HN13, BSX18, BS22, Ma20, BS22nonfiber, BDLLS21, LS22, BS24} for further results of detections of links. }

The above spectral sequence was obtained by constructing the \textit{instanton cube complex} $\CKh^\sharp (D)$ defined for a given diagram $D$ of a given link $L\subset \R^3$.  
Its homology gives singular instanton  homology $I^\sharp(L)$, together with the \textit{instanton homological filtration} such that the Khovanov complex $\CKh(D^*)$ naturally arises in the $E_1$ term of the induced spectral sequence. Lately in \cite{KM14}, Kronheimer and Mrowka also introduced the \textit{instanton quantum filtration} on $\CKh^\sharp(D)$ and proved that Khovanov homology $\Kh(L^*)$ arises in the $E_1$ term of the induced spectral sequence. Furthermore, for a (possibly non-orientable) link cobordism $S \subset [0,1]\times \R^3$ from $L$ to $L'$ and for diagrams $D$ and $D'$ of $L$ and $L'$, they constructed a \textit{cobordism map} between the instanton cube complexes
\[
    \phi^{KM}_S : \CKh^\sharp (D) \to \CKh^\sharp (D'),
\]
which is a doubly filtered chain map of orders \footnote{In general, given abelian groups with a decreasing filtration and a map $\phi$, we define $\ord(\phi) := \max\{k \in \Z \mid \phi(\mF^i) \subset \mF^{i+k}\}$.} 
\[
 (\operatorname{ord}_h( \phi^{KM}_S), \operatorname{ord}_q( \phi^{KM}_S))    \geq \left(\frac{1}{2}(S \cdot S),\ \chi(S)+\frac{3}{2}(S \cdot S)\right) 
\]
such that its induced map on homology coincides with the cobordism map on singular instanton homology,
\[
    I^\sharp(S): I^\sharp(L) \rightarrow I^\sharp(L').
\]
It is questioned therein, whether the induced map of $\phi^{KM}_S$ on the $E_2$ term (or the $E_1$ term when considering quantum filtration) coincide with the cobordism map $\Kh(S^*)$ of Khovanov homology.\footnote{
    $S^*$ denotes the image of $S$ under $\id \times r$, where $r$ is the reflection on $\R^3$ that gives the mirroring of links. 
}

Meanwhile in \cite{BHL19}, Baldwin, Hedden and Lobb gave a general framework to construct a spectral sequence for any link homology theory that satisfies a set of axioms called the {\it Khovanov--Floer theory}. In particular for singular instanton theory, their construction give rise to a cobordism map $\phi^{KF}_S$ between  (the quasi-isomorphism classes of) the filtered chain complexes. It is proved that the induced map on the $E_2$ term coincides with $\Kh(S^*)$ (\cite[Theorem 3.5]{BHL19}), but it is unclear whether the induced map on homology coincides with $I^\sharp(S)$.
(Their construction of $\phi^{KF}_S$ relies on an algebraic lemma \cite[Lemma 2.4]{BHL19}, which does not have a connection with the geometric construction of $I^\sharp(S)$.)
The following diagram depicts these gaps in the functoriality of the spectral sequence.

\[
\begin{tikzcd} 
& \Kh(S^*) &\\
\phi^{KM}_S \arrow[rd, "\text{induce}"'] \arrow[ru, "\text{induce?}", dashed] & & \phi^{KF}_S \arrow[lu, "\text{induce}"'] \arrow[ld, "\text{induce?}", dashed] \\
& I^\sharp(S) &          
\end{tikzcd}
\]

In this paper, we fill these gaps by constructing another filtered chain map $\phi^\sharp_S$ on the instanton cube complex $\CKh^\sharp$. The idea is to combine the two approaches: for each elementary move $S$ (i.e.\ a Reidemeister move or a Morse move), we further decompose $S$ into elementary moves $S_i$ as in \cite[Section 4]{BHL19}, then compose the corresponding cobordism maps $\phi^{KM}_{S_i}$ of \cite{KM14}. We will prove that this filtered chain map $\phi^\sharp_S$ induces both $\Kh(S^*)$ and $I^\sharp(S)$. To be precise, 

\begin{thm} \label{thm:main}
    Given a link cobordism $S \subset [0,1]\times \R^3 $ from $L$ to $L'$ and diagrams $D$ and $D'$ of $L$ and $L'$ respectively, there is a doubly filtered chain map 
    \[
        \phi^\sharp_S : \CKh^\sharp (D) \to \CKh^\sharp (D')
    \]
    of order
    \[
        \geq \left(\frac{1}{2}(S \cdot S),\ \chi(S)+\frac{3}{2}(S \cdot S)\right)
    \]
    whose induced map on the $E_2$ term with respect to the homological filtration (resp.\ the $E_1$ term with respect to the quantum filtration) coincides with the cobordism map of Khovanov homology
    \[
        \Kh(S^*): \Kh(L^*) \to \Kh(L'^*)
    \]
    and whose induced map on homology coincides with the cobordism map of singular instanton homology
    \[
        I^\sharp(S) : I^\sharp (L) \to I^\sharp (L').
    \]
    Here, $\chi(S)$ denotes the Euler characteristic of $S$ and $S \cdot S$ the normal Euler number.  
\end{thm}

\begin{rem}
    Both $\Kh(S^*)$ and $I^\sharp(S)$ are well-defined up to overall sign, so the statement in \Cref{thm:main} should also be regarded up to overall sign. For a cobordism $S$ consisting only of Morse moves and planar isotopies, similar results have been proved in \cite[Proposition 3.3]{LZ20}.
\end{rem}
For the proof of \Cref{thm:main}, that $\phi^\sharp_S$ induces $I^\sharp(S)$ is almost immediate from the functoriality (isotopy invariance) of $I^\sharp$. The main difficulty lies in proving that $\phi^\sharp_S$ induces $\Kh^\sharp(S^*)$, which is proved by careful observations on the spectral sequence, together with technical lemmas (\Cref{excision lemma,disjoint lemma,compatibility with Khovanov}) stating the behaviors of $\CKh^\sharp$ and its filtrations under the operation of taking the disjoint unions of links.

The following questions are left as future work. 

\begin{ques}
    Do our cobordism map $\phi_S^\sharp$ and Kronheimer--Mrowka's map $\phi^{KM}_S$ coincide up to filtered chain homotopy? 
\end{ques}

\begin{ques}
    Is the cobordism map $\phi_S^\sharp$ isotopy (rel boundary)  invariant up to filtered chain homotopy? 
\end{ques}

\subsection*{Connection with Khovanov-Floer theories}

Various link homology theories admit spectral sequences starting from Khovanov homology, such as the Heegaard Floer homology of branched covers \cite{OS05}, instanton homologies \cite{KM11u, Da15}, the monopole Floer homology of branched covers \cite{B11}, the framed instanton homology of branched covers \cite{Sca15}, Heegaard knot Floer homology \cite{Do24} and real monopole Floer homology \cite{Li24}. A formal treatment of such spectral sequences is given in \cite{BHL19} as {\it Khovanov--Floer theories}. Namely, for any link homology theory $\mathcal{H}$ (over $\F_2$) that arises as the homology of a filtered chain complex satisfying a set of axioms called the {\it Khovanov--Floer theory}, it is proved that it gives rise to a functor
\[
    E^{\mathcal{H}}_*: \operatorname{Link} \to \operatorname{Spec}_{\F_2} 
\]
where $\operatorname{Link}$ denotes the category of link cobordisms and $\operatorname{Spec}_{\F_2} $ denotes the category of spectral sequences over $\F_2$, such that for any link $L$ it gives $E^\mathcal{H}_2(L) \cong \Kh(L^*; \F_2)$ and $E^\mathcal{H}_\infty(L) \cong \mathcal{H}(L)$, and for any link cobordism $S$ the map $E^\mathcal{H}_2(S)$ coincides with $\Kh(S^*; \F_2)$. It is then proved that many of the aforementioned link homology theories (over $\F_2$) are actually Khovanov--Floer theories. 

Here we remark that, even if there exists a filtered chain map $\phi_S$ that induces $\mathcal{H}(S)$ on homology, one cannot tell whether $E^{\mathcal{H}}_\infty(S)$ coincides with the map induced from $\phi_S$.
Now, for the case $\mathcal{H} = I^\sharp$, since our map $\phi^\sharp$ is proved to induce $\Kh(S^*)$ on the $E_2$ term, the consequent induced maps necessarily coincide with the maps obtained from the general framework of Khovanov-Floer theory. Thus \Cref{thm:main} can be rephrased as follows.

\begin{prop}
\label{prop:kf-coincide}
    Let
    \[
        E^\sharp_*: \operatorname{Link} \to \operatorname{Spec}_{\F_2}     
    \]
    be the functor obtained from the Khovanov-Floer theory given by the homological filtration on $\CKh^\sharp$ over $\F_2$. For any link cobordism $S$, the morphism $E^\sharp_*(S)$ coincides (from the $E_2$ term and after) with the morphism induced from the filtered chain map $\phi^\sharp_S$ given in \Cref{thm:main}. In particular, $E^\sharp_\infty(S)$ coincides with the map induced from $\phi^\sharp_S$. 
\end{prop}

\begin{ques}
    Do analogous statements of \Cref{prop:kf-coincide} hold for other Khovanov-Floer theories such as Heegaard (monopole-tilde, framed instanton) Floer homology for double-branched covers, Heegaard knot Floer homology, plane Floer homology, and real monopole Floer homology?  
\end{ques}

The proof of \cref{thm:main} essentially uses the quantum filtration on $\CKh^\sharp$. Thus we may ask a more specific question, 

\begin{ques}
Do other Khovanov–Floer theories admit maps on cube complexes endowed with a bifiltration, having a similar property to \cref{thm:main}?
\end{ques}

Note that a quantum grading is also introduced in the plane Floer homology in \cite{Dae15}. 

\subsection*{On immersed cobordisms.} 

In \cite{KM13}, Kronheimer and Mrowka extended the cobordism map of $I^\sharp$ to \textit{immersed surfaces} in $[0,1]\times S^3$ using the blow-up construction at the double points originally observed in \cite{Kr97, KM11}. In a sequel paper, we will extend the cobordism map of $\Kh$ to immersed surfaces, and prove that the two maps correspond under the spectral sequence $\Kh \Rightarrow I^\sharp$. 

\subsection*{On local systems}

The aim in \cite{KM13} was to derive an integer-valued knot invariant $s^\sharp$ from $I^\sharp$ with the local system $\Q[T^{\pm 1}]$ (or its certain completion $\Q[[\lambda]]$), as an instanton gauge theoretic analogue of the Rasmussen invariant $s$ derived from (a variant of) Khovanov homology \cite{Ras10}. It is now known that $s$ and $s^\sharp$ are distinct \cite{Gong21} but the relation between the two invariants remains unknown. We expect that our immersed cobordism map can be extended to the setup with local systems, and would potentially lead to understanding the relation between $s$ and $s^\sharp$. 

\subsection*{Organization}

In \cref{Algebraic part of the proofs}, we summarize the algebraic statements needed to prove \Cref{thm:main}. Therein, we also give the explicit description of the filtered chain map $\phi^\sharp_S$. 
In \cref{Preliminaries for instanton theory}, we review the construction of the instanton cube complexes and their homological and quantum filtrations. We also introduce the {\it excision cobordism map} in \cref{Excision cobordism map}, which is the key ingredient of the proof. We prove several fundamental properties of the cobordism maps. 
In \cref{Disjoint formula}, we give the proofs of algebraic lemmas introduced in \cref{Algebraic part of the proofs} and complete the proof of the main result. 
In \cref{ori app}, we discuss orientations of parametrized singular instanton moduli spaces to define the maps over $\Z$. 

\subsection*{Acknowledgements}

The authors would like to thank John Baldwin for some questions regarding \cite{BHL19}. 
The first author and the fourth author gratefully acknowledge support from the Simons Center for Geometry and Physics, Stony Brook University at which the discussion with John has occurred. 

HI is partially supported by the Samsung Science and Technology Foundation (SSTF-BA2102-02) and the Jang Young Sil Fellowship from KAIST.
TS was partially supported by JSPS KAKENHI Grant Number 23K12982 and academist crowdfunding. MT was partially supported
by JSPS KAKENHI Grant Number 22K13921. 

%% file: 2.tex
\section{Algebraic part of the proofs}\label{Algebraic part of the proofs}

\subsection{Khovanov complex} 
\label{subsec:Khovanov}

First, we give a brief review of the construction of Khovanov homology, as defined in \cite{Khovanov:2000}. The reader may skip this section if they are familiar with the basic setup. 

Let $D$ be a link diagram with $n$ crossings. Here we assume that an ordering of the crossings is fixed. Each crossing admits a 0-resolution and a 1-resolution, as depicted in  \Cref{fig:1}. 

\begin{figure}[t]
	\centering
    \includegraphics[scale=0.4]{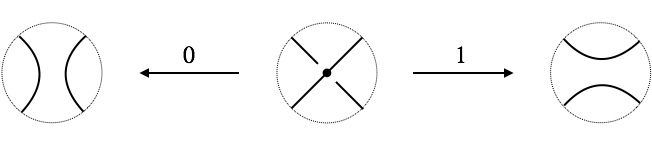}
    \vspace{1em}
	\caption{0-, 1-resolution of a crossing.}
	\label{fig:1}
\end{figure}

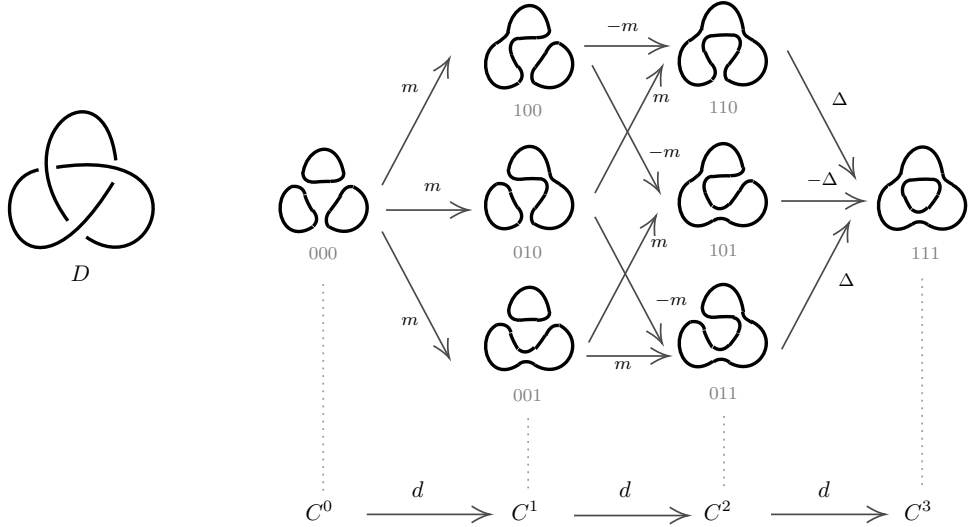
\begin{figure}[t]
	\centering
    \resizebox{.9\textwidth}{!}{
    \input{tikzcd/ckh}
    }
	\caption{A diagram $D$, its cube of resolutions $\Cube(D)$ and the complex $C(D)$.}
    \label{fig:CKh}
\end{figure}

A simultaneous choice of resolutions for all crossings is called a \textit{state}, which may be identified with an element $v \in \{0, 1\}^n$. The set $\{0, 1\}^n$ is endowed a partial order $\leq$, declared as $u \leq v$ if $u_i \leq v_i$ for each $1 \leq i \leq n$. Any state $v$ yields a diagram $D_v$, consisting of disjoint circles by resolving all crossings of $D$ accordingly. Let $r(D_v)$ denote the number of circles in $D_v$. The $1$-norm of a state $v$ is denoted $|v|_1$ and is called the \textit{weight} of $v$. Two states $u, v$ are \textit{adjacent} if $u < v$ and $|v - u|_1 = 1$. For adjacent states $u, v$, passing from $D_u$ to $D_v$ can be seen as performing a band surgery to the circle(s) of $D_u$ along the corresponding crossing, resulting in either two circles merging into one circle, or one circle splitting into two circles. Let $S_{uv}$ denote the $2$-dimensional cobordism from $D_u$ to $D_v$ that realizes this surgery. By considering all possible states, we obtain a \textit{cube of resolutions} $\Cube(D)$ for $D$, where on each vertex $v \in \{0, 1\}$ the resolved diagram $D_v$ is placed, and on each edge between adjacent vertices $u, v$ the cobordism $S_{uv}: D_u \rightarrow D_v$ is placed. $\Cube(D)$ may be regarded as a commutative cube in the category of $(1+1)$-dimensional cobordisms. \Cref{fig:CKh} depicts $\Cube(D)$ for a trefoil diagram $D$. 

Next, we transform $\Cube(D)$ into a commutative cube in the category of $R$-modules, where $R$ is any commutative ring with unity. (In this paper, we only consider the case $R = \Z$.) Let $V$ denote the free $R$-module generated by two distinct elements $\bv_+, \bv_-$. The $R$-module $V$ is given a \textit{Frobenius algebra} structure with \textit{unit} $\iota: R \rightarrow V$, 
\[
    \iota(1) = \bv_+,
\]
\textit{multiplication} $m: V \otimes V \rightarrow V$,
\begin{gather*}
  	m(\bv_+ \otimes \bv_+) = \bv_+, \\
  	m(\bv_+ \otimes \bv_-) = m(\bv_- \otimes \bv_+) = \bv_-, \\
  	m(\bv_- \otimes \bv_-) = 0,
\end{gather*}
\textit{counit} $\epsilon: V \rightarrow R$,
\[
    \epsilon(\bv_+) = 0,\quad
    \epsilon(\bv_-) = 1, 
\]
and \textit{comultiplication} $\Delta: V \rightarrow V \otimes V$,
\[
    \Delta(\bv_+) = \bv_+ \otimes \bv_- + \bv_- \otimes \bv_+,\quad 
    \Delta(\bv_-) = \bv_- \otimes \bv_-.
\]
We remark that $V$ originally comes from the truncated polynomial ring $R[X]/(X^2)$ and $\bv_+, \bv_-$ are elements corresponding to $1, X$ respectively.

Now, for each state $v$, we define the \textit{vertex module} $V(D_v)$ as the $r(D_v)$-fold tensor product of $V$. Here, $V(D_v)$ is generated by $2^{r(D_v)}$ elements of the form,
\[
    x_1 \otimes \cdots \otimes x_{r(D_v)},
    \quad x_i \in \{ \bv_\pm \}
\]
each of which can be identified with a simultaneous labeling of $\bv_+$ or $\bv_-$ on the circles of $D_v$, called an \textit{enhanced state} of $D$. Note that $V(\emptyset) = R$. 
For each pair of adjacent states $u, v$, we define the \textit{edge map} 
\[
    d_{uv}\colon V(D_u) \rightarrow V(D_v)
\]
as follows: depending on whether $S_{uv}\colon D_u \rightarrow D_v$ is a merge or a split, apply the multiplication $m$ or the comultiplication $\Delta$ to the label(s) on the corresponding circle(s), while leaving other labels unchanged. These assignments may be regarded as a tensor functor $\mathcal{F}$ from the category of $(1 + 1)$-cobordisms to the category of $R$-modules, i.e.\ a $(1 + 1)$-\textit{TQFT}, from which it follows that the resulting cube $\mathcal{F}(\Cube(D))$ is commutative. 

Next, we turn this commutative cube into a skew-commutative one, by taking a \textit{sign assignment} $e$, which is an assignment $e_{uv} \in \{ \pm 1\}$ for each pair of adjacent states $u, v$, such that for any square
\[
    \begin{tikzcd}
    u \arrow[r] \arrow[d] & v \arrow[d] \\
    v' \arrow[r]          & w          
    \end{tikzcd}
\]
we have
\[
    e_{vw} e_{uv} + e_{v'w} e_{uv'} = 0.
\]
For instance, the \textit{standard sign assignment} $e$ is defined as follows: for adjacent vertices $u, v$ whose components differ at index $i$, 
\[
    e_{uv} = (-1)^{\sum_{j < i} u_j}.
\]

Although sign assignments are not unique, given any two sign assignments $e, e'$, there is a unique vertex-wise transformation $f$ from $e$ to $e'$. This can be seen by regarding any sign assignment $e$ as a $1$-cochain of the cellular cochain complex of the $n$-dimensional cube $K(n) = [0, 1]^n$ over $\mathbb{F}_2$,
\[
    e \in C^1(K(n); \mathbb{F}_2)
\]
such that its coboundary $\delta e$ is the $2$-cochain that evaluates any $2$-cell $\sigma$ of $K(n)$ to $1$. Since $\delta(e - e') = 0$ and $K(n)$ is acyclic, there is a $0$-cochain $f$ such that $e - e' = \delta f$. See \cite[Definition 4.5]{LS14}.

Now, given any sign assignment $e$, by assigning each edge $d_{uv}$ the sign $e_{uv}$, the resulting cube becomes skew-commutative. By folding the cube into a sequence by taking direct sums over vertex modules having states of equal weights, we obtain a sequence of modules
\[
    C^i(D) = \bigoplus_{|v|_1 = i} V(D_v)
\]
and a sequence of maps defined by the sum of the edge maps
\[
    d^i = \sum_{\substack{u < v, \\ |v - u|_1 = 1}} e_{uv} d_{uv} .
\]
From the skew-commutativity of the cube, it follows that $d \circ d = 0$, hence obtain a chain complex $(C(D) ,d)$. Up to chain isomorphism, $C(D)$ is independent of the choice of the sign assignment $e$. Indeed, if we take another sign assignment $e'$, the above explained $0$-cochain $f$ induces a chain isomorphism between the corresponding complexes. Later, we shall fix a sign assignment that is compatible with that of the instanton cube complex $\CKh^\sharp$. See \Cref{prop:equal_to_Kh}. 

Each enhanced state is endowed a bigrading so that the complex $C(D)$ is bigraded with differential $d$ of bidegree $(1, 0)$. First, the module $V$ is given a grading so that $\bv_\pm$ has degree $\pm 1$ respectively. Let $n_+, n_-$ denote the number of positive, negative crossings of $D$ respectively. For an enhanced state 
\[
    x = x_1 \otimes \cdots \otimes x_r
\]
belonging to $V(D_v)$, its \textit{homological grading} is defined as 
\begin{align}\label{kh_ho_gr}
    \gr_h(x) = |v|_1 - n_-
\end{align}
and its \textit{quantum grading} is defined as 
\begin{align}\label{kh_q_gr}
    \gr_q(x) 
    &= |v|_1 + \sum_i \deg(x_i) + n_+ - 2n_- \\
    &= \gr_h(x) + \sum_i \deg(x_i) + n_+ - n_-.    
\end{align}
With the explicit definitions of $m$ and $\Delta$, it can be easily verified that the differential $d$ has bidegree $(1, 0)$, as claimed. The chain complex $(C(D), d)$ endowed with this bigrading is called the \textit{Khovanov complex}, and is denoted $(\CKh(D), d)$. Its homology is called the \textit{Khovanov homology} of $D$, and is denoted $\Kh(D)$.  

\begin{thm}[{\cite[Theorem 1]{Khovanov:2000}}] 
\label{thm:kh-inv}
    The isomorphism class of $\Kh(D)$ (as a bigraded $R$-module) is invariant under the Reidemeister moves.
\end{thm}

Thus for any link $L$ with diagram $D$, it is justified to refer to $\Kh(D)$ as the \textit{Khovanov homology} of $L$ and denote it by $\Kh(L)$. The following properties are well known, but we give proofs to clarify the correspondence between $\CKh$ and the later defined instanton cube complex $\CKh^\sharp$.

\begin{prop}[{\cite[(167), Proposition 32]{Khovanov:2000}}] 
\label{prop:ckh-isoms}
    There are canonical isomorphisms,
    \begin{enumerate}
        \item $\CKh(D \sqcup D') \cong \CKh(D) \otimes \CKh(D')$,
        \item $\CKh(D^*) \cong \CKh(D)^*$. 
    \end{enumerate}
    Here, $D^*$ denotes the mirror of $D$, and $\CKh(D)^*$ denotes the algebraic dual of $\CKh(D)$ with bigrading given by $(\CKh(D)^*)^{i, j} = (\CKh^{-i,-j}(D))^*$.
\end{prop}

\begin{rem}
    To be precise about signs, for (1), having fixed the signing convention for tensor products of complexes, the sign assignments for any two out of the three complexes determines one for the other. For (2), the sign assignment of $\CKh(D)$ determines that of $\CKh(D^*)$ under the given isomorphism. 
\end{rem}

\begin{proof} \
\begin{enumerate}
    \item If we take the standard sign assignments for the three complexes and the standard signing convention for the tensor product, one can see that the desired isomorphism becomes an identity, where each pair of enhanced states $x$ of $D$ and $y$ of $D'$ corresponds bijectively to an enhanced state $x \otimes y$ of $D \sqcup D'$.
    \item First, note that the Frobenius algebra $V$ is \textit{self-dual}, i.e.\ the \textit{dual Frobenius algebra} $V^* = \Hom_R(V, R)$ endowed with unit $\epsilon^*$, multiplication $\Delta^*$, counit $\iota^*$ and comultiplication $m^*$, is isomorphic to $V$ as Frobenius algebras under the correspondence $\bv_\pm^* \mapsto \bv_\mp$, where $\{\bv^*_\pm\}$ denotes the basis of $V^*$ dual to the basis $\{\bv_\pm\}$ of $V$. Let $\varphi$ denote this self-dual isomorphism $V \xrightarrow{\sim} V^*$. Now, observe that the cube of resolutions $\Cube(D^*)$ can be obtained from $\Cube(D)$ by replacing each state $v$ with $\bar{v}$ where $\bar{v}_i = 1 - v_i$, so that $D_v = D^*_{\bar{v}}$, and reversing each edge $S_{uv}\colon D_u \rightarrow D_v$ to $\bar{S}_{uv}\colon D^*_{\bar{u}} \leftarrow D^*_{\bar{v}}$. From these observations, one can see that the correspondence
    \[
        x = x_1 \otimes \cdots \otimes x_r \in V(D^*_v)
            \ \mapsto \ 
        \varphi(x) = \varphi(x_1) \otimes \cdots \otimes \varphi(x_r) \in V(D_{\bar{v}})^*
    \]
    gives an isomorphism $\CKh(D^*) \cong \CKh(D)^*$. Moreover, with $n_\pm(D^*) = n_\mp(D)$, one can confirm that 
    \begin{align*}
        \gr_h(\varphi(x)) 
        &= -(|\bar{v}|_1 - n_-) \\
        &= |v|_1 - n_+ \\
        &= \gr_h(x)
    \end{align*}
    and similarly
    \begin{align*}
        \gr_q(\varphi(x)) 
        &= -(|\bar{v}|_1 - \sum_i \deg(x_i) + n_+ - 2n_-) \\
        &= |v|_1 + \sum_i \deg(x_i) + n_- - 2n_+ \\
        &= \gr_q(x). 
        \qedhere
    \end{align*}
\end{enumerate}  
\end{proof}

\subsection{Instanton cube complex}
In this section, we summarize the algebraic setting of the instanton cube complexes, which are constructed in \cite{KM11u,KM14} due to Kronheimer and Mrowka.

 For an oriented link diagram $D$ with $n$ crossings of a given link $L$,  the chain homotopy equivalence class of 
 a $\Z/4$-graded chain complex
\[CKh^\sharp(D)= \bigoplus_{v \in \{0,1\}^n}C_v\] 
(where each $C_v$ is a $\Z/4$-graded free $\Z$-module)
with differential
\[
d^\sharp = \bigoplus_{\text{\footnotesize{$\begin{matrix}
 v,u \in \{0,1\}^n \\ v > u \end{matrix}$}}} d_{vu}
\]
is associated. 
Note that the maps 
\[
    d_{vu}\colon C_v \to C_u
\]
runs in the opposite direction compared to the edge maps of the Khovanov complex $\CKh$. The precise explanations are given in \cref{Preliminaries for instanton theory},
while we mention here that there exists a canonical isomorphism
\begin{align}
\label{eq:Phi}
\Phi_* \colon I^\sharp(L) \to H_*(CKh^\sharp(D))
\end{align}
as $\Z/4$-graded $\Z$-modules.
%
Moreover, $CKh^\sharp(D)$ has two additional structures;
{\it the $h$-filtration} and {\it the $q$-filtration}.
We first discuss the $h$-filtration.
\begin{defn}
{\it The $h$-filtration} on $\CKh^\sharp(D)$ is defined by
\[
h|_{C_v} = -|v|_1 + n_-.  
\]
\end{defn}
For each $i \in \Z$, set 
${\displaystyle \mF_i := \bigoplus_{h|_{C_v}\geq i}C_v}$. Then {\it the $h$-grading} of an element $x \in \CKh^\sharp(D)$ is defined by 
\[
\gr_h(x):= \max\{k \in \Z \mid x \in \mF_k\} \quad \in 
\quad \Z \cup \{\infty\}.
\] 
The order of chain maps is defined as the usual way: 
\begin{defn}
For a map $f \colon \CKh^\sharp(D) \to \CKh^\sharp(D')$,
{\it the $h$-order} of $f$ is defined by
\[
\ord_h(f) := \max\{ k \in \Z \mid \gr_h(f(x)) \geq \gr_h(x) + k \text{ for $\forall x \in \CKh^\sharp(D)$}\} 
\quad \in \quad \Z \cup \{\infty \}.
\]
\end{defn}
We see that $\ord_h(f) = \infty$ iff $f=0$. Moreover, it is shown in \cite{KM14} that $\ord_h(d^\sharp) \geq 1$. In particular, 
for the spectral sequence 
$\{E^r(CKh^\sharp(D))\}$ with respect to
the filtration $\{\mF_i\}$,
we have
\[
(E^0(\CKh^\sharp(D)),E^0(d^\sharp)) = (\CKh^\sharp(D),0)
\]
and
\begin{equation}   
\label{eq:d^1}
(E^1(\CKh^\sharp(D)),E^1(d^\sharp)) = \left(\CKh^\sharp(D),
\quad d^{\sharp 1} := \bigoplus_{v > u, \  |v-u|_1 = 1 } d_{vu} \right).
\end{equation}

Next, we discuss the $q$-filtration.
For a diagram $D$ with an orientation, each component $C_v$ of $\CKh^\sharp(D)$ is equipped with an identification 
\[
\gamma_v \colon C_v \to V^{\otimes r(D_v)},
\]
where $V=\langle \bv_+, \bv_- \rangle$ is a free abelian group. For the precise definition of the identification, see \cref{Preliminaries for instanton theory}.   
Make $V$ a graded abelian group by putting $\bv_+$ and $\bv_-$ in degrees 1 and $-1$ respectively, and give $V^{\otimes r(D_v)}$ the tensor-product grading.
Let $Q$ denote the grading on $C_v$ induced from $V^{\otimes r}$ via $\gamma_v$.
\begin{defn}
{\it The 
q-filtration} on $\CKh^\sharp(D)$ is defined by
\[q|_{C_v} = Q- |v|_1  - n_+ + 2 n_-.  \]
\end{defn}
For each $i \in \Z$, set 
\[
\mF^q_iC_v := \langle x \in C_v \mid q(x) \geq i \rangle 
\quad \text{and} \quad
\mF^q_i := \bigoplus_{v \in \{0,1\}^n} \mF^q_iC_v.
\]
Then {\it the q-grading} of an element $x \in \CKh^\sharp(D)$ is defined by 
\[
\gr_q(x):= \max\{k \in \Z \mid x \in \mF^q_k\} \quad \in 
\quad \Z \cup \{\infty\}.
\] 
Similar to the $h$-order, we also define the $q$-order:
\begin{defn}
For a map $f \colon \CKh^\sharp(D) \to \CKh^\sharp(D')$,
{\it the $q$-order} of $f$ is defined by
\[
\ord_q(f) := \max\{ k \in \Z \mid \gr_q(f(x)) \geq \gr_q(x) + k \text{ for $\forall x \in \CKh^\sharp(D)$}\} 
\quad \in \quad \Z \cup \{\infty \}.
\]
\end{defn}
We see that $\ord_q(f) = \infty$ iff $f=0$.
Moreover, it is shown in \cite[Lemma 10.1]{KM14} that 
\[
\left(qE^0(\CKh^\sharp(D)),qE^0(d^\sharp) \right) = \left(\CKh^\sharp(D), d^{\sharp 1}\right)
= \left(E^1(\CKh^\sharp(D)),E^1(d^\sharp) \right),
\]
where 
$\{qE^r(CKh^\sharp(D))\}$ denotes the spectral sequence with respect to the filtration $\{\mF^q_i\}$,
and the chain complex
$\left(\CKh^\sharp(D), d^{\sharp 1}\right)$ is that given in the equality (\ref{eq:d^1}).

We record a relation to the Khovanov complex here. 
\begin{prop}
\label{prop:equal_to_Kh}
Using the notations given in the proof of \Cref{prop:ckh-isoms}, we may regard the map $\gamma_v$ as 
\[
    \gamma_v \colon C_v \to V(D^*_{\bar{v}}).
\]
Then the direct sum
\[
\gamma = \bigoplus_{v} \gamma_v \colon
\left(\CKh^\sharp(D), d^{\sharp 1} \right)
\to (\CKh(D^*),d)
\]
is a chain isomorphism
as $\Z/2$-graded chain complexes,
where $d^{\sharp 1}$ is the map in the equality (\ref{eq:d^1}) and the $\Z/2$-gradings are the reductions of the $\Z/4$-grading on $CKh^\sharp$ and the $h$-grading on $CKh$ respectively. 
Moreover, for any $x \in \CKh^\sharp(D)$, we have
\[
\gr_h(\gamma(x))= \gr_h(x) \quad \text{and} \quad
\gr_q(\gamma(x)) = \gr_q(x).
\]
\end{prop}

Hereafter, we fix the sign assignment $e$ for $\CKh(D^*)$ so that edge maps correspond identically under $\gamma$. 
We summarize the following two properties of Kronheimer--Mrokwa's link cobordism maps defined in \cite{KM14}.

\begin{prop}[\text{\cite[Theorem 1.2]{KM14}}]
\label{prop:KM_map_R_moves}
For each Reidemeister move $Rn^{\varepsilon} \colon D \to D'$ ($n=1,2,3$, $\varepsilon = \pm 1$), there exists a chain homotopy equivalence map
\[
\phi^{KM}_{Rn^{\varepsilon}} : \CKh^\sharp (D)\to \CKh^\sharp (D')
\]
with 
\[
(\ord_h(\phi^{KM}_{Rn^{\varepsilon}}),\ord_q(\phi^{KM}_{Rn^{\varepsilon}}))\geq (0,0)
\]
and homotopy  inverse $\phi^{KM}_{Rn^{-\varepsilon}}$ such that
the chain homotopies $\phi^{KM}_{Rn^{-\varepsilon}} \circ \phi^{KM}_{Rn^{\varepsilon}} \simeq \id$ and $\phi^{KM}_{Rn^{\varepsilon}} \circ \phi^{KM}_{Rn^{-\varepsilon}} \simeq \id$ have order $(\ord_h, \ord_q) \geq (-1,0)$, and
the following diagram is commutative:
    \[
    \xymatrix{
        I^\sharp(L)
        \ar[r]^-{I^\sharp(Rn^{\varepsilon})}
        \ar[d]_-{\Phi_*}
        & I^\sharp(L')
        \ar[d]^-{\Phi_*}\\
        H_*(\CKh^\sharp(D))
        \ar[r]^-{(\phi^{KM}_{Rn^{\varepsilon}})_*}
        & H_*(\CKh^\sharp(D'))
    }
    \]
where $L$ (resp.\ $L'$) denotes the link represented by $D$ (resp.\ $D'$), $\Phi_*$ denotes the isomorphism \eqref{eq:Phi} and $I^\sharp(Rn^{\varepsilon})$ is the instanton cobordism map induced from the trace of the corresponding smooth isotopy. 
\end{prop}

\begin{prop}[\text{\cite[Proposition 1.5]{KM14}}]
\label{prop:KM_map_handles}
For each planar handle attachment $S=h^n \colon D \to D'$ ($n=0,1,2$), there exists a chain map
\[
\phi^{KM}_{h^n} : \CKh^\sharp (D)\to \CKh^\sharp (D')
\]
with 
\[
(\ord_h(\phi^{KM}_{h^n}),\ord_q(\phi^{KM}_{h^n}))\geq \left(\frac{1}{2}(S \cdot S), \chi(S)+\frac{3}{2}(S \cdot S)\right)
\]
such that the following diagram is commutative:
    \[
    \xymatrix{
        I^\sharp(D)
        \ar[r]^-{I^\sharp(h^n)}
        \ar[d]_-{\Phi_*}
        & I^\sharp(D')
        \ar[d]^-{\Phi_*}\\
        H_*(\CKh^\sharp(D))
        \ar[r]^-{(\phi^{KM}_{h^n})_*}
        & H_*(\CKh^\sharp(D'))
    }
    \]
where $\Phi_*$ denotes the isomorphism \eqref{eq:Phi} and $I^\sharp(h^n)$ is the instanton cobordism map for the surface cobodism obtained by attaching an $n$-handle.
\end{prop}

\subsection{Definition of $\phi^\sharp_S$}
Now, we shall give 
the definition 
of $\phi^\sharp_S$ appeared in \cref{thm:main}.
\begin{defn}
Suppose that a cobordism $S\colon K \to K'$ is represented by a movie (i.e. a composition of elementary cobordisms)
\[
D=D_0 \overset{S_1}{\longrightarrow} D_1 
\overset{S_2}{\longrightarrow} \cdots 
\overset{S_m}{\longrightarrow} D_m = D'.
\]
Then we associate to $S$ a chain map $\phi^\sharp_S$ in the following way:
\begin{itemize}
    \item[(1)] If $S_i$ is 
    a planar handle attachment, 
    then we just define 
    \[\phi^\sharp_{S_i}:=\phi^{KM}_{S_i}.
    \]
    \item[(2)] If $S_i$ is either of Reidemeister moves except for $R3$, then we first decompose $S_i$ into a movie $S_i=S'_{m'}\circ \cdots \circ S'_{1}$ as shown in \Cref{fig:movies}, and define 
    \[
\phi^\sharp_{S_i}:=\phi^{KM}_{S'_{m'}}\circ \cdots \circ \phi^{KM}_{S'_{1}}.
\]
    \item[(3)] If $S_i$ is $R3$, then we decompose $S_i$ into a movie shown in  \Cref{fig:RIIImovies}, and define 
    \[
    \phi^\sharp_{S_i} := (\phi^\sharp_{R2^{-1}})^3 \circ (\phi^{KM}_{h^1})^3 \circ \phi^{KM}_{R3} \circ (\phi^{KM}_{R2})^3 \circ (\phi^{KM}_{h^0})^3.
    \]
    
    \item[(4)] Finally, we set $\phi^\sharp_S := \phi^\sharp_{S_m} \circ \cdots \circ \phi^\sharp_{S_1}$.
\end{itemize}
(Here we note that $R3^{-1}=R3$.)
\end{defn}

\begin{figure}[t]
    \centering
    \includegraphics[width=0.5\linewidth]{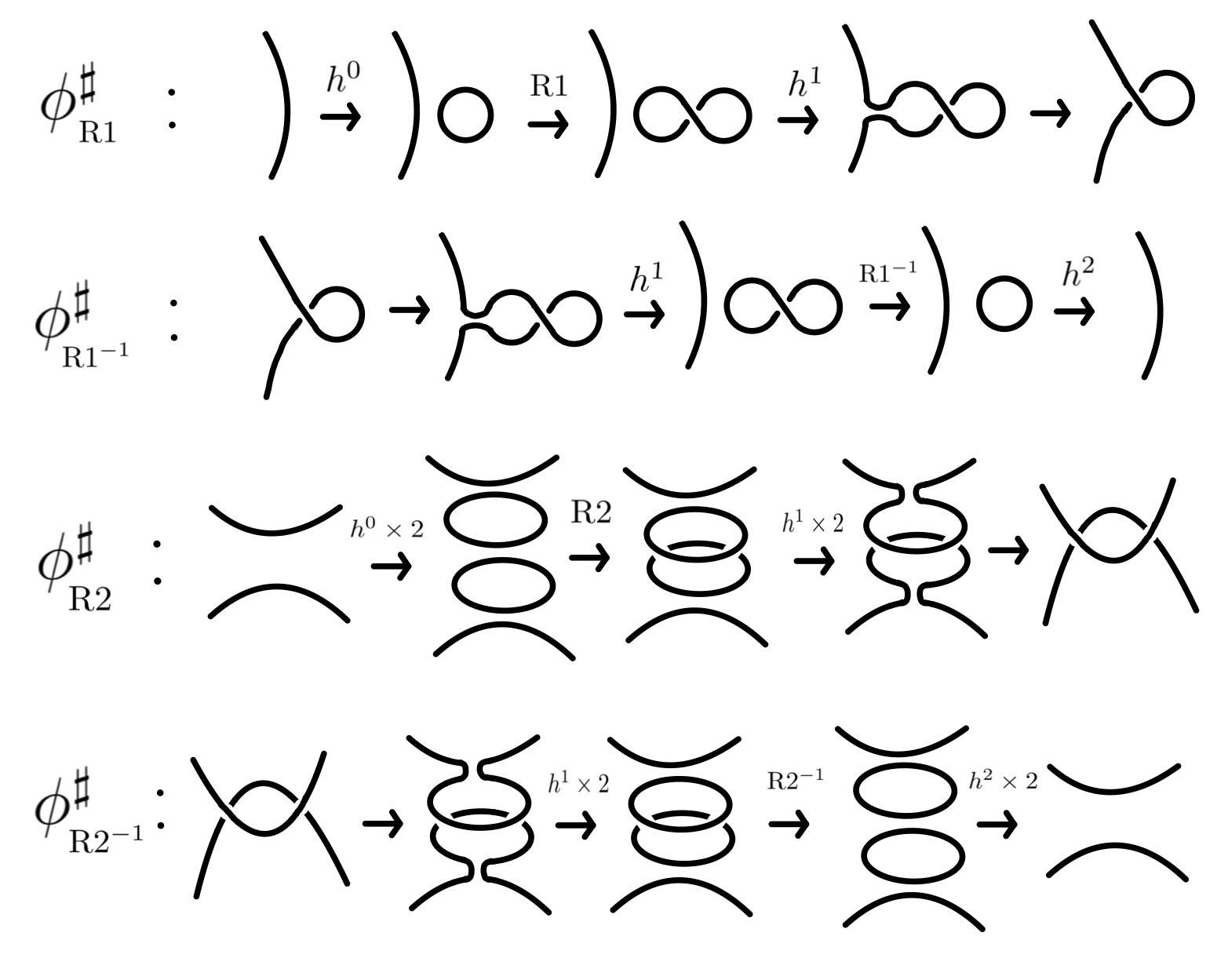}
    \caption{The definitions of $\phi^\sharp_{R1}$, $\phi^\sharp_{R1^{-1}}$, $\phi^\sharp_{R2}$ and $\phi^\sharp_{R2^{-1}}$}
    \label{fig:movies}
\end{figure}

\begin{figure}[t]
    \centering
    \includegraphics[width=0.6\linewidth]{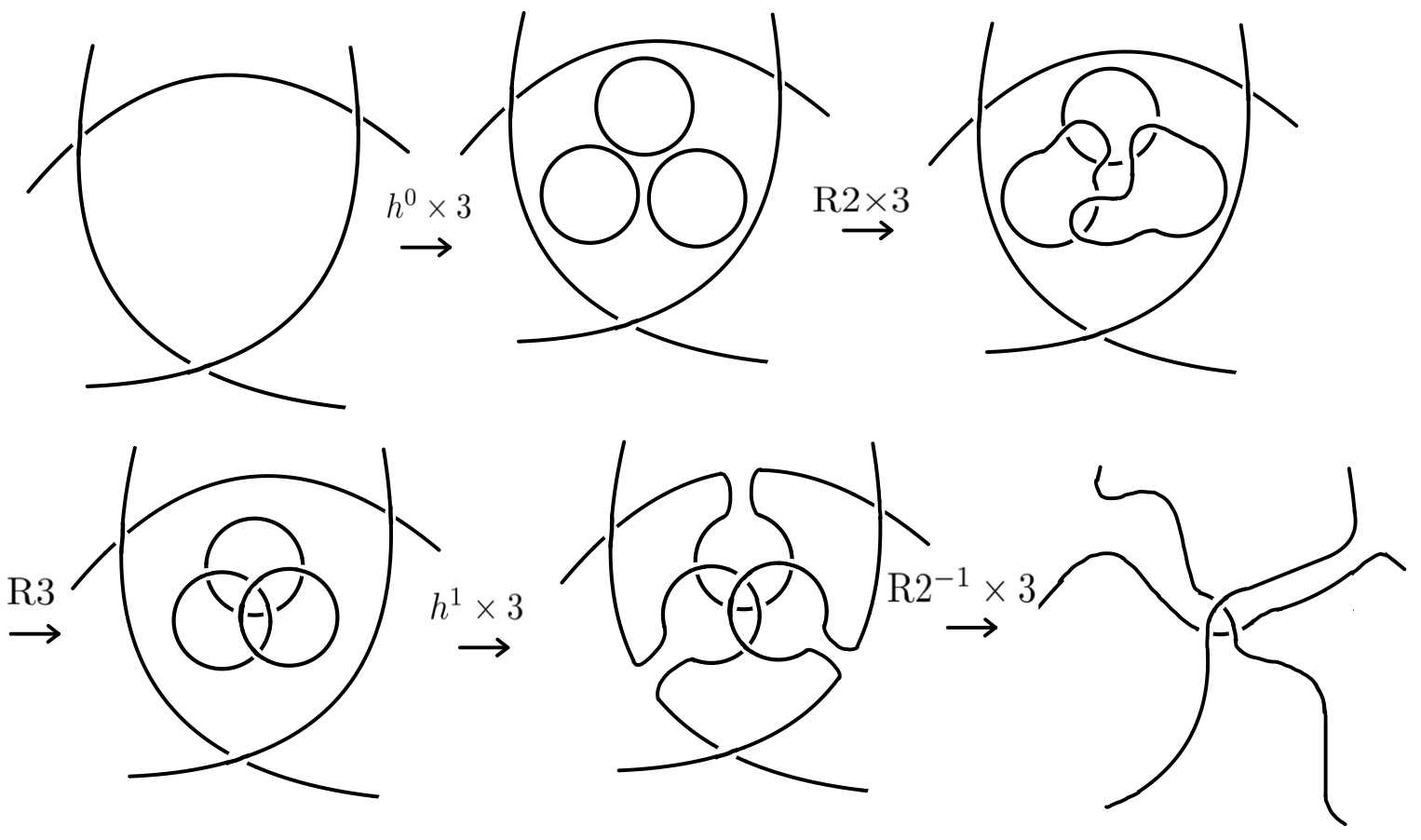}
    \caption{The definition of $\phi^\sharp_{R3}$}
    \label{fig:RIIImovies}
\end{figure}

It follows from \Cref{prop:KM_map_R_moves}, \Cref{prop:KM_map_handles} and the functoriality of $I^\sharp$-functor that $\phi^\sharp_S$
is a map of order 
\[
(\ord_h(\phi^{\sharp}_S),\ord_q(\phi^{\sharp}_S))\geq \displaystyle \left(\tfrac{1}{2}(S \cdot S), \chi(S)+\tfrac{3}{2}(S \cdot S) \right)
\]
and
can be regarded as a representative of $I^\sharp(S)$. 
Therefore, \Cref{thm:main} is reduced to the following proposition.

\begin{prop}
\label{prop:main1_reduced}
For any elementary cobordism $S \colon D \to D'$, we have the commutative diagrams
    \begin{equation} \label{eq:h-filt}
    \begin{split}
    \xymatrix@C=44pt{
        H_*(E^1(\CKh^\sharp(D)))
        \ar[r]^-{(E^1(\phi^\sharp_S))_*}
        \ar[d]_-{\gamma_*}
        & H_*(E^1(\CKh^\sharp(D')))
        \ar[d]^-{\gamma_*}\\
        \Kh(D^*)
        \ar[r]^-{\Kh(S^*)}
        & \Kh(D'^*)
    }
    \end{split}
    \end{equation}
    and
    \begin{equation} \label{eq:q-filt}
    \begin{split}
    \xymatrix@C=44pt{
        H_*(qE^0(\CKh^\sharp(D)))
        \ar[r]^-{(qE^0(\phi^\sharp_S))_*}
        \ar[d]_-{\gamma_*}
        & H_*(qE^0(\CKh^\sharp(D')))
        \ar[d]^-{\gamma_*}\\
        \Kh(D^*)
        \ar[r]^-{\Kh(S^*)}
        & \Kh(D'^*)
    }
    \end{split}
    \end{equation}
up to sign, where $\gamma_*$ denotes the isomorphism induced from
\Cref{prop:equal_to_Kh}. 
\end{prop}

\subsection{Lemmas from instanton theories}

In order to prove \cref{prop:main1_reduced}, we summarize three key lemmas from instanton theories. 
The proofs of these lemmas will be given in Sections 3 and 4. 
\begin{lem}\label{excision lemma}
Suppose $D_1$ and $D_2$ are diagrams. 
For a disjoint union $D_1 \sqcup D_2$ of two diagrams, there exists a $\Z$-module map
\[
\Psi \colon \CKh^\sharp(D_1)\otimes \CKh^\sharp(D_2) \to \CKh^\sharp(D_1 \sqcup D_2)
\]
with order $(\ord_h, \ord_q) \geq (0,0)$
such that both $E^1(\Psi)$ and $qE^0(\Psi)$ are chain maps over $\Z$.
Moreover,
we have the commutative diagrams
\[
    \begin{split}
    \xymatrix@C=44pt{
        E^1(\CKh^\sharp(D_1))\otimes E^1(\CKh^\sharp(D_2))
        \ar[r]^-{E^1(\Psi)}
        \ar[d]_-{\gamma \otimes \gamma}
        & E^1(\CKh^\sharp(D_1 \sqcup D_2))
        \ar[d]^-{\gamma}\\
        \CKh(D_1^*) \otimes \CKh(D_2^*)
        \ar[r]^-{\cong}
        & \CKh(D_1^* \sqcup D_2^*)
    }
    \end{split}
\]
    and
\[
    \begin{split}
    \xymatrix@C=44pt{
        qE^0(\CKh^\sharp(D_1))\otimes qE^0(\CKh^\sharp(D_2))
        \ar[r]^-{qE^0(\Psi)}
        \ar[d]_-{\gamma \otimes \gamma}
        & qE^0(\CKh^\sharp(D_1 \sqcup D_2))
        \ar[d]^-{\gamma}\\
        \CKh(D_1^*) \otimes \CKh(D_2^*)
        \ar[r]^-{\cong}
        & \CKh(D_1^* \sqcup D_2^*). 
    }
    \end{split}
\]
\end{lem}

\begin{rem}
We call the map $\Psi$ the {\it excision cobordism map}.
Proposition \ref{general rel} implies that the excision map $\Psi$ defines a chain map and a quasi-isomorphism from the tensor product $\CKh^\sharp(D_1)\otimes \CKh^\sharp(D_2)$ to $CKh^{\sharp}(D_1 \sqcup D_2)$ over the coefficient $\mathbb{F}_2$.
We have not defined a differential $\CKh^\sharp(D_1)\otimes \CKh^\sharp(D_2)$ over $\Z$ so that $\Psi$ can be seen as a chain map over $\Z$. Actually, our excision map $\Psi$ is no longer a chain map over the coefficient $\Z$ with respect to a natural tensor product $\CKh^\sharp(D_1)\otimes \CKh^\sharp(D_2)$.
Since we do not need to treat the higher degree parts of $\Psi$ for the proof of the main result, we only established that the induced maps on $E^1$ and $qE^0$ are chain maps over $\Z$. 
\end{rem}

\begin{lem}\label{disjoint lemma}
For the disjoint union 
\[(S \sqcup [0,1]\times D_2) \colon D_1 \sqcup D_2 \to D'_1 \sqcup D_2
\]
of an elementary cobordism $S \colon D_1 \to D'_1$ and the product cobordism $[0,1] \times D_2$,
the diagrams

\[
    \xymatrix@C=40pt{
        E^1(\CKh^\sharp(D_1)) \otimes E^1(\CKh^\sharp(D_2))
        \ar[r]^-{E^1(\phi^{KM}_S) \otimes \id}
        \ar[d]_-{E^1(\Psi)}
        &E^1(\CKh^\sharp(D'_1)) \otimes E^1(\CKh^\sharp(D_2))
        \ar[d]^-{E^1(\Psi)}\\
        E^1(\CKh^\sharp(D_1 \sqcup D_2))
        \ar[r]^-{E^1(\phi^{KM}_{S \sqcup ([0,1]\times D_2)})}
        & E^1(\CKh^\sharp(D'_1 \sqcup D_2))
    }
\]
and
\[
    \xymatrix@C=40pt{
        qE^0(\CKh^\sharp(D_1)) \otimes
        qE^0(\CKh^\sharp(D_2))
        \ar[r]^-{qE^0(\phi^{KM}_S) \otimes \id}
        \ar[d]_-{qE^0(\Psi)}
        &qE^0(\CKh^\sharp(D'_1)) \otimes qE^0(\CKh^\sharp(D_2))
        \ar[d]^-{qE^0(\Psi)}\\
        qE^0(\CKh^\sharp(D_1 \sqcup D_2))
        \ar[r]^-{qE^0(\phi^{KM}_{S \sqcup ([0,1]\times D_2)})}
        & qE^0(\CKh^\sharp(D'_1 \sqcup D_2))
    }
\]
are commutative up to chain homotopy over $\Z$ if $D_1$, $D_1'$ and $D_2$ are diagrams. 
\end{lem}

\begin{lem}\label{compatibility with Khovanov}
If $S$ is 
a planar handle attachment, 
then the diagrams (\ref{eq:h-filt}) and (\ref{eq:q-filt}) are commutative.
\end{lem}

\subsection{Proof of \Cref{prop:main1_reduced}}
Now, we give a proof of \Cref{prop:main1_reduced}. 
\begin{proof}
By \Cref{compatibility with Khovanov}, we only need to prove the commutativity of the diagrams (\ref{eq:h-filt}) and (\ref{eq:q-filt}) for the cases where $S$ is either of Reidemeister moves.
Moreover, Lemmas \ref{excision lemma} and \ref{disjoint lemma}
imply that the commutativity of the bottom face in the following cubic diagram is equivalent to that of the top face.
(The following cubic diagram is related to (\ref{eq:h-filt}), while a similar diagram related to (\ref{eq:q-filt}) is also obtained.)
\[
\begin{tikzcd}[row sep = 40pt]
& \CKh(D^*_1)\otimes \CKh(D^*_2) \arrow[d, "\cong"] \arrow[r, "\Kh(S^*)\otimes \id"]
& \CKh(D'^*_1)\otimes \CKh(D^*_2) \arrow[d, "\cong"] \\
& \CKh(D^*_1 \sqcup D^*_2) 
\arrow[r, "{\Kh(S^* \sqcup [0,1]\times D^*_2)}"]
\arrow[ldd, pos=0.7, "\gamma"', leftarrow] 
& \CKh(D'^*_1 \sqcup D^*_2)\\
E^1(\CKh^\sharp(D_1))\otimes E^1(\CKh^\sharp(D_2)) 
\arrow[d, "E^1(\Psi)"] 
\arrow[r, "E^1(\phi^\sharp)\otimes \id"] 
\arrow[ruu, pos=0.15, "\gamma \otimes \gamma"] &
E^1(\CKh^\sharp(D'_1))\otimes E^1(\CKh^\sharp(D_2)) \arrow[d, "E^1(\Psi)"] 
\arrow[ruu, pos=0.15, "\gamma \otimes \gamma"] & \\
E^1(\CKh^\sharp(D_1 \sqcup D_2)) 
\arrow[r, "{E^1(\phi^\sharp_{S \sqcup [0,1]\times D_2})}"] & 
E^1(\CKh^\sharp(D'_1 \sqcup D_2)) 
\arrow[ruu, pos=0.35, "\gamma"']& 
\end{tikzcd}
\]
Consequently, we only need to prove the commutativity of the diagrams (\ref{eq:h-filt}) and (\ref{eq:q-filt}) 
for the maps 
\begin{gather*}
\phi^{KM}_{R1}\circ \phi^{KM}_{h^0}, \quad \phi^{KM}_{h^2}\circ\phi^{KM}_{R1^{-1}}, \quad
\phi^{KM}_{R2}\circ (\phi^{KM}_{h^0})^2, \quad (\phi^{KM}_{h^2})^2\circ\phi^{KM}_{R2^{-1}}
\quad \text{and} \quad
\phi^{KM}_{R3}\circ (\phi^{KM}_{R2})^3\circ (\phi^{KM}_{h^0})^3,
\end{gather*}
which are corresponding to the movies
shown in \Cref{fig:special movies}.

\begin{figure}[t]
    \centering
    \includegraphics[width=0.6\linewidth]{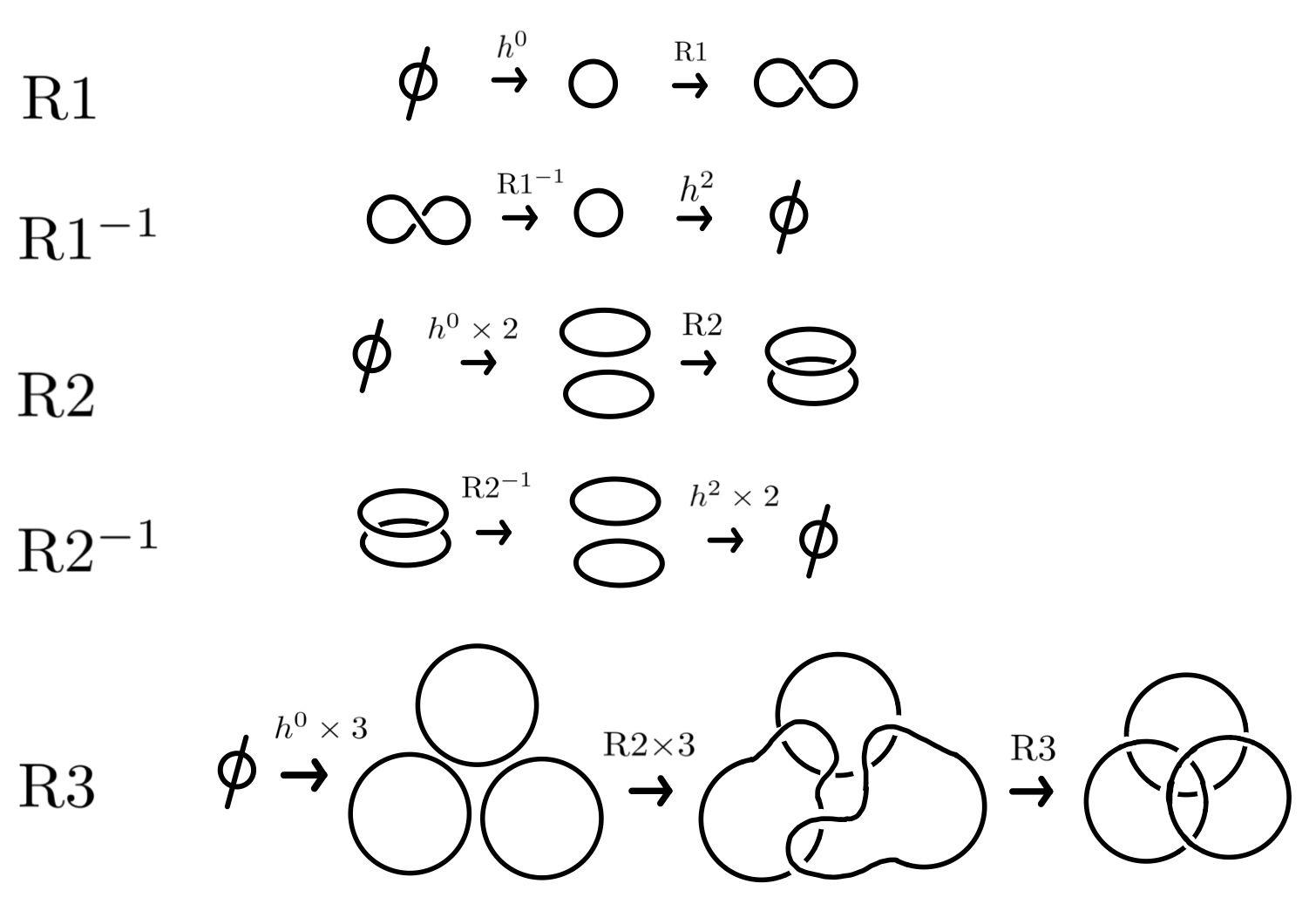}
    \caption{}
    \label{fig:special movies}
\end{figure}
Here we note that all the above maps are either of the two forms
\[
\phi \circ (\phi^{KM}_{h^0})^l \quad \text{and} \quad
(\phi^{KM}_{h^2})^l \circ \phi
\]
where $\phi\colon \CKh^\sharp(D) \to \CKh^\sharp(D')$ is a chain homotopy equivalence map derived from a sequence of Reidemeister moves between two diagrams $D,D'$ for the $l$ component unlink. (We denote the resulting cobordism by 
$S$.) Therefore, the remaining part of the proof can be decomposed into the following four claims.

\begin{claim}
If we equip $(E^1(\CKh^\sharp(D^{(')})),E^1(d^\sharp))$ with
the $q$-filtration induced from the equality 
$\left(E^1(\CKh^\sharp(D^{(')})),E^1(d^\sharp) \right) = \left(\CKh^\sharp(D^{(')}), d^{\sharp 1}\right)$, then the map
\[
\left(E^1(\phi)\right)_* \colon H_*(E^1(\CKh^\sharp(D))) \to H_*(E^1(\CKh^\sharp(D')))
\]
is a $q$-filterd isomorphism. Namely, for each $p \in \Z$, we have
\[
\left(E^1(\phi)\right)_*\left(\mF^q_pH_*(E^1(\CKh^\sharp(D)))\right) =
\mF^q_pH_*(E^1(\CKh^\sharp(D'))),
\]
where 
\[
\mF^q_pH_*(E^1(\CKh^\sharp(D))) := 
\{[x] \in H_*(E^1(\CKh^\sharp(D))) \mid 
x \in E^1(\CKh^\sharp(D)),\  \gr_q(x) \geq p\}.
\]
Moreover, the map $\gamma_*^{-1} \circ \Kh(S) \circ \gamma_*$ is also a q-filtered isomorphism.
\end{claim}

\begin{proof}
Immediately follows from 
Propositions \ref{prop:equal_to_Kh} and \ref{prop:KM_map_R_moves}.
\end{proof}

\begin{claim}
If $f,g \colon H_*(E^1(\CKh^\sharp(D))) \to H_*(E^1(\CKh^\sharp(D'))$ are $q$-filtered isomorpshisms, 
then we have
\[
f \circ \left(E^1(\phi^{KM}_{h^0})^l)\right)_*
= \pm g \circ \left(E^1(\phi^{KM}_{h^0})^l\right)_*
\quad \text{and} \quad
\left(E^1(\phi^{KM}_{h^2})^l\right)_* \circ f
= \pm \left(E^1(\phi^{KM}_{h^2})^l\right)_* \circ g.
\]
\end{claim}

\begin{proof}
The map $f \circ \left(E^1(\phi^{KM}_{h^0})^l)\right)_*$ is decomposed as follows;
\[
\begin{tikzcd} 
H_*(E^1(\CKh^\sharp(\emptyset)) 
\arrow[r, equal]
&[-.5cm] \mF^q_{0}H_*(E^1(\CKh^\sharp(\emptyset))\\
\arrow[r, "(E^1(\phi^{KM}_{h^0})^l))_*" ]
&[1cm]  \mF^q_{l}H_*(E^1(\CKh^\sharp(D))
\arrow[r, "f'"]
&[-.2cm]  \mF^q_{l}H_*(E^1(\CKh^\sharp(D'))
\arrow[r, hookrightarrow]
&[-.2cm] H_*(E^1(\CKh^\sharp(D')),
\end{tikzcd}
\]
where $f':= f|_{\mF^q_{l}H_*(E^1(\CKh^\sharp(D))}$.
Since $f$ is a $q$-filtered isomorphism, 
$f'$ is an isomorphism between infinite cyclic groups, and hence unique up to sign.

Similarly, the map $\left(E^1(\phi^{KM}_{h^2})^l\right)_* \circ f$ is decomposed as follows;
\[
\begin{tikzcd} 
H_*(E^1(\CKh^\sharp(D))
\arrow[r, two heads]
&[-.2cm] 
\frac{H_*(E^1(\CKh^\sharp(D))}{\mF^q_{-l+1}H_*(E^1(\CKh^\sharp(D))}
\arrow[r, "f''"]
&  
\frac{H_*(E^1(\CKh^\sharp(D')}{\mF^q_{-l+1}H_*(E^1(\CKh^\sharp(D'))}\\
&
\phantom{aaaaaaaaa}
\arrow[r, "(E^1(\phi^{KM}_{h^2})^l)_*"]
&  
\frac{H_*(E^1(\CKh^\sharp(\emptyset))}{\mF^q_{1}H_*(E^1(\CKh^\sharp(\emptyset))}
\arrow[r, equal]
&[-.8cm] H_*(E^1(\CKh^\sharp(\emptyset)),
\end{tikzcd}
\]
where $f''$ is induced from $f$ on the quotient groups.
Since $f$ is a 
$q$-filtered isomorphism, 
$f''$ is also an isomorphism between infinite cyclic groups, and hence unique up to sign.
\end{proof}

\begin{claim}
For the decompositions $qE^0(\CKh^\sharp(D^{(')}))=\bigoplus_{p\in \Z}qE^0_{p}(\CKh^\sharp(D^{(')}))$ as chain complexes, the restriction
\[
qE^0_p(\phi) := qE^0(\phi)|_{qE^0_p(\CKh^\sharp(D))}
\colon qE^0_p(\CKh^\sharp(D)) \to qE^0_p(\CKh^\sharp(D'))
\]
is a chain homotopy equivalence map for each $p\in \Z$.
As a consequence, we have the direct decomposition
\begin{equation}   
\label{eq:q-graded}
(qE^0(\phi))_*=\bigoplus_{p \in \Z}
\left(qE^0_p(\phi)\right)_* \colon
\bigoplus_{p\in \Z} H_*(qE^0_p(\CKh^\sharp(D))) \to \bigoplus_{p\in \Z}H_*(qE^0_p(\CKh^\sharp(D')))
\end{equation}
of isomorphisms. 
\end{claim}

\begin{proof}
This immediately follows from elementary arguments in homological algebra.    
\end{proof}

\begin{rem}
We call an isomorphism satisfying (\ref{eq:q-graded})
{\it a $q$-graded isomorphism}. Note that for the cobordism $S$ coresponding to $\phi$, the map $\gamma^{-1} \circ \Kh(S) \circ \gamma$ is also a $q$-graded isomorphism.
\end{rem}

\begin{claim}
If $f,g \colon H_*(qE^0(\CKh^\sharp(D))) \to H_*(qE^0(\CKh^\sharp(D'))$ are $q$-graded isomorpshisms, 
then we have
\[
f \circ \left(qE^0(\phi^{KM}_{h^0})^l)\right)_*
= \pm g \circ \left(qE^0(\phi^{KM}_{h^0})^l\right)_*
\quad \text{and} \quad
\left(qE^0(\phi^{KM}_{h^2})^l\right)_* \circ f
= \pm \left(qE^0(\phi^{KM}_{h^2})^l\right)_* \circ g.
\]
\end{claim}

\begin{proof}
The map $f \circ \left(qE^0(\phi^{KM}_{h^0})^l)\right)_*$ is decomposed as follows;
\[
\begin{tikzcd} 
H_*(qE^0(\CKh^\sharp(\emptyset)) 
\arrow[r, equal]
&[-.5cm] H_*(qE^0_{0}(\CKh^\sharp(\emptyset))\\
\arrow[r, "(qE^0_0(\phi^{KM}_{h^0})^l))_*" ]
&[1cm] H_*(qE^0_l(\CKh^\sharp(D))
\arrow[r, "f_l"]
&[-.2cm]  H_*(qE^0_l(\CKh^\sharp(D'))
\arrow[r, hookrightarrow]
&[-.2cm] H_*(qE^0(\CKh^\sharp(D')),
\end{tikzcd}
\]
where $f_l = f|_{H_*(qE^0_l(\CKh^\sharp(D)))}$.
Since $f$ is a $q$-graded isomorphism, 
$f_l$ is an isomorphism between infinite cyclic groups, and hence unique up to sign.

Similarly, the map $\left(qE^0(\phi^{KM}_{h^2})^l\right)_* \circ f$ is decomposed as follows;
\[
\begin{tikzcd} 
H_*(qE^0(\CKh^\sharp(D))
\arrow[r, two heads]
&[-.2cm] 
H_*(qE^0_{-l}(\CKh^\sharp(D))
\arrow[r, "f_{-l}"]
&  
H_*(qE^0_{-l}(\CKh^\sharp(D'))\\
&
\phantom{aaaaaaaaa}
\arrow[r, "(qE^0(\phi^{KM}_{h^2})^l)_*"]
&  
H_*(qE^0_{0}(\CKh^\sharp(\emptyset))
\arrow[r, equal]
&[-.8cm] H_*(qE^0(\CKh^\sharp(\emptyset)),
\end{tikzcd}
\]
where $f_{-l} = f|_{H_*(qE^0_{-l}(\CKh^\sharp(D)))}$.
Since $f$ is a $q$-graded isomorphism, 
$f_{-l}$ is also an isomorphism between infinite cyclic groups, and hence unique up to sign.
\end{proof}

Now, the proof of \Cref{prop:main1_reduced} is completed.
\end{proof}

Hence, the proof of \cref{thm:main} reduces to proving \cref{excision lemma}, \cref{disjoint lemma}, and \cref{compatibility with Khovanov}. 

%% file: tikzcd/ckh.tex
\tikzset{every picture/.style={line width=0.75pt}} 

\begin{tikzpicture}[x=0.75pt,y=0.75pt,yscale=-1,xscale=1]

\draw [line width=1.5]    (101.91,117.73) .. controls (94.83,58.33) and (35.33,101.83) .. (73.69,153.34) ;
\draw [line width=1.5]    (84.48,162.91) .. controls (125.33,191.33) and (154.33,108.33) .. (66.83,121.83) ;
\draw [line width=1.5]    (56.46,123.28) .. controls (19.03,130.68) and (47.01,217.63) .. (100.39,130.78) ;

\draw [line width=1.5]    (236.58,128.54) .. controls (232.29,92.54) and (196.22,118.9) .. (219.47,150.12) ;
\draw [line width=1.5]    (226.01,155.92) .. controls (250.77,173.15) and (268.35,122.84) .. (215.31,131.03) ;
\draw [line width=1.5]    (209.03,131.9) .. controls (186.34,136.39) and (203.3,189.09) .. (235.65,136.45) ;
\draw [line width=1.5]    (357.58,43.87) .. controls (353.29,7.87) and (317.22,34.24) .. (340.47,65.46) ;
\draw [line width=1.5]    (347.01,71.26) .. controls (371.77,88.49) and (389.35,38.18) .. (336.31,46.36) ;
\draw [line width=1.5]    (330.03,47.24) .. controls (307.34,51.72) and (324.3,104.42) .. (356.65,51.78) ;

\draw [line width=1.5]    (357.58,126.87) .. controls (353.29,90.87) and (317.22,117.24) .. (340.47,148.46) ;
\draw [line width=1.5]    (347.01,154.26) .. controls (371.77,171.49) and (389.35,121.18) .. (336.31,129.36) ;
\draw [line width=1.5]    (330.03,130.24) .. controls (307.34,134.73) and (324.3,187.43) .. (356.65,134.78) ;

\draw [line width=1.5]    (357.58,209.87) .. controls (353.29,173.87) and (317.22,200.24) .. (340.47,231.46) ;
\draw [line width=1.5]    (347.01,237.26) .. controls (371.77,254.49) and (389.35,204.18) .. (336.31,212.36) ;
\draw [line width=1.5]    (330.03,213.24) .. controls (307.34,217.72) and (324.3,270.42) .. (356.65,217.78) ;

\draw [line width=1.5]    (471.58,42.2) .. controls (467.29,6.2) and (431.22,32.57) .. (454.47,63.79) ;
\draw [line width=1.5]    (461.01,69.59) .. controls (485.77,86.82) and (503.35,36.51) .. (450.31,44.69) ;
\draw [line width=1.5]    (444.03,45.57) .. controls (421.34,50.06) and (438.3,102.76) .. (470.65,50.11) ;

\draw [line width=1.5]    (471.58,125.21) .. controls (467.29,89.21) and (431.22,115.57) .. (454.47,146.79) ;
\draw [line width=1.5]    (461.01,152.59) .. controls (485.77,169.82) and (503.35,119.51) .. (450.31,127.7) ;
\draw [line width=1.5]    (444.03,128.57) .. controls (421.34,133.06) and (438.3,185.76) .. (470.65,133.12) ;

\draw [line width=1.5]    (471.58,208.2) .. controls (467.29,172.2) and (431.22,198.57) .. (454.47,229.79) ;
\draw [line width=1.5]    (461.01,235.59) .. controls (485.77,252.82) and (503.35,202.51) .. (450.31,210.69) ;
\draw [line width=1.5]    (444.03,211.57) .. controls (421.34,216.06) and (438.3,268.76) .. (470.65,216.11) ;

\draw [line width=1.5]    (588.91,127.21) .. controls (584.62,91.21) and (548.55,117.57) .. (571.81,148.79) ;
\draw [line width=1.5]    (578.35,154.59) .. controls (603.11,171.82) and (620.69,121.51) .. (567.65,129.7) ;
\draw [line width=1.5]    (561.36,130.57) .. controls (538.67,135.06) and (555.63,187.76) .. (587.99,135.12) ;

\draw  [color={rgb, 255:red, 255; green, 255; blue, 255 }  ,draw opacity=1 ][fill={rgb, 255:red, 255; green, 255; blue, 255 }  ,fill opacity=1 ] (217.22,128.69) .. controls (218.82,131.85) and (217.6,135.68) .. (214.5,137.25) .. controls (211.4,138.83) and (207.59,137.56) .. (205.98,134.41) .. controls (204.38,131.25) and (205.6,127.42) .. (208.7,125.85) .. controls (211.8,124.27) and (215.61,125.54) .. (217.22,128.69) -- cycle ;
\draw [line width=1.5]    (212,124.54) .. controls (212,129.54) and (215.5,130.54) .. (218.5,130.54) ;
\draw [line width=1.5]    (205,133.29) .. controls (208,132.54) and (212.25,133.29) .. (213.25,137.79) ;

\draw  [color={rgb, 255:red, 255; green, 255; blue, 255 }  ,draw opacity=1 ][fill={rgb, 255:red, 255; green, 255; blue, 255 }  ,fill opacity=1 ] (338.22,127.03) .. controls (339.82,130.18) and (338.6,134.01) .. (335.5,135.59) .. controls (332.4,137.17) and (328.59,135.89) .. (326.98,132.74) .. controls (325.38,129.59) and (326.6,125.76) .. (329.7,124.18) .. controls (332.8,122.6) and (336.61,123.88) .. (338.22,127.03) -- cycle ;
\draw [line width=1.5]    (333,122.88) .. controls (333,127.88) and (336.5,128.88) .. (339.5,128.88) ;
\draw [line width=1.5]    (326,131.63) .. controls (329,130.88) and (333.25,131.63) .. (334.25,136.13) ;

\draw  [color={rgb, 255:red, 255; green, 255; blue, 255 }  ,draw opacity=1 ][fill={rgb, 255:red, 255; green, 255; blue, 255 }  ,fill opacity=1 ] (338.47,210.03) .. controls (340.07,213.18) and (338.85,217.01) .. (335.75,218.59) .. controls (332.65,220.17) and (328.84,218.89) .. (327.23,215.74) .. controls (325.63,212.59) and (326.85,208.76) .. (329.95,207.18) .. controls (333.05,205.6) and (336.86,206.88) .. (338.47,210.03) -- cycle ;
\draw [line width=1.5]    (333.25,205.88) .. controls (333.25,210.88) and (336.75,211.88) .. (339.75,211.88) ;
\draw [line width=1.5]    (326.25,214.63) .. controls (329.25,213.88) and (333.5,214.63) .. (334.5,219.13) ;

\draw  [color={rgb, 255:red, 255; green, 255; blue, 255 }  ,draw opacity=1 ][fill={rgb, 255:red, 255; green, 255; blue, 255 }  ,fill opacity=1 ] (452.47,208.36) .. controls (454.07,211.51) and (452.85,215.34) .. (449.75,216.92) .. controls (446.65,218.5) and (442.84,217.22) .. (441.23,214.07) .. controls (439.63,210.92) and (440.85,207.09) .. (443.95,205.51) .. controls (447.05,203.93) and (450.86,205.21) .. (452.47,208.36) -- cycle ;
\draw [line width=1.5]    (447.25,204.21) .. controls (447.25,209.21) and (450.75,210.21) .. (453.75,210.21) ;
\draw [line width=1.5]    (440.25,212.96) .. controls (443.25,212.21) and (447.5,212.96) .. (448.5,217.46) ;

\draw  [color={rgb, 255:red, 255; green, 255; blue, 255 }  ,draw opacity=1 ][fill={rgb, 255:red, 255; green, 255; blue, 255 }  ,fill opacity=1 ] (332.98,55.34) .. controls (329.46,55.6) and (326.39,53) .. (326.14,49.53) .. controls (325.88,46.06) and (328.53,43.04) .. (332.05,42.78) .. controls (335.58,42.52) and (338.64,45.12) .. (338.9,48.59) .. controls (339.16,52.06) and (336.51,55.08) .. (332.98,55.34) -- cycle ;
\draw [line width=1.5]    (339,45.88) .. controls (335,45.88) and (333,50.88) .. (335,55.38) ;
\draw [line width=1.5]    (333,41.88) .. controls (332.54,44.45) and (328.25,47.38) .. (325.75,48.88) ;

\draw  [color={rgb, 255:red, 255; green, 255; blue, 255 }  ,draw opacity=1 ][fill={rgb, 255:red, 255; green, 255; blue, 255 }  ,fill opacity=1 ] (446.48,53.93) .. controls (442.96,54.19) and (439.89,51.59) .. (439.64,48.12) .. controls (439.38,44.65) and (442.03,41.62) .. (445.55,41.36) .. controls (449.08,41.1) and (452.14,43.7) .. (452.4,47.17) .. controls (452.66,50.64) and (450.01,53.66) .. (446.48,53.93) -- cycle ;
\draw [line width=1.5]    (452.5,44.46) .. controls (448.5,44.46) and (446.5,49.46) .. (448.5,53.96) ;
\draw [line width=1.5]    (446.5,40.46) .. controls (446.04,43.03) and (441.75,45.96) .. (439.25,47.46) ;

\draw  [color={rgb, 255:red, 255; green, 255; blue, 255 }  ,draw opacity=1 ][fill={rgb, 255:red, 255; green, 255; blue, 255 }  ,fill opacity=1 ] (446.98,136.68) .. controls (443.46,136.94) and (440.39,134.34) .. (440.14,130.87) .. controls (439.88,127.4) and (442.53,124.37) .. (446.05,124.11) .. controls (449.58,123.85) and (452.64,126.45) .. (452.9,129.92) .. controls (453.16,133.39) and (450.51,136.41) .. (446.98,136.68) -- cycle ;
\draw [line width=1.5]    (453,127.21) .. controls (449,127.21) and (447,132.21) .. (449,136.71) ;
\draw [line width=1.5]    (447,123.21) .. controls (446.54,125.78) and (442.25,128.71) .. (439.75,130.21) ;

\draw  [color={rgb, 255:red, 255; green, 255; blue, 255 }  ,draw opacity=1 ][fill={rgb, 255:red, 255; green, 255; blue, 255 }  ,fill opacity=1 ] (563.82,138.93) .. controls (560.29,139.19) and (557.23,136.59) .. (556.97,133.12) .. controls (556.71,129.65) and (559.36,126.62) .. (562.89,126.36) .. controls (566.41,126.1) and (569.48,128.7) .. (569.73,132.17) .. controls (569.99,135.64) and (567.34,138.66) .. (563.82,138.93) -- cycle ;
\draw [line width=1.5]    (569.83,129.46) .. controls (565.83,129.46) and (563.83,134.46) .. (565.83,138.96) ;
\draw [line width=1.5]    (563.83,125.46) .. controls (563.37,128.03) and (559.08,130.96) .. (556.58,132.46) ;

\draw  [color={rgb, 255:red, 255; green, 255; blue, 255 }  ,draw opacity=1 ][fill={rgb, 255:red, 255; green, 255; blue, 255 }  ,fill opacity=1 ] (355.06,41.37) .. controls (358.53,40.69) and (361.88,42.91) .. (362.54,46.33) .. controls (363.21,49.74) and (360.93,53.06) .. (357.46,53.73) .. controls (354,54.41) and (350.64,52.19) .. (349.98,48.77) .. controls (349.31,45.36) and (351.59,42.04) .. (355.06,41.37) -- cycle ;
\draw [line width=1.5]    (349.75,45.38) .. controls (354.87,45.51) and (357.75,43.13) .. (357.18,40.64) ;
\draw [line width=1.5]    (355,54.38) .. controls (356.25,51.88) and (358.5,47.63) .. (362.75,48.63) ;

\draw  [color={rgb, 255:red, 255; green, 255; blue, 255 }  ,draw opacity=1 ][fill={rgb, 255:red, 255; green, 255; blue, 255 }  ,fill opacity=1 ] (233.81,126.03) .. controls (237.28,125.36) and (240.63,127.58) .. (241.29,130.99) .. controls (241.96,134.41) and (239.68,137.72) .. (236.21,138.4) .. controls (232.75,139.08) and (229.39,136.86) .. (228.73,133.44) .. controls (228.06,130.02) and (230.34,126.71) .. (233.81,126.03) -- cycle ;
\draw [line width=1.5]    (228.5,130.04) .. controls (233.62,130.17) and (236.5,127.79) .. (235.93,125.3) ;
\draw [line width=1.5]    (233.75,139.04) .. controls (235,136.54) and (237.25,132.29) .. (241.5,133.29) ;

\draw  [color={rgb, 255:red, 255; green, 255; blue, 255 }  ,draw opacity=1 ][fill={rgb, 255:red, 255; green, 255; blue, 255 }  ,fill opacity=1 ] (355.06,207.87) .. controls (358.53,207.19) and (361.88,209.41) .. (362.54,212.83) .. controls (363.21,216.24) and (360.93,219.56) .. (357.46,220.23) .. controls (354,220.91) and (350.64,218.69) .. (349.98,215.27) .. controls (349.31,211.86) and (351.59,208.54) .. (355.06,207.87) -- cycle ;
\draw [line width=1.5]    (349.75,211.88) .. controls (354.87,212.01) and (357.75,209.63) .. (357.18,207.14) ;
\draw [line width=1.5]    (355,220.88) .. controls (356.25,218.38) and (358.5,214.13) .. (362.75,215.13) ;

\draw  [color={rgb, 255:red, 255; green, 255; blue, 255 }  ,draw opacity=1 ][fill={rgb, 255:red, 255; green, 255; blue, 255 }  ,fill opacity=1 ] (469.31,122.7) .. controls (472.78,122.02) and (476.13,124.24) .. (476.79,127.66) .. controls (477.46,131.08) and (475.18,134.39) .. (471.71,135.07) .. controls (468.25,135.74) and (464.89,133.52) .. (464.23,130.11) .. controls (463.56,126.69) and (465.84,123.38) .. (469.31,122.7) -- cycle ;
\draw [line width=1.5]    (464,126.71) .. controls (469.12,126.84) and (472,124.46) .. (471.43,121.97) ;
\draw [line width=1.5]    (469.25,135.71) .. controls (470.5,133.21) and (472.75,128.96) .. (477,129.96) ;

\draw  [color={rgb, 255:red, 255; green, 255; blue, 255 }  ,draw opacity=1 ][fill={rgb, 255:red, 255; green, 255; blue, 255 }  ,fill opacity=1 ] (356.74,138.09) .. controls (359.67,138.35) and (362.22,135.75) .. (362.43,132.28) .. controls (362.64,128.81) and (360.44,125.79) .. (357.52,125.53) .. controls (354.59,125.27) and (352.05,127.87) .. (351.83,131.34) .. controls (351.62,134.81) and (353.82,137.83) .. (356.74,138.09) -- cycle ;
\draw [line width=1.5]    (351.75,128.63) .. controls (355.07,128.63) and (357.5,132.63) .. (354.5,137.63) ;
\draw [line width=1.5]    (357.5,125.13) .. controls (357.73,129.47) and (361.17,130.13) .. (363.25,131.63) ;

\draw  [color={rgb, 255:red, 255; green, 255; blue, 255 }  ,draw opacity=1 ][fill={rgb, 255:red, 255; green, 255; blue, 255 }  ,fill opacity=1 ] (470.49,53.43) .. controls (473.42,53.69) and (475.97,51.09) .. (476.18,47.62) .. controls (476.39,44.15) and (474.19,41.12) .. (471.27,40.86) .. controls (468.34,40.6) and (465.8,43.2) .. (465.58,46.67) .. controls (465.37,50.14) and (467.57,53.16) .. (470.49,53.43) -- cycle ;
\draw [line width=1.5]    (465.5,43.96) .. controls (468.82,43.96) and (471.25,47.96) .. (468.25,52.96) ;
\draw [line width=1.5]    (471.25,40.46) .. controls (471.48,44.81) and (474.92,45.46) .. (477,46.96) ;

\draw  [color={rgb, 255:red, 255; green, 255; blue, 255 }  ,draw opacity=1 ][fill={rgb, 255:red, 255; green, 255; blue, 255 }  ,fill opacity=1 ] (470.24,219.93) .. controls (473.17,220.19) and (475.72,217.59) .. (475.93,214.12) .. controls (476.14,210.65) and (473.94,207.62) .. (471.02,207.36) .. controls (468.09,207.1) and (465.55,209.7) .. (465.33,213.17) .. controls (465.12,216.64) and (467.32,219.66) .. (470.24,219.93) -- cycle ;
\draw [line width=1.5]    (465.25,210.46) .. controls (468.57,210.46) and (471,214.46) .. (468,219.46) ;
\draw [line width=1.5]    (471,206.96) .. controls (471.23,211.31) and (474.67,211.96) .. (476.75,213.46) ;

\draw  [color={rgb, 255:red, 255; green, 255; blue, 255 }  ,draw opacity=1 ][fill={rgb, 255:red, 255; green, 255; blue, 255 }  ,fill opacity=1 ] (588.08,138.43) .. controls (591,138.69) and (593.55,136.09) .. (593.76,132.62) .. controls (593.98,129.15) and (591.78,126.12) .. (588.85,125.86) .. controls (585.92,125.6) and (583.38,128.2) .. (583.17,131.67) .. controls (582.95,135.14) and (585.15,138.16) .. (588.08,138.43) -- cycle ;
\draw [line width=1.5]    (583.08,128.96) .. controls (586.4,128.96) and (588.83,132.96) .. (585.83,137.96) ;
\draw [line width=1.5]    (588.83,125.46) .. controls (589.07,129.81) and (592.51,130.46) .. (594.58,131.96) ;

\draw  [color={rgb, 255:red, 255; green, 255; blue, 255 }  ,draw opacity=1 ][fill={rgb, 255:red, 255; green, 255; blue, 255 }  ,fill opacity=1 ] (228.02,155.07) .. controls (226.17,158.08) and (222.26,159.05) .. (219.3,157.22) .. controls (216.34,155.4) and (215.44,151.48) .. (217.29,148.47) .. controls (219.14,145.46) and (223.04,144.5) .. (226.01,146.32) .. controls (228.97,148.14) and (229.87,152.06) .. (228.02,155.07) -- cycle ;
\draw [line width=1.5]    (228,147.54) .. controls (225,150.54) and (225.5,155.54) .. (227.25,156.79) ;
\draw [line width=1.5]    (217.72,147.05) .. controls (219.93,149.22) and (221.51,153.24) .. (218.19,156.44) ;

\draw  [color={rgb, 255:red, 255; green, 255; blue, 255 }  ,draw opacity=1 ][fill={rgb, 255:red, 255; green, 255; blue, 255 }  ,fill opacity=1 ] (349.02,70.41) .. controls (347.17,73.42) and (343.26,74.38) .. (340.3,72.56) .. controls (337.34,70.73) and (336.44,66.81) .. (338.29,63.8) .. controls (340.14,60.79) and (344.04,59.83) .. (347.01,61.65) .. controls (349.97,63.48) and (350.87,67.4) .. (349.02,70.41) -- cycle ;
\draw [line width=1.5]    (349,62.88) .. controls (346,65.88) and (346.5,70.88) .. (348.25,72.13) ;
\draw [line width=1.5]    (338.72,62.38) .. controls (340.93,64.55) and (342.51,68.57) .. (339.19,71.77) ;

\draw  [color={rgb, 255:red, 255; green, 255; blue, 255 }  ,draw opacity=1 ][fill={rgb, 255:red, 255; green, 255; blue, 255 }  ,fill opacity=1 ] (348.77,153.41) .. controls (346.92,156.42) and (343.01,157.38) .. (340.05,155.56) .. controls (337.09,153.73) and (336.19,149.81) .. (338.04,146.8) .. controls (339.89,143.79) and (343.79,142.83) .. (346.76,144.65) .. controls (349.72,146.48) and (350.62,150.4) .. (348.77,153.41) -- cycle ;
\draw [line width=1.5]    (348.75,145.88) .. controls (345.75,148.88) and (346.25,153.88) .. (348,155.13) ;
\draw [line width=1.5]    (338.47,145.38) .. controls (340.68,147.55) and (342.26,151.57) .. (338.94,154.77) ;

\draw  [color={rgb, 255:red, 255; green, 255; blue, 255 }  ,draw opacity=1 ][fill={rgb, 255:red, 255; green, 255; blue, 255 }  ,fill opacity=1 ] (462.77,68.74) .. controls (460.92,71.75) and (457.01,72.71) .. (454.05,70.89) .. controls (451.09,69.07) and (450.19,65.15) .. (452.04,62.14) .. controls (453.89,59.13) and (457.79,58.16) .. (460.76,59.99) .. controls (463.72,61.81) and (464.62,65.73) .. (462.77,68.74) -- cycle ;
\draw [line width=1.5]    (462.75,61.21) .. controls (459.75,64.21) and (460.25,69.21) .. (462,70.46) ;
\draw [line width=1.5]    (452.47,60.72) .. controls (454.68,62.88) and (456.26,66.9) .. (452.94,70.1) ;

\draw  [color={rgb, 255:red, 255; green, 255; blue, 255 }  ,draw opacity=1 ][fill={rgb, 255:red, 255; green, 255; blue, 255 }  ,fill opacity=1 ] (340.98,238.68) .. controls (337.85,237.04) and (336.62,233.21) .. (338.24,230.13) .. controls (339.85,227.05) and (343.7,225.88) .. (346.83,227.53) .. controls (349.96,229.17) and (351.19,233) .. (349.57,236.08) .. controls (347.96,239.16) and (344.11,240.33) .. (340.98,238.68) -- cycle ;
\draw [line width=1.5]    (348.49,238.15) .. controls (345.29,235.36) and (340.34,236.2) .. (339.21,238.03) ;
\draw [line width=1.5]    (348.28,227.86) .. controls (346.27,230.21) and (342.37,232.06) .. (338.95,228.97) ;

\draw  [color={rgb, 255:red, 255; green, 255; blue, 255 }  ,draw opacity=1 ][fill={rgb, 255:red, 255; green, 255; blue, 255 }  ,fill opacity=1 ] (455.73,237.27) .. controls (452.6,235.63) and (451.37,231.8) .. (452.99,228.72) .. controls (454.6,225.63) and (458.45,224.47) .. (461.58,226.11) .. controls (464.71,227.75) and (465.94,231.58) .. (464.32,234.66) .. controls (462.71,237.74) and (458.86,238.91) .. (455.73,237.27) -- cycle ;
\draw [line width=1.5]    (463.24,236.73) .. controls (460.04,233.94) and (455.09,234.79) .. (453.96,236.62) ;
\draw [line width=1.5]    (463.03,226.45) .. controls (461.02,228.8) and (457.12,230.64) .. (453.7,227.55) ;

\draw  [color={rgb, 255:red, 255; green, 255; blue, 255 }  ,draw opacity=1 ][fill={rgb, 255:red, 255; green, 255; blue, 255 }  ,fill opacity=1 ] (455.52,154.27) .. controls (452.35,152.63) and (451.11,148.81) .. (452.75,145.74) .. controls (454.39,142.66) and (458.29,141.5) .. (461.46,143.13) .. controls (464.63,144.76) and (465.87,148.58) .. (464.23,151.66) .. controls (462.59,154.73) and (458.69,155.9) .. (455.52,154.27) -- cycle ;
\draw [line width=1.5]    (463.13,153.73) .. controls (459.89,150.95) and (454.87,151.79) .. (453.73,153.62) ;
\draw [line width=1.5]    (462.93,143.47) .. controls (460.88,145.81) and (456.93,147.66) .. (453.47,144.58) ;

\draw  [color={rgb, 255:red, 255; green, 255; blue, 255 }  ,draw opacity=1 ][fill={rgb, 255:red, 255; green, 255; blue, 255 }  ,fill opacity=1 ] (572.56,156.27) .. controls (569.43,154.63) and (568.2,150.8) .. (569.82,147.72) .. controls (571.44,144.63) and (575.28,143.47) .. (578.41,145.11) .. controls (581.54,146.75) and (582.77,150.58) .. (581.16,153.66) .. controls (579.54,156.74) and (575.69,157.91) .. (572.56,156.27) -- cycle ;
\draw [line width=1.5]    (580.07,155.73) .. controls (576.88,152.94) and (571.92,153.79) .. (570.79,155.62) ;
\draw [line width=1.5]    (579.86,145.45) .. controls (577.85,147.8) and (573.95,149.64) .. (570.53,146.55) ;

\draw [color={rgb, 255:red, 74; green, 74; blue, 74 }  ,draw opacity=1 ]   (258.67,132.33) -- (297.4,59.43) ;
\draw [shift={(298.33,57.67)}, rotate = 117.98] [color={rgb, 255:red, 74; green, 74; blue, 74 }  ,draw opacity=1 ][line width=0.75]    (10.93,-4.9) .. controls (6.95,-2.3) and (3.31,-0.67) .. (0,0) .. controls (3.31,0.67) and (6.95,2.3) .. (10.93,4.9)   ;
\draw [color={rgb, 255:red, 74; green, 74; blue, 74 }  ,draw opacity=1 ]   (384,137) -- (422.73,64.1) ;
\draw [shift={(423.67,62.33)}, rotate = 117.98] [color={rgb, 255:red, 74; green, 74; blue, 74 }  ,draw opacity=1 ][line width=0.75]    (10.93,-4.9) .. controls (6.95,-2.3) and (3.31,-0.67) .. (0,0) .. controls (3.31,0.67) and (6.95,2.3) .. (10.93,4.9)   ;
\draw [color={rgb, 255:red, 74; green, 74; blue, 74 }  ,draw opacity=1 ]   (258.67,159.67) -- (297.4,232.57) ;
\draw [shift={(298.33,234.33)}, rotate = 242.02] [color={rgb, 255:red, 74; green, 74; blue, 74 }  ,draw opacity=1 ][line width=0.75]    (10.93,-4.9) .. controls (6.95,-2.3) and (3.31,-0.67) .. (0,0) .. controls (3.31,0.67) and (6.95,2.3) .. (10.93,4.9)   ;
\draw [color={rgb, 255:red, 74; green, 74; blue, 74 }  ,draw opacity=1 ]   (384,151) -- (422.73,223.9) ;
\draw [shift={(423.67,225.67)}, rotate = 242.02] [color={rgb, 255:red, 74; green, 74; blue, 74 }  ,draw opacity=1 ][line width=0.75]    (10.93,-4.9) .. controls (6.95,-2.3) and (3.31,-0.67) .. (0,0) .. controls (3.31,0.67) and (6.95,2.3) .. (10.93,4.9)   ;
\draw [color={rgb, 255:red, 74; green, 74; blue, 74 }  ,draw opacity=1 ]   (261,147) -- (307.67,147) ;
\draw [shift={(309.67,147)}, rotate = 180] [color={rgb, 255:red, 74; green, 74; blue, 74 }  ,draw opacity=1 ][line width=0.75]    (10.93,-4.9) .. controls (6.95,-2.3) and (3.31,-0.67) .. (0,0) .. controls (3.31,0.67) and (6.95,2.3) .. (10.93,4.9)   ;
\draw [color={rgb, 255:red, 74; green, 74; blue, 74 }  ,draw opacity=1 ]   (379,233) -- (425.67,233) ;
\draw [shift={(427.67,233)}, rotate = 180] [color={rgb, 255:red, 74; green, 74; blue, 74 }  ,draw opacity=1 ][line width=0.75]    (10.93,-4.9) .. controls (6.95,-2.3) and (3.31,-0.67) .. (0,0) .. controls (3.31,0.67) and (6.95,2.3) .. (10.93,4.9)   ;
\draw [color={rgb, 255:red, 74; green, 74; blue, 74 }  ,draw opacity=1 ]   (377.67,50.33) -- (424.33,50.33) ;
\draw [shift={(426.33,50.33)}, rotate = 180] [color={rgb, 255:red, 74; green, 74; blue, 74 }  ,draw opacity=1 ][line width=0.75]    (10.93,-4.9) .. controls (6.95,-2.3) and (3.31,-0.67) .. (0,0) .. controls (3.31,0.67) and (6.95,2.3) .. (10.93,4.9)   ;
\draw [color={rgb, 255:red, 74; green, 74; blue, 74 }  ,draw opacity=1 ]   (494,229) -- (532.73,156.1) ;
\draw [shift={(533.67,154.33)}, rotate = 117.98] [color={rgb, 255:red, 74; green, 74; blue, 74 }  ,draw opacity=1 ][line width=0.75]    (10.93,-4.9) .. controls (6.95,-2.3) and (3.31,-0.67) .. (0,0) .. controls (3.31,0.67) and (6.95,2.3) .. (10.93,4.9)   ;
\draw [color={rgb, 255:red, 74; green, 74; blue, 74 }  ,draw opacity=1 ]   (497.33,53) -- (536.06,125.9) ;
\draw [shift={(537,127.67)}, rotate = 242.02] [color={rgb, 255:red, 74; green, 74; blue, 74 }  ,draw opacity=1 ][line width=0.75]    (10.93,-4.9) .. controls (6.95,-2.3) and (3.31,-0.67) .. (0,0) .. controls (3.31,0.67) and (6.95,2.3) .. (10.93,4.9)   ;
\draw [color={rgb, 255:red, 74; green, 74; blue, 74 }  ,draw opacity=1 ]   (493.67,141.67) -- (540.33,141.67) ;
\draw [shift={(542.33,141.67)}, rotate = 180] [color={rgb, 255:red, 74; green, 74; blue, 74 }  ,draw opacity=1 ][line width=0.75]    (10.93,-4.9) .. controls (6.95,-2.3) and (3.31,-0.67) .. (0,0) .. controls (3.31,0.67) and (6.95,2.3) .. (10.93,4.9)   ;
\draw [color={rgb, 255:red, 74; green, 74; blue, 74 }  ,draw opacity=1 ]   (380.33,225) -- (419.06,152.1) ;
\draw [shift={(420,150.33)}, rotate = 117.98] [color={rgb, 255:red, 74; green, 74; blue, 74 }  ,draw opacity=1 ][line width=0.75]    (10.93,-4.9) .. controls (6.95,-2.3) and (3.31,-0.67) .. (0,0) .. controls (3.31,0.67) and (6.95,2.3) .. (10.93,4.9)   ;
\draw [color={rgb, 255:red, 74; green, 74; blue, 74 }  ,draw opacity=1 ]   (382.33,61.67) -- (421.06,134.57) ;
\draw [shift={(422,136.33)}, rotate = 242.02] [color={rgb, 255:red, 74; green, 74; blue, 74 }  ,draw opacity=1 ][line width=0.75]    (10.93,-4.9) .. controls (6.95,-2.3) and (3.31,-0.67) .. (0,0) .. controls (3.31,0.67) and (6.95,2.3) .. (10.93,4.9)   ;
\draw [color={rgb, 255:red, 155; green, 155; blue, 155 }  ,draw opacity=1 ] [dash pattern={on 0.84pt off 2.51pt}]  (224,190.67) -- (224,315) ;
\draw [color={rgb, 255:red, 155; green, 155; blue, 155 }  ,draw opacity=1 ] [dash pattern={on 0.84pt off 2.51pt}]  (344.67,269.67) -- (344.67,311.67) ;
\draw [color={rgb, 255:red, 155; green, 155; blue, 155 }  ,draw opacity=1 ] [dash pattern={on 0.84pt off 2.51pt}]  (460,268.33) -- (460,310.33) ;
\draw [color={rgb, 255:red, 155; green, 155; blue, 155 }  ,draw opacity=1 ] [dash pattern={on 0.84pt off 2.51pt}]  (576.67,187.33) -- (576.67,311.67) ;
\draw [color={rgb, 255:red, 74; green, 74; blue, 74 }  ,draw opacity=1 ]   (249.67,326.33) -- (320.33,326.33) ;
\draw [shift={(322.33,326.33)}, rotate = 180] [color={rgb, 255:red, 74; green, 74; blue, 74 }  ,draw opacity=1 ][line width=0.75]    (10.93,-4.9) .. controls (6.95,-2.3) and (3.31,-0.67) .. (0,0) .. controls (3.31,0.67) and (6.95,2.3) .. (10.93,4.9)   ;
\draw [color={rgb, 255:red, 74; green, 74; blue, 74 }  ,draw opacity=1 ]   (371.67,327) -- (437,327) ;
\draw [shift={(439,327)}, rotate = 180] [color={rgb, 255:red, 74; green, 74; blue, 74 }  ,draw opacity=1 ][line width=0.75]    (10.93,-4.9) .. controls (6.95,-2.3) and (3.31,-0.67) .. (0,0) .. controls (3.31,0.67) and (6.95,2.3) .. (10.93,4.9)   ;
\draw [color={rgb, 255:red, 74; green, 74; blue, 74 }  ,draw opacity=1 ]   (487.67,326.33) -- (553,326.33) ;
\draw [shift={(555,326.33)}, rotate = 180] [color={rgb, 255:red, 74; green, 74; blue, 74 }  ,draw opacity=1 ][line width=0.75]    (10.93,-4.9) .. controls (6.95,-2.3) and (3.31,-0.67) .. (0,0) .. controls (3.31,0.67) and (6.95,2.3) .. (10.93,4.9)   ;

\draw (214,167.67) node [anchor=north west][inner sep=0.75pt]  [font=\footnotesize,color={rgb, 255:red, 128; green, 128; blue, 128 }  ,opacity=1 ] [align=left] {000};
\draw (334.33,83) node [anchor=north west][inner sep=0.75pt]  [font=\footnotesize,color={rgb, 255:red, 128; green, 128; blue, 128 }  ,opacity=1 ] [align=left] {100};
\draw (334.33,167) node [anchor=north west][inner sep=0.75pt]  [font=\footnotesize,color={rgb, 255:red, 128; green, 128; blue, 128 }  ,opacity=1 ] [align=left] {010};
\draw (334.33,251) node [anchor=north west][inner sep=0.75pt]  [font=\footnotesize,color={rgb, 255:red, 128; green, 128; blue, 128 }  ,opacity=1 ] [align=left] {001};
\draw (449.67,81.67) node [anchor=north west][inner sep=0.75pt]  [font=\footnotesize,color={rgb, 255:red, 128; green, 128; blue, 128 }  ,opacity=1 ] [align=left] {110};
\draw (449.67,165.67) node [anchor=north west][inner sep=0.75pt]  [font=\footnotesize,color={rgb, 255:red, 128; green, 128; blue, 128 }  ,opacity=1 ] [align=left] {101};
\draw (449.67,249.67) node [anchor=north west][inner sep=0.75pt]  [font=\footnotesize,color={rgb, 255:red, 128; green, 128; blue, 128 }  ,opacity=1 ] [align=left] {011};
\draw (568.67,167.67) node [anchor=north west][inner sep=0.75pt]  [font=\footnotesize,color={rgb, 255:red, 128; green, 128; blue, 128 }  ,opacity=1 ] [align=left] {111};
\draw (74,178.07) node [anchor=north west][inner sep=0.75pt]    {$D$};
\draw (268.67,71.4) node [anchor=north west][inner sep=0.75pt]  [font=\footnotesize]  {$m$};
\draw (282,129.4) node [anchor=north west][inner sep=0.75pt]  [font=\footnotesize]  {$m$};
\draw (268.67,208.73) node [anchor=north west][inner sep=0.75pt]  [font=\footnotesize]  {$m$};
\draw (389.33,33.4) node [anchor=north west][inner sep=0.75pt]  [font=\footnotesize]  {$-m$};
\draw (416.67,77.4) node [anchor=north west][inner sep=0.75pt]  [font=\footnotesize]  {$m$};
\draw (414,108.07) node [anchor=north west][inner sep=0.75pt]  [font=\footnotesize]  {$-m$};
\draw (415.33,162.73) node [anchor=north west][inner sep=0.75pt]  [font=\footnotesize]  {$m$};
\draw (418,194.73) node [anchor=north west][inner sep=0.75pt]  [font=\footnotesize]  {$-m$};
\draw (394,234.73) node [anchor=north west][inner sep=0.75pt]  [font=\footnotesize]  {$m$};
\draw (522,76.07) node [anchor=north west][inner sep=0.75pt]  [font=\footnotesize]  {$\Delta $};
\draw (526,182.73) node [anchor=north west][inner sep=0.75pt]  [font=\footnotesize]  {$\Delta $};
\draw (507.33,123.4) node [anchor=north west][inner sep=0.75pt]  [font=\footnotesize]  {$-\Delta $};
\draw (212,317.07) node [anchor=north west][inner sep=0.75pt]    {$C^0$};
\draw (333.33,317.07) node [anchor=north west][inner sep=0.75pt]    {$C^1$};
\draw (446.67,317.07) node [anchor=north west][inner sep=0.75pt]    {$C^2$};
\draw (564.67,317.07) node [anchor=north west][inner sep=0.75pt]    {$C^3$};
\draw (274.67,305.07) node [anchor=north west][inner sep=0.75pt]    {$d$};
\draw (396.67,305.07) node [anchor=north west][inner sep=0.75pt]    {$d$};
\draw (514,305.07) node [anchor=north west][inner sep=0.75pt]    {$d$};

\end{tikzpicture}

%% file: 3.tex
\section{Preliminaries for instanton theory}\label{Preliminaries for instanton theory}

\subsection{Review of singular instanton  homology}
\label{subsection:Review of instanton knot homology}
In this subsection, we review the construction of 
the singular instanton functor:
\[I^{\sharp}:\bf{Link}(\R^3)\rightarrow \bf{Abel}_{\Z/4},\]
where $\bf{Link}(\R^3)$
denotes the category whose objects are links in $\R^3$ and morphisms are isotopy classes of (possibly unoriented and disconnected) link cobordisms in $[0,1]\times \R^3$ between links, and the category $\bf{Abel}_{\Z/4}$ denotes the category whose objects are absolutely $\Z/4$ graded $\Z$-modules and morphisms are 
graded $\Z$-module maps with any grading shifts. 

We briefly review the construction of singular instanton homology groups following Kronheimer and Mrowka's work \cite{KM11, KM11u}. For more details, see \cite{KM11, KM11u}. 
Let $L\subset Y$ be a link in an oriented $3$-manifold $Y$.   We fix an orbifold structure on $Y$ which is singular along the link $L$, whose local model is described by the quotient of $\R^2\times \R$ by the $\Z/2$-action  \[(x_1, x_2, x_3)\mapsto (-x_1, -x_2, x_3)\]
where the last factor 
$\R$ corresponds to the link locus.
Let $\check{Y}$ denote the $3$-manifold $Y$ equipped with the above orbifold structure. 
We fix an $SO(3)$-bundle ${P}\rightarrow Y\setminus L$ which extends as an orbifold bundle $\check{P}\rightarrow \check{Y}$.
Suppose the Poincar\'e dual of the second Stiefel--Whitney class 
$w_2 (P)$ 
is represented by a compact $1$-manifold $\omega$ in $(Y, L)$, which $\#\partial \omega \cap L_j$ is odd for at least one component $L_j \subset L$. 
From $\omega$, one can associate a {\it singular bundle data}, which provides an  $SO(3)$-bundle and a model $SO(3)$-orbifold connection which is independent of orientations of links. 
An $SO(3)$-orbifold connection on $\check{P}$ means a smooth $SO(3)$-connection defined over the link complement $Y\setminus L$ which extends smoothly over the $\Z/2$-branched cover of orbifold charts. One can see that an $SO(3)$-orbifold connection has a holonomy of order 2 along a small meridian of each link component. We consider the space of $SO(3)$-orbifold connections on the bundle $\check{P}$, denoted by $\mathcal{C}(Y, L, \check{P})$. 
 A smooth $SO(3)$-connection on $P\vert_{Y\setminus L}$ of order 2 holonomy along the shrinking meridians of $L$  extends as an orbifold $SO(3)$-connection on $\check{P}$.
 The notation $\mathcal{C}(Y, L)$ denotes the affine space of $SO(3)$-orbifold connections on $\check{P}$ with an appropriate Sobolev completion. 
We consider the bundle of group 
\[G(Y, L, \check{P}):=\check{P}\times_{SO(3)} SU(2).\]
over the orbifold $\check{Y}$, called the {\it determinant one gauge group}. 
Smooth orbifold sections of $G(Y, L, \check{P})$ are called determinant one gauge transformations of $\check{P}$.
The {\it gauge group} $\mathcal{G}(Y, L, \check{P})$ is an appropriate Sobolev completion of the space of smooth orbifold sections of $G(Y, L, \check{P})$. 
As in the non-singular case, the group $\mathcal{G}(Y, L, \check{P})$ has a structure of Hilbert Lie group and acts smoothly on $\mathcal{C}(Y, L, \check{P})$.
The quotient space 
\[\mathcal{B}(Y, L, \check{P}) := \mathcal{C}(Y, L, \check{P})/ \mathcal{G}(Y, L, \check{P})
\]
is called the {\it configuration space}. 
Framed singular instanton homology group is, roughly speaking, the infinite-dimensional analogue of Morse homology on the space $\mathcal{B}(Y, L,\check{P})$ for the Chern-Simons functional \[CS:\mathcal{B}(Y, L, \check{P}) \rightarrow \R/\Z.\]

In general, critical points of the Chern-Simons functional can be degenerated. 
We instead consider the $\pi$-perturbed Chern-Simons functional \[CS_{\pi}=CS+f_{\pi}.\]
Here, $\pi$ is assumed to be an element of the specific $l^1$-Banach space, and $f_{\pi}$ be an associated gauge invariant $\R/\Z$-valued functional.
The critical point set of the $\pi$-perturbed Chern-Simons functional $CS_{\pi}$ is denoted by
$\mathfrak{C}_{\pi}$. 
For an arbitrary small generic choice of perturbation $\pi$, we can assume that $\mathfrak{C}_{\pi}$ is non-degenerate \cite[Proposition 3.12]{KM11}.
 Then, the singular instanton homology $I^{\omega}(Y, L)$ is defined as the homology of the chain complex: 
%
\begin{eqnarray*}
    C^{\omega}(Y, L)&:=&\displaystyle \bigoplus_{\beta\in \mathfrak{C}^*_{\pi}}\Z\cdot \beta,\\
    d(\beta_1)&:=&\sum_{\beta_2\in \mathfrak{C}^*_{\pi}}n(\beta_1, \beta_2)\beta_2, 
\end{eqnarray*}
where $n(\beta_1, \beta_2)$ is defined by oriented counting of the $0$-dimensional components of the moduli space 
\[
\breve{M}(\beta_1, \beta_2) := \left\{ [A] \in \mathcal{B}(\R \times \check{Y}; \beta_1, \beta_2,  p^*\check{P} ) \ \middle| \  F^+_A + V_\pi(A) =0 \right\}/\R,
\]
where 
\begin{itemize}
\item $p$ denotes the projection $\R \times \check{Y} \to \check{Y}$,
\item $F_A^+$ denotes the self-dual part of the curvature of $A$ with respect to $p^* \check{g} + dt^2$ for a choice of an orbifold metric for $(Y, L)$, where $t$ denotes the cylindrical coordinate, 
    \item $V_\pi(-)$ denotes the holonomy perturbation of the ASD-equation on $\R \times \check{Y}$ corresponding to $\pi$,  and 
    \item $\mathcal{B}(\R \times \check{Y};  \beta_1, \beta_2 , \pi^* \check{P}) $ denote the $4$-dimensional version of the orbifold configuration space whose limiting values are $\beta_1$ and $\beta_2$ respectively. 
\end{itemize}
We shall explain how these orientations are defined in \Cref{ori app}.

Let us move to link cobordism maps in this theory. 
We fix a smooth compact $4$-manifold $W$ from an oriented $3$-manifold $Y$ to another $3$-manifold $Y'$ and 
a link $L$ in $Y$ (resp.\ $L'$ in $Y'$) with admissible bundle $\om$ (resp.\ $\om'$).
Let us put a link cobordism $S \subset W$ from $L \subset Y$ to $L' \subset Y'$ with an extension of the admissible bundles of the boundary.  Then, for choices of 
\begin{itemize}
    \item orbifold Rieman metrics $\check{g}_L$ and $\check{g}_{L'}$, non-degenerate and regular holonomy perturbations $\pi_L$ and $\pi_{L'}$ for Chern-Simons functional of $(Y, L)$ and $(Y', L')$, 
    \item an orbifold Riemann metric $\check{g}$ on  
    \[
   \overline{S}:=  (-\infty, -1] \times L\cup S \cup [1, \infty) \times L'  
   \ \subset \ (-\infty, 0 ]\times S^3 \cup_{Y} W\cup_{Y'} [0, \infty )\times Y' =: \overline{W}
    \]
    which are $dt^2 + \check{g}_L$ and $dt^2+ \check{g}_{L'}$ on the ends, 
    \item holonomy perturbations on $(\overline{W}, \overline{S})$ so that the finite energy instanton moduli spaces are regular, 
    \item critical points $\beta_1$, $\beta_2$ of perturbed Chern--Simons functionals for $(Y, L)$ and $(Y', L')$, 
\end{itemize}
one can associate 
a singular instanton moduli space
\begin{align}\label{gen moduli}
{M}_{\check{g}}(W, S ; \beta_1, \beta_2) := \left\{ [A] \in \mathcal{B}(W, S; \beta_{1}, \beta_{2}, \pi^*\check{P}) 
\ \middle| \  F^{+_{\check{g}} }_A + V_\pi(A) =0  \right\}. 
\end{align}
Then, we define 
\[
I^\sharp(S) : I ^\omega (Y, L) \to I ^\omega (Y', L')
\]
by just counting the $0$-dimensional part of the instanton moduli space ${M}_{\check{g}}(W, S ; \beta_1, \beta_2)$. 
This makes sense even for non-orientable surfaces, and such bundle data can be described as singular bundle data.  For the details, see \cite{KM11u}. 

Now, for a given link $L\subset \R^3$, the {\it 
framed singular instanton homology
} $I^{\sharp}$ is defined as 
 \begin{eqnarray*}
 I^{\sharp}(L):=I^{\omega}(\R^3 \cup \{\infty \} , L\sqcup H),  
 \end{eqnarray*} 
 where $H$ is the Hopf link put near the $\infty \in \R^3 \cup \{\infty\} = S^3$ with an arc connecting two components as $\omega$.

Let $W$ be an oriented compact 4-manifold cobordism from $S^3 \to S^3$ with a fixed path connecting base points in $S^3$.
 For a given link cobordism $S \subset W$ with an path connecting base points $\infty \in S^3$ from $L \subset S^3$ to $L' \subset S^3$, we consider the associated link cobordism 
\[
{S}^\sharp := S \sqcup [0,1]\times H
\ \subset \  W 
\]
with the admissible bundle $\check{P}$ whose Stiefel-Whitney class is given by $[0,1]\times \omega$ in $W$, where the embedding of $[0,1]\times H$ into $W$ is made by taking a neighborhood of the fixed path. 
We shall define the {\it framed cobordism map} by 
\[
I^\sharp(S) := I^{[0,1]\times \omega}_S : I^\sharp(L) \to  I^\sharp(L'). 
\]
Also, when a dot in $S$ is associated, one can define the dotted cobordism map 
\[
I ^\sharp(S, \cdot)   : I^\sharp(L) \to  I^\sharp(L')
\]
by evaluating the first Chern class of the $U(1)$-universal bundle arising from the base point fibration for the point in the instanton moduli spaces. See \cite{KM11u, KM21} for the details of the maps.

Again the orientations will be explained in \Cref{ori app}. Note that the cobordism maps are well-defined even for 4-manifolds with more than two boundary components with suitably fixed paths. See \cite{KM11u}.

\subsection{Instanton cube complexes}

In this section, we briefly review the construction of instanton cube complexes. For a given link diagram $D$ of $L \subset \R^3$, we shall introduce a doubly filtered complex $\CKh^\sharp (D)$
which induces a spectral sequence from $\Kh(L^*)$ to $I^\sharp(L)$.

Let $L$ be a link in $\R^3 \subset S^3$. The point $\infty \in S^3$ is regarded as the base point.
A {\it crossing} of $L$ means an embedding of pairs
\[
c : (B^3, L_2) \hookrightarrow (\R^3, L)
\]
which is 
orientation-preserving on $B^3$ and $c(L_2)= c(B^3) \cap L$, where $L_2$ is the (unoriented) standard tangle described in the middle in \cref{fig:1}. 

Take an ordered set $\{c_1, \cdots, c_N \}$
of crossings of $L$. 
For an element $v \in \Z^N$, we define a link $L_v$ by replacing $c_i(L_2)$  with either $c_i(L_0)$, $c_i(L_1)$ or $c_i(L_2)$ depending on the value $v(i) \operatorname{mod} 3$ ($i=1,2,\ldots, N$), where $L_0$, $L_1$ and $L_2$ are given in \cref{fig:1}. 
For a pair $v,u \in \Z^N$ with $v \geq u$, we have a link cobordism
 \[
 S_{vu} : L_v \to L_u \ \subset \  [0,1]\times S^3
 \]
used in \Cref{subsec:Khovanov}. 
 For this link cobordism $S_{vu}  \subset [0,1] \times S^3$, Kronheimer--Mrowka \cite{KM11u}
 constructed a family of orbifold metrics called $G_{vu}$ over 
 \[ S^\sharp :=
\Big( ( -\infty , -1]  \times (L_v \sqcup H) \Big) \cup S_{vu} \cup \Big( [1, \infty)  \times (L_u \sqcup H) \Big)  \ \subset \  \R \times S^3 
 \]
 such that 
 \[
 \dim G_{vu}  = | v -u |_1 = \sum_{i=1 }^N |v_i - u_i|, 
 \]
with a proper $\R $-action obtained as $\R$-translation. It is observed that $| v -u|_1  = -\chi (S_{vu})$.
We will use $\breve{G}_{vu} := G_{vu}/ \R$ whose dimension is 
\[
\begin{cases} 
 | v -u|_1-1 &\text{ if } v \neq u \\ 
  0 & \text{ if } v = u
\end{cases}. 
\]
Let us denote the framed 
singular instanton
complex (by putting the Hopf link with non-trivial Stiefel--Whitney class at $\infty \in S^3$) of $L_v$ by $C_v$ which are again parametrized by $\Z^N$.
In order to get the cube instanton complex, we fix the following data: 
\begin{itemize}
\item orbifold Riemann metrics of $(S^3, L_v)$ for each $v \in \Z^N$,
    \item regular and non-degenerate perturbations $\{\pi_v \}$ of the Chern--Simons functional for $L_v$, 
    \item holonomy perturbation for the cobordisms $S^\sharp_{vu}$ so that $M_{G_{vu}'} (\beta_1, \beta_2; S^\sharp_{vu})$ is regular for each pair of critical points $\beta_1, \beta_2$.
\end{itemize}. 

Then, as signed countings (with certain sign corrections) of the $0$-dimensional component of the parametrized and perturbed ASD-moduli spaces with respect to the family $G_{vu}$
\[
\breve{M}_{vu}(\beta_{1}, \beta_{2})_{0}:= \bigcup_{g \in {\breve{G}_{vu}}} {M}_g ([0,1] \times S^3, S^\sharp_{vu}; \beta_1, \beta_2)_{-\mathrm{dim}\breve{G}_{vu}}
\]
we have a collection of maps
\[
f_{vu} : C_v \to C_u
\]
parametrized by each pair $(u, v) \in \Z^N \times \Z^N$ with $v \geq u$. 
The sign of $f_{vu}$ is given as 
\[
f_{vu}(\beta_1) =\sum_{\beta_2 }  (-1)^{s(v,u)-\mathrm{dim}\breve{G}_{vu}} \#\breve{M}_{vu}(\beta_{1}, \beta_{2})_{0}, 
\]
where 
\[
s(v,u)=\frac{1}{2}|v-u|(|v-u|-1)+ \sum v_i. 
\]

\begin{rem}Here we mention that the definition of $f_{uv}$ does not match with that of \cite{KM11u}. 
We use the notation $\breve{M}_{vu}(\beta_{1}, \beta_{2})_{0}$ to denote the moduli space parametrized by family of metrics ${\breve{G}_{vu}}$ over the standard cobordism $S_{vu}$ with critical limiting points $\beta_{1}$ and $\beta_{2}$ on the ends. In \cite{KM11u}, Kronheimer and Mrowka also use the same notation to denote the essentially same moduli space equipped with a different orientation, which is described as follows: 
Suppose the moduli space $\breve{M}_{vu}(\beta_{1}, \beta_{2})$ is naturally oriented as a family of instanton moduli spaces.
Put \[\langle\bar{m}_{vu}(\beta_1), \beta_2\rangle:=(-1)^{\mathrm{dim}\breve{G}_{vu}}\#\breve{M}_{vu}(\beta_{1}, \beta_{2})_{0}.\]
Then, with this convention, we have:  
\[
f_{vu} :=\sum_{\beta_2 }  (-1)^{s(v,u)} \bar{m}_{vu}.
\]
\end{rem}
For a pair $(v,u)$, we call it {\it of type $n$} if $v \geq u$ and 
\[
\max\{ v_i - u_i \} = n.  
\] 
The neck stretching argument enables us to prove the following conclusion, which is proven in \cite{KM11u}. 
\begin{thm}
Let $(v,u)$ be a pair of type $n \leq 2$. 
Then one has 
\[
\sum_{\{ w |  v \geq w \geq u \} } f_{wu} \circ  f_{v w} =0. 
\]
\end{thm}

For a fixed $(v,u)$ of type $\leq 1$ with $v \geq u$, we define a chain complex $( \bC[vu], \bF[vu] )  $ by
\[
\left( \bC[vu] :  = \bigoplus_{v \geq v' \geq u}  C_{v'} , \quad \bF[vu]  := \bigoplus _{v \geq v' \geq u' \geq u } f_{v'u'} \right) . 
\]

If we put $v=(1, \cdots, 1)$ and $u = (0, \cdots, 0)$, then we define
\[
(\CKh^\sharp(D), d^\sharp):=  ( \bC[vu],\bF[vu] ) 
\]
and call it by an {\it instanton cube complex}. 

 \begin{thm}[Kronheimer--Mrokwa, \cite{KM11u}]
The homology of $(\CKh^\sharp(D), d^\sharp )$ is canonically isomorphic to $I^\sharp (L)$ for a given link $L \subset \R^3$. 
\end{thm}

\subsection{Filtrations on instanton cube complexes}

We review how to define quantum and cohomological gradings on instanton cube complexes, which are originally defined in \cite{KM14}.

Those gradings are defined for the following extension of link diagram introduced by Kronheimer--Mrowka: 
\begin{defn}
    We say that a link $L \subset \R^3$ with a collection $\{c_i\}_{i=1}^N$ of crossings is a {\it pseudo-diagram} of $L$ if, for all $v \in \{0,1\}^N$, the link $L_v$ is an unlink.  (We often denote the pair $(L,\{c_i\}_{i=1}^N)$ by $D$, and set $D_v := L_v$ for each $v \in \Z ^N$.)
\end{defn}

For a pseudo-diagram $D=(L, \{c_i\})$, we denote
$\CKh^\sharp(D) := \CKh^\sharp(L, \{c_i\})$.
If we consider a strongly admissible diagram, then the complex $ ( \bC[vu],\bF[vu] ) $ gives a spectral sequence convergent to the homology of $ ( \bC[ww],\bF[ww])$for $w=2v-u$.

Let $L$  be a link in $\R^3$ and $D$ be its 
pseudo-diagram. 
We first note that the critical point set of the orbifold Chern--Simons functional is identified with a certain $SU(2)$-representation variety, as in the case of usual instanton Floer homology: 
\[
R^\omega(L \sqcup H) =\{ \rho \in \operatorname{Hom}(\pi_1(S^3 \setminus L \sqcup H \sqcup \omega ), SU(2)) | \operatorname{Tr}(\rho (m_i))=0, \rho (l) = -\id \}/SU(2),
\]
where the $m_i$ are meridians of each component of $L \sqcup H$ and $l$ is a small meridian of an arc $\omega$ representing $w_2$ of the admissible $SO(3)$-bundle. 

Note that for a given orientation, we have an identification 
\[
R^\omega(D_v \sqcup H) \cong \overbrace{S^2\times \cdots \times S^2}^{r(D_v)}.  
\]
We take a Morse perturbation of $R^\omega(D_v \sqcup H)$ of the orbifold Chern--Simons functional so that the perturbed critical point set is identified with the Morse complex of the sum of standard Morse functions on $S^2$. 
Through the identification 
\[
C^\sharp (D_v) \cong \overbrace{H_*(S^2; \Z)\otimes \cdots \otimes H_*(S^2; \Z)}^{r(D_v)}, 
\]
we define the quantum degree $Q(\zeta)$ of $\zeta$ in $C^\sharp (D_v) $ as its absolute Morse degree. One can alternatively use excision cobordism maps to define $Q$-filtrations and see that these are filtered equivalent as mentioned in \cite{KM14}. 
For a fixed orientation of $L$, we have the corresponding unique resolution $D_o$ for $o \in \{0,1\}^N$, which is compatible with the induced orientation of $D$.
\begin{defn}
Let $D$ be a link diagram with $N$ crossings $\{c_{1}, \cdots c_{N}\}$.
Fix a resolution $v \in \{0,1\} ^N $ of $D$.
We 
define the {\it instanton quantum grading} by 
\[
q (\zeta) : = Q (\zeta) - \left( \sum_{i \in \{ 1, \cdots, N\}  } v(c_i) \right) + \frac{3}{2}\sigma (v, o) - n_+ + 2n_- 
\]
for $0\neq \zeta \in C^\sharp(D_v)$, where 
\begin{itemize}
    \item[(I)]  $n_-$ and $n_+$ are the numbers of negative and positive crossings of $D$, 
    \item[(II)] $\sigma (v,u) \in \Z$ denotes 
    \[
    S_{v w} \cdot S_{vw}  - S_{uw }\cdot  S_{uw}, 
    \]
    where $\cdot$ means the self-intersection number in $[0,1]\times \R^3$ and $w \in \{ 0,1\}^N $ is any resolution such that 
    $v \geq w, u \geq w $.
\end{itemize}
\end{defn}
Note that this definition is analogous to \eqref{kh_q_gr}. 
Next, we give the definition of homological grading:
\begin{defn}
For a resolution $v \in \{0,1\} ^N $, we define the {\it instanton homological grading} $h$ as follows: 
\[
h( \zeta) :=  - \left( \sum_{i \in \{ 1, \cdots, N\}  } v(c_i) \right) + \frac{1}{2}\sigma (v, o)  + n_- 
\]
for $0\neq \zeta \in C^\sharp (D_v)$.  
For $ 0 \in C^\sharp(D_v)$, we define these gradings as $\infty$.
\end{defn}
This definition looks analogous to the cohomological grading in \eqref{kh_ho_gr}.

\begin{thm}[Kronheimer--Mrokwa, \cite{KM11u}]
    We have an identification between $E_2$-term of $(\CKh^\sharp(D), d^\sharp )$ and the Khovanov chain complex of $D^*$. 
\end{thm}

\begin{rem}
 Although Kronheimer--Mrokwa focused on the reduced version of the instanton cube complexes in the main theorem of \cite{KM11u}, the setting of Proposition 6.7, Theorem 6.8, and Theorem 8.2 in \cite{KM11u} contains the case of framed cube complexes.  
\end{rem}

\subsection{Isotopy trace maps on instanton cube complexes}\label{isotopy_trace}
In this section, we shall introduce a cobordism map for cube complexes. Originally, Kronheimer--Mrowka defined certain cobordism maps for isotopy traces in \cite[Section 5]{KM14}, but we give a slightly more general definition.

We take a link cobordism in 
\[
T\subset [0,1]\times \R^3: D \to D' 
\]
such that $D$ and $D'$ are the same pseudo diagram near each crossing and the cobordism $T$ is product on neighborhoods of the crossings. Then, we define 
\begin{align}\label{Tbordism}
([0,1] \times S^3, T_{uv} ) 
:= ([0,1] \times S^3, T ) \circ  ([0,1]\times S^3, S_{uv}). 
\end{align}
One can make a family of orbifold Riemann metrics with cylindrical ends $G^T_{uv}$ with dimension $|u-v|_1$.
We define a cobordism map 
\begin{align}\label{familycobmap}
\phi^{KM}_T: \CKh^{\sharp} (K) \to \CKh^{\sharp}(K')
\end{align}
by counting parametrized instanton moduli spaces over $G^T_{uv}$ with a certain correction of orientations. More precisely, we define 
\begin{eqnarray*}
    \phi^{KM}_{T}&:=&\sum_{1\geq u\geq v\geq 0}\phi^{KM}_{T_ {uv}}
    \end{eqnarray*}
    by setting
    
    \begin{eqnarray*}
  \phi^{KM}_{T_{uv}}:=(-1)^{s^{T}(u, v)}m^{T}_{uv}, \quad   \langle m_{uv}^{T}(\alpha), \beta \rangle:=\#M^{T}_{uv}(\alpha, \beta),
\end{eqnarray*}
where 
\[s^{T}(u, v):=\frac{1}{2}|u-v|_{1}(|u-v|_{1}-1)+\sum_{i}u_{i}-\sum_{i}v_{i}\]
and $M^{T}_{uv}(\alpha, \beta)$ is singular instanton moduli space parametrized by $G^T_{uv}$. 
See \Cref{ori_of_KMmap} for the orientations of these moduli spaces.
\begin{prop}\label{chain map eqn}
The map $\phi^{KM}_{T}$ is a chain map on cube complexes.
Moreover, if $T$ is a trace of isotopy fixing the neighborhoods of the crossings of pseudo-diagrams, then $\phi^{KM}_{T}$ is q-filtered. 
\end{prop}
\begin{proof}

The relation of the chain map is given by counting the oriented boundary of the 1-dimensional moduli space
$M^{T}_{uv}(\alpha, \beta)_{1}$.
\Cref{KMT_orientation} in Appendix gives the comparison of induced orientation on each boundary face of $M^{T}_{uv}(\alpha, \beta)_{1}$.
Hence, we have
        \begin{equation}\label{eqn3}
        \sum_{u > u'} (-1)^{(|u-u'|_{1}+1)(|u'-v|_{1}-1)} m^{T}_{u'v}\circ \bar{m}_{uu'} 
        - \sum_{v' > v} (-1)^{|u-v'|_{1}|v'-v|_{1}+|v'-v|_{1}-1} \bar{m}_{v'v} \circ m^{T}_{uv'} = 0.
\end{equation}
This equation is equivalent to
\begin{equation}\label{eqn4}
    \sum_{u > u'} (-1)^{s(u, u')+s^{T}(u', v)} {m^{T}_{u'v}} \circ \bar{m}_{uu'} 
        - \sum_{v' > v} (-1)^{s^{T}(u, v')+s(v', v)} \bar{m}_{v'v} \circ m^{T}_{uv'} = 0.
\end{equation}
To see this, for a positive integer $N$ and 
\[
u , v, u', u'', v, v'' \in \{ 0,1\}^N, 
\]
we put
\begin{eqnarray*}
    \delta_{1}(u', u'')&:=&(|u-u'|_{1}+1)(|u'-v|_{1}-1)-(|u-u''|_{1}+1)(|u''-v|_{1}-1),\\
    \delta_{2}(u', v')&:=&(|u-u'|_{1}+1)(|u'-v|_{1}-1)-|u-v'|_{1}|v'-v|_{1}-|v'-v|_{1}+1,\\
    \delta_{3}(v', v'')&:=&|u-v'|_{1}|v'-v|_{1}+|v'-v|_{1}-|u-v''|_{1}|v''-v|_{1}-|v''-v|_{1}.
\end{eqnarray*}
We also put
\begin{eqnarray*}
    \epsilon_{1}(u', u'')&:=&s(u, u')+s^{T}(u', v)-s(u, u'')-s^{T}(u'', v),\\
    \epsilon_{2}(u', v')&:=&s(u, u')+s^{T}(u', v)-s^{T}(u, u')-s(u', v),\\
    \epsilon_{3}(v', v'')&:=&s^{T}(u, v')+s(v', v)-s^{T}(v, v'')-s(v'',v). 
\end{eqnarray*}
Then we have:
\begin{claim}
Fix $u \geq v \in \{0,1\}^N$.  For 
\[
u', u'', v', v'' \in \{0,1\}^N , 
\]
we have 
\begin{align*}
\begin{cases}
&\delta_{1}(u', u'')\equiv_{(2)}\epsilon_{1}(u', u'') \text{ if } u \geq u'\geq  u''\geq v \\
&\delta_{2}(u', v')\equiv_{(2)}\epsilon_{2}(u', v') \text{ if }
u \geq u'\geq  v'\geq v \\
&\delta_{3}(v', v'')\equiv_{(3)}\epsilon_{3}(v', v'') \text{ if } u \geq v'\geq  v''\geq v. 
\end{cases}
\end{align*}
\end{claim}
These imply that (\ref{eqn3}) and (\ref{eqn4}) are equivalent.
Hence, we obtain the relation:
\[\phi^{KM}_{T}\circ d^{\sharp}-d'^{\sharp}\circ \phi^{KM}_{T}=0.\]
The latter part of the statement is proven in \cite{KM14}. 
\end{proof}

This is proven by elemental computation, but we write a sketch of proof for readers. 
\begin{proof}[Proof of the claim]
Let $u\ge v\in\{0,1\}^N$ and let $w\in\{u',u'',v',v''\}$ satisfy $u\ge w\ge v$ coordinatewise. 
For each coordinate we have $(u_i,v_i)\in\{(0,0),(1,1),(1,0)\}$, and on coordinates with $(u_i,v_i)=(1,0)$
\[
|u_i-w_i|+|w_i-v_i|=1.
\]
Summing over all coordinates gives the fundamental identity
\begin{equation}\label{eq:dist}
|u-w|_1+|w-v|_1 = |u-v|_1 =: n.
\end{equation}
Let
\[
a=|u-u'|_1,\ b=|u'-v|_1,\ c=|u-u''|_1,\ d=|u''-v|_1,
\]
then \eqref{eq:dist} implies
\begin{equation}\label{eq:ab}
a+b = n = c+d.
\end{equation}

A direct expansion of $\epsilon_1$ using the definitions of $s$ and $s^T$ gives
\[
\epsilon_1(u',u'') \equiv \frac12\bigl(a(a-1)+b(b-1)-c(c-1)-d(d-1)\bigr)\pmod 2.
\]
On the other hand,
\[
\delta_1(u',u'') = (a-1)(b-1)-(c-1)(d-1)=ab-cd-a-b+c+d,
\]
which equals $ab-cd$ by \eqref{eq:ab}.  Consider
\[
D:=a(a-1)+b(b-1)-c(c-1)-d(d-1)-2(ab-cd),
\]
which simplifies via \eqref{eq:ab} to $D=4(a-c)(a+c-n)$. Hence $D\equiv 0\pmod 4$ and
\[
\epsilon_1(u',u'')\equiv \delta_1(u',u'')\pmod 2.
\]

The proofs of $\delta_2\equiv\epsilon_2$ and $\delta_3\equiv\epsilon_3$ are identical polynomial manipulations:
set
\[
|u-u'|_1=a,\ |u'-v|_1=b,\ |u-v'|_1=e,\ |v'-v|_1=f
\]
or
\[
|u-v'|_1=e,\ |v'-v|_1=f,\ |u-v''|_1=g,\ |v''-v|_1=h,
\]
and use $a+b=n=e+f=g+h$ from \eqref{eq:dist}. In each case
\[
2(\epsilon_i-\delta_i) \in 4\Z,
\]
so $\epsilon_i\equiv \delta_i\pmod 2$.
\end{proof}

For the trace of isotopy between $K$ and $K'$, the chain map $\phi^{KM}_{T}$ gives a $q$-filtered chain homotopy equivalence between the associated cube complexes. See \cite[Proposition 5.1]{KM14} for the detail.
Note that $\phi^{KM}_T$ combined with certain dropping crossing maps, which will be explained in the next section, has been used to get maps between cube complexes corresponding to Reidemeister moves. 

\subsection{Kronheimer--Mrowka's maps on instanton cube complexes}

Since our main result involves Kronheimer--Mrowka's cobordism maps $\phi^{KM}_S$, we shall briefly review their constructions.

First, we have maps called {\it dropping/adding crossings}. 
Let $(D,N)$ be a pseudo diagram, where $N$ describes the set of crossings.  Pick a crossing $c \in N$ such that $(D, N':= N \setminus \{c\})$ is still a pseudo diagram. Then, we have a decomposition 
\[
CKh^\sharp_i(D, N) := \bigoplus_{\substack{u \in \{0,1\}^{N}\\ u(c) = i }} C^\sharp (D_u).   
\]
Then, the {\it dropping map }
\[\Phi^{\sharp}:\CKh_{1}^\sharp(D, N)\oplus \CKh_{0}^\sharp(D, N)\rightarrow \CKh^\sharp(D, N')\] is defined by
$\Phi^\sharp =\left[{F}_{1,-1}, {F}_{0,-1}\right]$
where ${F}_{ij}$ are the components of the differential on the cube complex of $(D,N)$ induced from a standard cobordism $S_{uw}$ such that $u(c)=i$ and $w(c)=j$. 
In other words, these maps are written as: 
\begin{align*}
  &  {F}_{ij} = \sum_{ \substack{v \in \{0,1\}^N,  v(c)=i \\ u \in \{0,1\}^N, v \geq u,  u(c)=j } } f_{vu} 
\end{align*}

If we further suppose the pair $(N,N')$ is {\it admissible}, this map $\Phi^{\sharp}$ is confirmed to be a filtered chain homotopy equivalence with respect to quantum gradings. 
We also have the opposite direction of the map:
\[
\Psi^\sharp : \CKh^\sharp(D, N') \to \CKh_{1}^\sharp(D, N)\oplus \CKh_{0}^\sharp(D, N)
\]
defined in a similar way, which is called the {\it adding map}. 

If we take one more crossing $c' \in N$ and put $N'':= N \setminus \{c, c'\}$ with certain conditions, we have similar quantum filtration preserving maps 
\begin{align*} \CKh^\sharp(D, N) \xrightarrow{\Psi^\sharp} \CKh^\sharp(D, N') \xrightarrow{\Psi^\sharp} \CKh^\sharp(D, N'')  \\ 
\CKh^\sharp(D, N'') \xrightarrow{\Phi^\sharp} \CKh^\sharp(D, N')  \xrightarrow{\Phi^\sharp} \CKh^\sharp(D, N)
\end{align*}
such that each $\Psi^\sharp$ and $\Phi^\sharp$ are inverse up to homotopy each other.

Now, we explain the cobordism maps for Reidemeister moves given in \cite[Proposition 8.1]{KM14}. 
Let $S: D\to D'$ be one of the Reidemeister RI.  
We first apply the map associated to the trace of isotopy $T$ and compose it with the adding map: 
\[
\phi^{KM}_S := \Psi^\sharp \circ \phi^{KM}_T : \CKh^\sharp (D)  \to \CKh^\sharp (D'). 
\]
For $RI^{-1}$, we use the dropping map combined with the isotopy map. The maps of the other moves are similar, see \cite[Figure 2]{KM14} for RIII.  

Next, we consider Morse moves. Suppose $S$ is a 0-handle attachment from $D$ to $D'= D \sqcup U_1$. This gives an explicit disk $S$ cobordism from $D$ to $D'$. We have associate map \eqref{familycobmap} for $S$. This is the definition of $0$-handle map. The $ 2$-handle case is similar. 
The 1-handle case is described as follows: 
 $S$ is a 1-handle cobordism from $D$ to $D'$. We add the crossing $c$ on the diagram $D$ written by $D''$, and regard the $1$-handle attach operation as a cobordism map induced from the change of  $1$-resolution $D$ to $0$-resolution $D'$ for the crossing $c$.
Then we have the standard link cobordism $S_{uw}$ such that $|u-w_{1}|_{\infty}=1$, $u(c)=1$, and $ w_{1}(c)=0 $ for a specified crossing $c$.
This induces a chain map 
\[
\phi^{KM}_S : \CKh^\sharp (D''_1=D) \to \CKh^\sharp (D''_0= D') 
\]
as a component of the differential of $\CKh^\sharp (D'')$.

%% file: 4.tex
\section{Excision cobordism map}
\label{Excision cobordism map}
In this section, we shall construct the {\it excision map}
\[
\Psi: \CKh^\sharp (D_1) \otimes \CKh^\sharp (D_2) \to \CKh^\sharp (D_1\sqcup D_2)
\]
for pseudo diagrams $D_1$ and $D_2$ of given links. Originally, excision cobordism maps were introduced by Floer and further discussed in \cite{BD95} due to Braam and Donaldson. 
Also, it was used to prove well-definedness of sutured instanton homology \cite{KM10} and to compute the $E_2$-term of the spectral sequence in \cite{KM11u}. Our construction is based on Kronheimer--Mrowka's cobordism map argument in \cite{KM10}.

Let $K_1$ and $K_2$ be links in $\R^3$. 
Let us regard $S^3$ as the compactification $\R^3 \cup \{ \infty \} $.
When we consider the framed instanton homology, we further put the Hopf links denoted by $H$ on a small neighborhood of $\infty$ in $S^3$.
To form the excision map $\Psi$, we first construct the {\it excision cobordism} 
\[
(W, S) :  (S^3, K_1\sqcup H )  \sqcup (S^3, K_2\sqcup H ) \to (S^3, H) \sqcup (S^3, K_1\sqcup K_2\sqcup H ) 
\]
described as follows: 
Fix a label of $H$ as $H= H^+_i \cup H^-_i$,
and let $V^+_1$ (resp.\ $V^+_2$) be a tubular neighborhood of $H^+_1$ (resp.\ $H^+_2$)
in $S^3$, which are again contained in a small neighborhood of $\infty \in S^3$.
Then $Y'_i$ is obtained from $S^3$
by cutting open along $T_i := \partial V^+_i$ ($i=1,2$).
Note that each $Y'_i$ is regarded as the disjoint union 
 of $V^+_i$ and $V^-_i := S^3 \setminus \Int (V^+_i)$,
 where $V^-_i$ is a solid torus with center line $H^-_i$.
Moreover,  there is a 3-ball $B_i$ in $V^-_i\setminus H^-$ containing $K_i$ ($i=1,2$).
Next, we consider a diffeomorphism $h \colon T_1 \to T_2$ interchanging longitudes 
and meridians,
and let $U$ be the 2-dimensional manifold with corners shown in \Cref{fig:excision1}. 
The boundary of $U$ is decomposed into the lower horizontal arcs $l_1, l_2$, the upper horizontal arcs $u^+$, $u^-$, and the vertical arcs $s^\pm_1, s^\pm_2$. 
Take a Morse function $f \colon U \to [0,1]$ with a single index 1 critical point so that $f^{-1}(0)= l_1 \cup l_2$, $f^{-1}(1) = u^+ \cup u^-$ and the restriction on each $s^\pm_i$ is bijective.
Then, we define 
\[
W := \left(U \times T_1 \right) 
\cup \left( [0,1] \times Y'_1 \right)
\cup  \left( [0,1] \times Y'_2 \right)
\]
with the following gluing properties:
for each $\varepsilon \in \{\pm 1\}$,
\begin{itemize}
    \item $s^{\varepsilon}_1 \times T_1$ is glued with $[0,1] \times \partial V^{\varepsilon}_1$ by $f \times \id$, and
    \item $s^{\varepsilon}_2 \times T_1$ is glued with $[0,1] \times \partial V^{\varepsilon}_2$ by $f \times h$.
\end{itemize}
It follows from the construction that the domain of the cobordism $W$ is
\[
(V_1^+ \cup V_1^-) \sqcup (V_2^+ \cup V_2^-) = S^3 \sqcup S^3
\]
and the codomain is 
\[
(V_1^+ \cup V_2^+) \sqcup (V_1^- \cup V_2^-) = S^3 \sqcup S^3.
\]

\begin{figure}[t]
    \centering
    \includegraphics[width=0.4\linewidth]{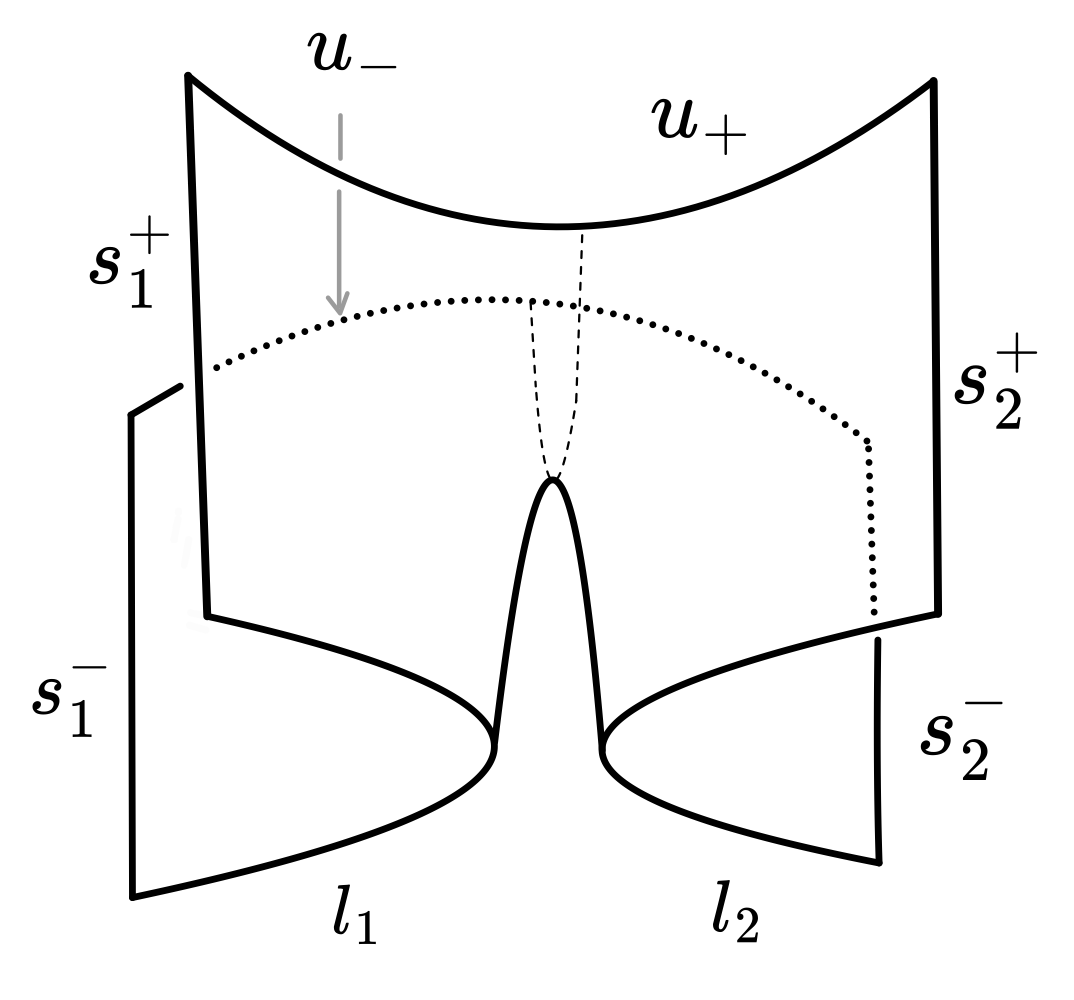}
    \caption{}
    \label{fig:excision1}
\end{figure}

One can take a Morse function $W \to [0,1]$ 
as a natural extension of the product of $f$ with a standard Morse function on $T_1 \cong S^1 \times S^1$, which induces a handle decomposition of $W$ consisting of 4 handles, whose indices are $1$, $2$, $2$ and $3$ respectively. (We denote the Morse function on $W$ by $f$ again.)

Finally, we define a link cobordism $S$ by
\[
S:= [0,1] \times (K_1\sqcup H \sqcup K_2 \sqcup H) 
\subset \left( [0,1] \times Y'_1 \right)
\cup  \left( [0,1] \times Y'_2 \right)
\subset W.
\]
We write $Y_1' \sqcup Y_2'$
as
$Y_{\sqcup}$ and record this embedding $Y_{\sqcup} \hookrightarrow W$.
By the construction of $W$, we see that the domain of $(W,S)$ is
\[
\left(V^+_1 \cup V^-_1, H^+_1 \sqcup (K_1 \sqcup H^-_1) \right) 
\sqcup 
\left(V^+_2 \cup V^-_2, H^+_2 \sqcup (K_1 \sqcup H^-_2) \right) 
\]
and the codomain is
\[
\left(V^+_1 \cup V^+_2, H^+_1 \sqcup H^+_2 \right) 
\sqcup 
\left(V^-_1 \cup V^-_2, (K_1 \sqcup H^-_1) \sqcup (K_2 \sqcup H^-_2) \right).
\]
We consider the admissible bundles on $(S^3, K_1\sqcup H )$, $(S^3, K_2\sqcup H )$, $(S^3, H)$, and $(S^3, K_1\sqcup K_2\sqcup H ) $ as $SO(3)$-bundles whose Stiefel--Whitney class are arcs connecting the two components of the Hopf link. This bundle on boundary naturally has an extension to an $SO(3)$-bundle on $(W, S)$.

\begin{lem}\label{topological computation of excision}
    One can see 
    \[
    W \cong \left([0,1]\times S^3 \right) \# \left([0,1]\times S^3 \right) \# \left(S^2 \times S^2 \right). \]
    In particular, $\chi(W)=\sigma(W)=0$.
    Moreover, we see $\chi(S)=0$ and $S \cdot S =0$.
\end{lem}
\begin{proof}
For a handle decomposition derived from $f$, 
we can draw a Kirby diagram of $W$ on the disjoint union of two planes, shown in the leftmost of \Cref{fig:excision2}.
Here, the upper and lower cuboids in the figure are the attaching region of the single 1-handle, whose boundaries are identified by the reflection with respect to the green rectangle drawn in the figure. This implies that all three diagrams in \Cref{fig:excision2}
represent diffeomorphic 4-manifolds. Moreover, it is easy to see that  the rightmost diagram represents 
$\left([0,1]\times S^3 \right) \# \left([0,1]\times S^3 \right) \# \left(S^2 \times S^2 \right)$.
(Note that this handle decomposition has a single 3-handle, and the codomain of $W$ is $S^3 \sqcup S^3$.)

Next, let us compute $\chi(S)$ and $S \cdot S$.
Since $S$ is finitely many copies of $[0,1] \times S^1$, the equality $\chi(S)=0$ holds. For the value of $S \cdot S$,
it follows from the existence of the 3-balls $B_i$
that 
\begin{align*}
S \cdot S &= \sum_{\varepsilon_j\in\{+,-\}} 
\sum_{i_j=1,2} ([0,1] \times H^{\varepsilon_1}_{i_1} )\cdot ([0,1] \times H^{\varepsilon_2}_{i_2} )
\\[2mm] 
&= \sum_{\varepsilon_j\in\{+,-\}} 
\sum_{i_j=1,2} \big( \lk(\{0\}\times H^{\varepsilon_1}_{i_1}, \{0\} \times H^{\varepsilon_2}_{i_2})
-
\lk(\{1\}\times H^{\varepsilon_1}_{i_1}, \{1\} \times H^{\varepsilon_2}_{i_2})\big).
\end{align*}
Here,  we see 
\[
\lk(\{0\}\times H^{\varepsilon_1}_{i_1}, \{0\} \times H^{\varepsilon_2}_{i_2})
= 
\begin{cases}
+1 & (i_1 = i_2, \ \varepsilon_1 \neq \varepsilon_2)\\
0 & (\text{otherwise})
\end{cases}
\]
and
\[
\lk(\{1\}\times H^{\varepsilon_1}_{i_1}, \{1\} \times H^{\varepsilon_2}_{i_2})
= 
\begin{cases}
+1 & (i_1 \neq i_2, \ \varepsilon_1 = \varepsilon_2)\\
0 & (\text{otherwise})
\end{cases}.
\]
These imply that $S \cdot S = 0$.
\end{proof}

\begin{figure}[t]
    \centering
    \includegraphics[width=0.5\linewidth]{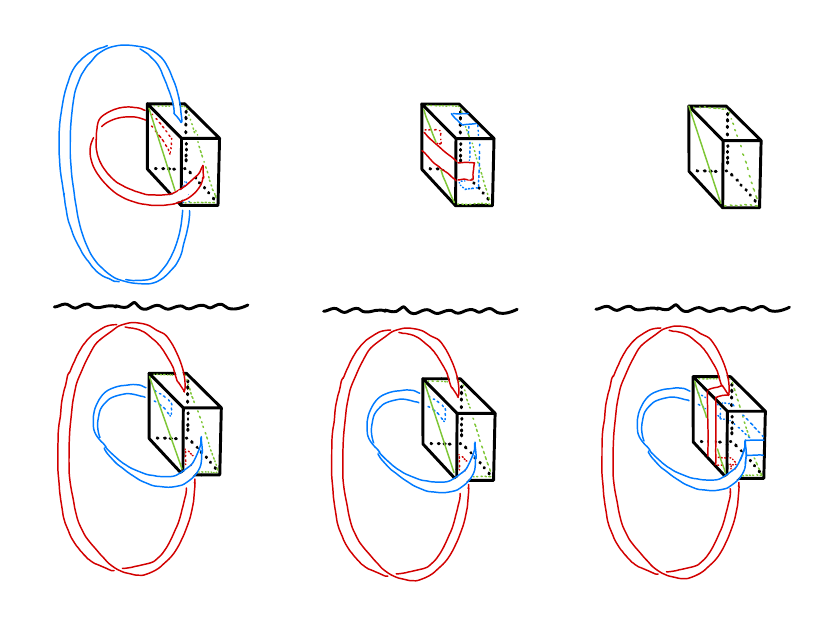}
    \caption{}
    \label{fig:excision2}
\end{figure}

In the latter sections, we fix this Morse function $f$ on $W$.
Note that $W$ naturally contains a link cobordism $S$ from $(K_1 \sqcup H)  \sqcup (K_2 \sqcup H)$ to $(K_1 \sqcup K_2 \sqcup H) \sqcup H$.
We have a product neighborhood 
\begin{align}\label{prod nbd}
 K_1 \times [0,1] \sqcup K_2 \times [0,1] \subset D^3_+ \times [0,1] \sqcup D^3_- \times [0,1 ] \subset W
\end{align}
which is a neighborhood of components of $S$ with $\partial = K_1, K_2$. Here $[0,1] \subset D^3_+ \times [0,1] \sqcup D^3_-$ are the tubles 
depicted by red in \cref{Fig: excision}.  

Note that the Morse function $f$ on $D^3_+ \times [0,1] \sqcup D^3_- \times [0,1 ]$ is just a projection to the second component.

\begin{figure}\label{Fig: excision}
\begin{center}
\includegraphics[width=0.8\textwidth]{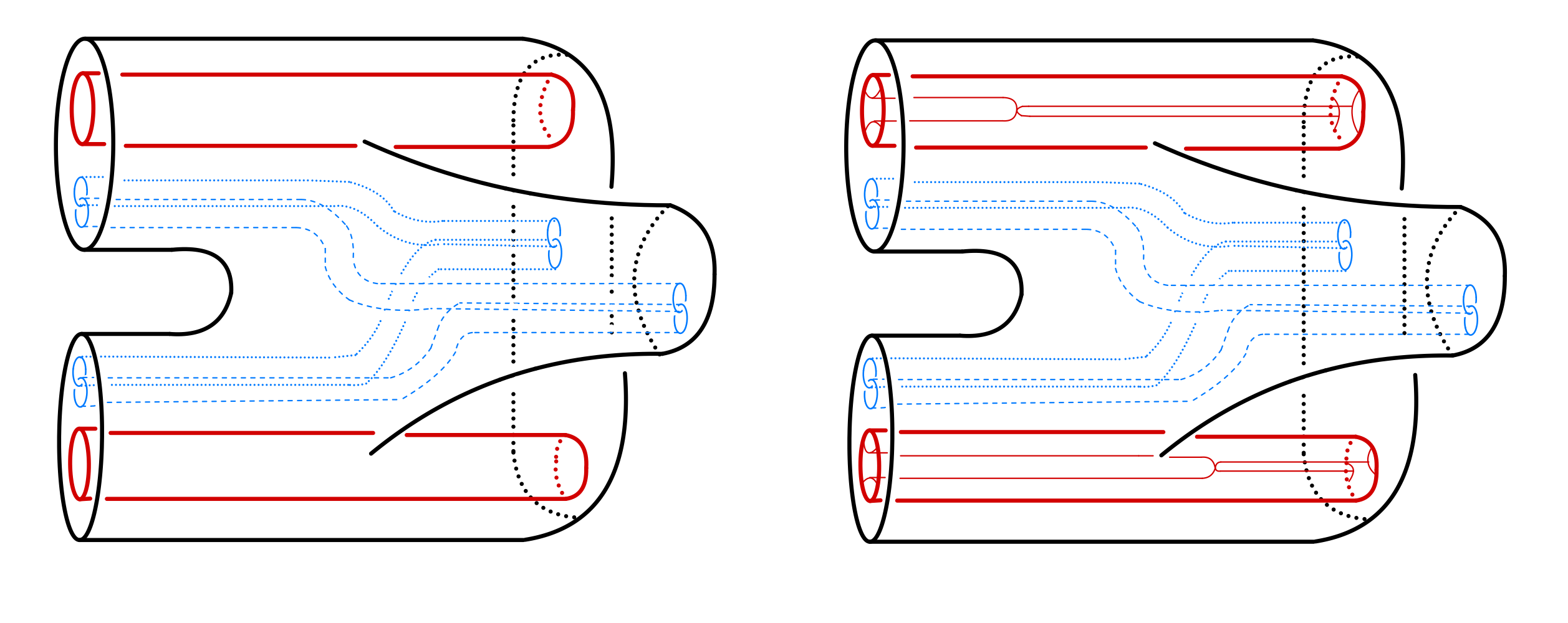}
\end{center}
 \caption{Excision cobordism}
\end{figure}

Next, we will construct a cobordism $(W, S_{uv;w})$ involving a change of resolutions.
We fix an orbifold Riemannian metric $\check{g}_{W}$ on the cobordism $(W, S)$, which has a product form near embedded surfaces $[0, 1]\times (K_1\sqcup K_2)$ inside $W$.
Let $D_1$ (resp. $D_2$) be a link diagram of $K_1$ (resp. $K_2$) equipped with $N_1$ (resp. $N_2$ ) crossings.
For resolutions $u\in \{0, 1\}^{N_1}$, $v\in \{0, 1\}^{N_{2}}$ and $w\in \{0, 1\}^{N_1+N_2}$, we define an excision cobordism between resolved link diagrams:
\[(W, S_{uv;w}): (S^3, (D_1)_{u}\sqcup H_{\omega})\sqcup (S^3, (D_2)_{v}\sqcup H_{\omega})\rightarrow (S^3, (D_1\sqcup D_2)_{w}\sqcup H_{\omega})\sqcup (S^3, H_{\omega}). \]
Let $u\in \Z^{N_{1}}$, $v\in \Z^{N_{2}}$ and $w\in \Z^{N_{1}+N_{2}}$ be resolutions of pseudo-diagrams $D_1$, $D_2$, and $D_1 \sqcup D_2$ so that the resolutions are the unlinks. 
Suppose $uv \geq w$.
Note that a resolution $w$ for the diagram $D_1\sqcup D_2$ can be written $w=(w_1, w_2 )$ using two resolutions $w_1$ and $w_2$ for diagrams $D_1$ and $D_2$ respectively.
We consider the cobordism: 
\[(W, S_{uv;w}): (S^3, (D_1)_{u}\sqcup H_{\omega})\sqcup (S^3, (D_2)_{v}\sqcup H_{\omega})\rightarrow (S^3, (D_1\sqcup D_2)_{w}\sqcup H_{\omega})\sqcup (S^3, H_{\omega}).
\]
obtained by inserting $S_{uw_1}$ and $S_{vw_2 }$ in \eqref{prod nbd}, i.e. 
\begin{align*}
S_{uv;w} \cap \left([0,1] \times D^3_+ \sqcup [0,1] \times D^3_- \right) & = S_{uw_1}\sqcup S_{vw_2 } \subset [0,1] \times D^3_+ \sqcup [0,1] \times D^3_-  \\ 
S_{uv;w} \cap ([0,1] \times D^3_+ \sqcup [0,1] \times D^3_-)^c & = S \cap ([0,1] \times D^3_+ \sqcup [0,1] \times D^3_-)^c. 
\end{align*}
We will simply denote this cobordism by $S_{uv;w}$, and $\bar{S}_{uv;w}$ denotes the non-compact cobordism obtained by attaching a half cylinder on each boundary component. We also extend the Morse function $f$ to the non-compact manifold natural way. We also define $\bar{W}$ in a similar way so that $\bar{S}_{uv;w}\subset \bar{W}$.

For convenience, we introduce the following notations: 
For $u\in \Z^{N_1}$, $v\in\Z^N_{2}$, and $w=(w_{1}, w_{2})\in \Z^{N_{1}}\times \Z^{N_{2}}$, we define 
\begin{eqnarray*}
     |uv-w|_{1}&:=&|u-w_{1}|_{1}+|v-w_{2}|_{1} \\
     |uv-w|_{\infty}&:=&\textrm{max}\{|u-w_{1}|_{\infty}, |v-w_{2}|_{\infty}\}
\end{eqnarray*}

For two link resolutions $u\in \Z^{N_{1}}$ and $v\in \Z^{N_{2}}$, we 
write $uv$ as for the integral lattice point $(u, v)\in \Z^{N_{1}}\times \Z^{N_{2}}$.
Then, $|uv-w|_{1}$ and $|uv-w|_{\infty}$ define norms on the product integral lattice $\Z^{N_{1}}\times \Z^{N_{2}}$, $|uv-w|_{1}$ and $|uv-w|_{\infty}$.
\begin{defn}
    For the cobordism $(W, S_{uv;w})$ constructed as above, we say $(W, S_{uv;w})$ is of {\it type $n$} if 
    \[|uv-w|_{\infty}=n.\]
\end{defn}
We define the space of orbifold metrics $G_{uv;w}$ associated with the excision cobordism $S_{uv;w}$ with resolutions.
For standard cobordism $S_{uw}$ in a cylinder, recall that Kronheimer and Mrowka introduced an associated family of orbifold metrics ${G}_{uw}$ in \cite[Section 3.9 and Section 6]{KM11u} for the case
\begin{itemize}
    \item[(a)] arbitrary pair of resolutions  $u\geq w$ with $|u-w|_{\infty}\leq 1$,
    \item[(b)] pair of resolutions $u>w$ of $|u-w|_{\infty}$ equals to $2$ or $3$, and $|u(c_{i})-w(c_{i})|\leq 1$ for all but one crossing.
\end{itemize}
The explicit construction of these families of metrics depends on the choice of the initial metric, however, we assume that these families of metrics are explicitly fixed for any resolutions.
Recall that we have an embedding $[0,1]\times Y_{\sqcup} \hookrightarrow W$ which naturally extends to 
\[
\R \times Y_{\sqcup } \hookrightarrow \overline{W}. 
\]
Consider that the resolutions satisfy $u_{i}>w_{i}$ or $v_{i}>w_{i}$ for some $i$, and fix a tube 
\[
\mathbb{R}\times B_{i} \subset \R \times Y_{\sqcup }
\]
where $B_{i}$ is a small $3$-ball associated to the crossing $c_{i}$ of $D$ for each $ i$. 
For convenience, assume that $u_i >w_i$ holds.
Any orbifold metric $\check{g}_{uw_{1}}(\tau)\in G_{uw_{1}}$ has the property that it is $\mathbb{R}$-translation invariant over a neighborhood of $\mathbb{R}\times \partial B_{i} \subset \R \times Y_{\sqcup }$.
Moreover, every orbifold metric contained in $G_{uw_{1}}$ are isometric on the neighborhood of 
$\mathbb{R}\times \partial B_{i} \subset \R \times Y_{\sqcup } $.
A similar property holds for any metrics in $G_{vw_{2}}$.
Hence, as an initial orbifold metric on $(\bar{W}, \bar{S}_{uv;w})$, we can fix an metric $\check{g}_{uv;w}(0)$ which is isometric to both of initial metrics $\check{g}_{uw_{1}}(0)\in G_{uw_{1}}$ or $\check{g}_{vw_{2}}(0)\in G_{vw_{2}}$ on each tube $\mathbb{R}\times \partial B_{i}$.
Moreover, we assume that the initial orbifold metric $g_{uv;w}(0)$ is isometric to the standard orbifold metric of the cylinder on each neighborhood of the boundary of $(W, S_{uv;w})$.
From the construction, any pair of metrics $(g_{uw_{1}}, g_{vw_{2}})\in G_{uw_{1}}\times G_{vw_{2}}$ are glued with the initial metric $g_{uv;w}(0)$ at $\mathbb{R}\times \partial B_i$ and forms a new orbifold metric on the excision cobordism.
This construction gives a map 
\[G_{uw_{1}}\times G_{vw_{2}}\rightarrow \operatorname{Met}^{\mathrm{orb}}(\bar{W}, \bar{S}_{uv;w}),\]
where $\operatorname{Met}^{\mathrm{orb}}(\bar{W}, \bar{S}_{uv;w})$ denotes the space of orbifold Riemann metrics which form product metrics on ends. 
We define the family of metrics $G_{uv;w}$ as the image of this map.
The space of metrics $G_{uv;w}$ has dimension $|uv-w|_{1}$ and is noncompact if $|uv-w|_{1}\geq 1$.
We obtain the natural compactification $G_{uv;w}^{+}$ of $G_{uv;w}$ by attaching \textit{broken metrics}.

Recall the case of standard cobordism $S_{uu'}\subset[0,1]\times S^3$  with $u>u'$.
Assume that a sequence of resolutions $\sigma(u, u')=(u_{0}, u_{1}, \cdots, u_{k})$ satisfies \[u=u_{0}> u_{1}>\cdots >u_{k-1}>u_{k}=u'.\]
For $k\geq 2$, We call an element in the product space
\[\breve{G}_{\sigma(u, u')}:=\breve{G}_{uu_{1}}\times \breve{G}_{u_{1}u_{2}}\times\cdots \times \breve{G}_{u_{k-1}u'}.\]
 a \textit{broken metric}.
If $k=1$, we simply define $\breve{G}_{\sigma(u, u')}:=\breve{G}_{uu'}$.



For the case of excision cobordism $S_{uv;w}$ with $uv>w$, we define broken metrics as follows: First, let $u'v'$and $w'$ be resolutions satisfying
$u'v'\geq w', u>u', v>v'$ and $w'>w$.
Then we consider the product space
\begin{eqnarray}\label{cpt face}
    G_{\sigma([u,u'], [v,v']; [w', w])}:=\breve{G}_{\sigma(u,u')}\times \breve{G}_{\sigma(v, v')}\times G_{u'v';w'}\times \breve{G}_{\sigma(w', w)}.
\end{eqnarray}
We extend the definition of (\ref{cpt face}) in the general case of $uv\geq u'v'\geq w'\geq w$.
For example, if $u=u', v>v'$ , and $w'>w$ holds, we define $G_{\sigma([u,u'], [v,v']; [w', w])}$
by dropping the factor $\breve{G}_{\sigma(u, u')}$ from (\ref{cpt face}), and other cases are similar.
In particular, if both $uv=u'v'$ and $w'=w$ hold, $G_{\sigma([u,u], [v,v]; [w, w])}$ is equal to $G_{uv;w}$.
Unless $uv=u'v'$ or $w'=w$ hold, we call an element in $G_{\sigma([u,u'], [v,v']; [w', w])}$ a broken metric in $S_{uv;w}$.
Firstly, we consider compactification of $G_{uv;w}$ when $|uv-w|_{\infty}=1$.
We define a compactification of $G_{uv;w}$ by 
\[G^{+}_{uv;w}:=\bigcup_{\sigma}G_{\sigma([u,u'], [v,v']; [w', w])}\]
equipped with the natural topology.
Here, $\sigma$ runs over every partitions between the resolutions $uv>w$, including $G_{\sigma([u,u], [v,v]; [w, w])}=G_{uv;w}$ itself.
If the broken metrics in $\sigma([u, u'], [v, v']; [w, w'])$ break into $k+1$  metrics,
the local structure near the face $G_{\sigma([u, u'], [v, v']; [w, w'])}$
is given by 
\[(T, \infty]^{k}\times G_{\sigma([u, u'], [v, v']; [w, w'])}\]
where $T>0$ is a large enough positive number. 
We call $G_{\sigma([u, u'], [v, v']; [w, w'])}$ is a codimension $k$ face of $G^{+}_{uv;w}$.
In particular, there are 
three types of codimension $1$ faces in $G_{uv;w}^{+}$:
\[ \breve{G}_{uu'}\times G_{u'v;w},\ \breve{G}_{vv'}\times G_{uv'; w},\ G_{uv;w'}\times \breve{G}_{w'w}.\]

The above broken metrics are sufficient to compactify the space $G_{uv;w}$ when $|uv-w|_\infty = 1$.
However, for the general case, we need to include additional types of broken metrics.
Here, we briefly describe the compactification for type 2 excision cobordisms, following the construction in \cite[Section 7.2]{KM11u}.
In the case of type 2, the compactification $G_{uv;w}^+$ is obtained by adding broken metrics associated with stretching along a separating hypersurface $(S^3, S^1)$, in addition to the standard broken metrics described above.
This additional boundary stratum corresponds to the geometric splitting of the cobordism analogous to the splitting of $(S^4, \mathbb{RP}^2)$ discussed in \cite[Lemma 7.2]{KM11u}.
With this compactification, $G_{uv;w}^+$ has the structure of a manifold with corners.

We consider the moduli space of instantons over the excision cobordism $(W, S_{uv;w})$ parametrized by the family of metrics $G_{uv; w}$.
As in the case of differential $f_{vu}$, for critical points $\alpha$, $\beta$ and $\gamma$ of perturbed Chern--Simons functionals of $D_u$, $D_v$ and $D'_w$ respectively, we consider the family moduli space
\[{M}_{uv; w}(\alpha, \beta; \gamma):=\bigcup_{\check{g}\in {G}_{uv;w}} M_{\check{g}}(W, S_{uv;w}; \alpha, \beta; \gamma), \]
where $M_{\check{g}}(W, S_{uv;w}; \alpha, \beta; \gamma)$ denotes the instanton moduli space for the orbifold metric $\check{g}$ introduced in \eqref{gen moduli}. 
We define the space of broken instantons over the face of broken metrics $G_{\sigma([u, u'], [v, v'];[w', w])}$.
For example, in the case of $u'v'\geq w', u>u', v>v'$ and $w'>w$, 
we define 
\[M_{\sigma([u, u'], [v, v'];[w', w])}(\alpha, \beta;\gamma):=\bigcup_{\alpha', \beta', \gamma'} \breve{M}_{\sigma(u, u')}(\alpha, \alpha')\times \breve{M}_{\sigma(v, v')}(\beta, \beta')\times M_{u'v';w'}(\alpha', \beta'; \gamma')\times\breve{M}_{\sigma(w', w)}(\gamma', \gamma).\]
We extend this construction to the general case of $uv\geq u'v'\geq w'\geq w$. 
For example, if $u=u', v>v'$, and $w'>w$ holds, we define
$M_{\sigma([u, u'], [v, v'];[w', w])}(\alpha, \beta;\gamma)$
by dropping the factor $\breve{M}_{\sigma(u, u')}(\alpha, \alpha')$,
and other cases are similar.
In particular, we have $M_{\sigma([u, u], [v, v];[w, w])}(\alpha, \beta;\gamma):=M_{uv;w}(\alpha, \beta;\gamma)$.


For the type 2 cobordism, note that the moduli spaces corresponding to the boundary strata involving connections on the separating domain $(S^3, S^1)$ are in fact empty.
This is because the relevant $SO(3)$-bundle in the separating domain $(S^3, S^1)$ is no longer admissible in the sense of the non-integral condition \cite[Definition 3.1]{KM11u}. 

We define the compactification of the moduli space

\[M^{+}_{uv;w}(\alpha, \beta;\gamma):=\bigcup_{\sigma}M_{\sigma([u,u'], [v, v'];[w', w])}(\alpha, \beta;\gamma).\]

Let $d$ be an integer, and write ${M}_{uv; w}(\alpha, \beta;\gamma)_{d}$ as the union of component with index $d$.
We have a natural compactification ${M}_{uv; w}^+(\alpha, \beta;\gamma)_{1}$ of ${M}_{uv; w}(\alpha, \beta;\gamma)_{1}$ whose boundaries are listed as follows: 
\begin{itemize}
    \item ${M}_{uv;w'}( \alpha, \beta;\gamma')_{0}\times \breve{M}_{w'w}(\gamma', \gamma)_{0}$, 
    \item $\breve{M}_{uu'}(\alpha, \alpha')_{0}\times {M}_{u'v;w}(\alpha', \beta;\gamma)_{0}$, 
    \item $\breve{M}_{vv'}(\beta, \beta')_{0}\times {M}_{uv';w}(\alpha, \beta';\gamma)_{0}$.   
\end{itemize}
One can also prove ${M}^{+}_{uv;w}(\alpha, \beta; \gamma)_{1}$ and ${M}^{+}_{uv;w}(\alpha, \beta; \gamma)_{0}$ are compact by the standard argument.

Now, we define a $\mathbb{Z}$-module map 
\[
\Psi: \CKh^\sharp (D)  \otimes \CKh^\sharp (D')  \to \CKh^\sharp (D \sqcup D')
\]
 in the following way. 
 Firstly, consider the decomposition of the product chain complex
 $\CKh^\sharp (D)  \otimes \CKh^\sharp (D')$
 into the summand of the form $C^{\sharp}({D}_{u})\otimes C^{\sharp}({D'}_{v})$.

 On each summand $C^{\sharp}({D}_{u})\otimes C^{\sharp}({D'}_{v})$ and a resolution $w$ for the diagram $D_1\sqcup D_{2}$, we define a linear map 
 \[\Psi_{uv;w}: C^{\sharp}({D}_{u})\otimes C^{\sharp}({D'}_{v})\rightarrow C^{\sharp}((D\sqcup D')_{w}).\]

 \begin{defn}
    We define
 \[ \Psi_{uv;w}(\alpha \otimes \beta):=(-1)^{s(u, v; w)}\sum_{\gamma\in \mathfrak{C}_{\pi}((D\sqcup D')_{w})}\#{M}_{uv;w}(\alpha, \beta;\gamma)_{0}\cdot \gamma.\]
 Here $s$ is an integer-valued function given by the formula:
 \[s(u, v; w)=\frac{1}{2}(|uv-w|_{1}-1)|uv-w|_{1}+\sum_{i}u_{i}+\sum_{i}v_{i}-\sum_{i}w_{i}.\]
  \end{defn}
The above construction linearly extends to get the desired map $\Psi_{uv;w}$.

\begin{rem}
    In this paper, we mainly focus on its induced map on $E^1$ page of the spectral sequences. Therefore, we do not need higher components of the moduli spaces. However, in order to compare the sign convention of Kronheimer and Mrowka \cite{KM11u} with ours, we consider corresponding orientations on higher moduli spaces as well. 
\end{rem}

The next statement is an analog of \cite[Lemma 2.1]{KM14} in our situation.
\begin{prop}\label{quantum ex}
Let $K_1$ and $K_2$ be unlinks.
    For a link cobordism 
    \[
    (W, S):(S^3, K_{1})\sqcup (S^3, K_{2})\rightarrow (S^3, K_{1}\sqcup K_{2})
    \]
    with a family of metrics $G$, the induced map on the cube complexes has the order 
    \[\chi(S)+S\cdot S-4\left\lfloor \frac{S\cdot S}{8}\right\rfloor+\mathrm{dim}G\]
    with respect to the filtration by $Q$.
    In particular, if $S\cdot S <8$ then the induced map has order
    \[\chi(S)+S\cdot S+\mathrm{dim}G.\]
\end{prop}

\begin{proof}
We follow a parallel argument demonstrated in \cite[Lemma 2.1]{KM14}.
Let $(W, S_{uv;w})$ be the cobordism of pairs concerning.
We fix a critical point $\beta_{0}=v_{+}\otimes \cdots \otimes v_{+}$ on the end corresponding to $(K_{1})_{u}$ , and similarly, we fix a critical point $\beta_{1}$ for the end $(K_{2})_{v}$.
These critical points have the minimal Morse index among the critical points set at each end.
Instead, we fix a critical point $\beta'=v_{-}\otimes \cdots \otimes v_{-}$ of the maximal Morse index for the end $(K_{1}\sqcup K_{2})_{w}$.
Filling up each end $(K_{1})_{u}$, $(K_{2})_{v}$, and $(K_{1}\sqcup K_{2})_{w}$ by disconnecting disks in the cylinder, we obtain a cobordism $(\bar{W}, \bar{S}_{uv;w})$ connecting copies of $(S^3, H_{\omega})\sqcup (S^{3}, H_{\omega})$ with the opposite orientations.
Let $u_{0}$ be the unique generator over the pair $(S^3,H_{\omega})$.
Then the virtual dimension formula of the parametrized moduli space $M_{G}(\bar{W}, \bar{S}_{uv;w}; u_{0}, u_{0}; u_{0}, u_{0})$ reduces to the index formula for a closed pair \cite{KM93i} and hence we have the formula:
\[
\begin{aligned}
\mathrm{dim}M_{G}(\bar{W}, \bar{S}_{uv;w}; u_{0}, u_{0}; u_{0}, u_{0})_{\kappa} &= 8\kappa -\frac{3}{2} (\chi (\bar{W}) + \sigma (\bar{W}) )  + \chi (\bar{S}_{uv;w}) + \frac{1}{2}\bar{S} \cdot \bar{S}+\mathrm{dim}G \\
&= 8\kappa -\frac{3}{2} (\chi ({W}) + \sigma ({W}) ) 
+ \chi ({S}_{uv;w}) + \frac{1}{2}{S} \cdot {S} \\
&\quad +b_{0}((K_{1})_{u})+b_{0}((K_{2})_{v})+b_{0}((K_{1}\sqcup K_{2})_{w})+\mathrm{dim}G.
\end{aligned}
\]
The second equality follows from the additivity of the Euler characteristic and the signature.

Finally,  we replace generators $\beta_{0}$, $\beta_{1}$, and $\beta'$ by arbitrarily choices.
This operation changes the index formula by the difference of Morse indices;
\[Q(\beta_{1})+Q(\beta_{2})-Q(\beta')-b_{0}((K_{1})_{u})-b_{1}((K_{2})_{v})-b_{0}((K_{1}\sqcup K_{2})_{w}).\]

Hence, we obtain the index formula:
\begin{align*}
\textrm{ind}(D_{A})&= 8 \kappa(A) - \frac{3}{2} (\chi (W) + \sigma (W) )  + \chi ( S) + \frac{1}{2}S \cdot S+  Q(\beta_0) +Q(\beta_{1})-  Q(\beta') + \dim G  . 
\end{align*} 

For the excision cobordism $(W, S)$, we have \[\chi(W)=\sigma(W)=\chi(S)=0\]
from \cref{topological computation of excision} and hence the desired index formula follows.
\end{proof}

\begin{prop}\label{general rel}
For pseudo-diagrams $D_{u}$, $D'_{v}$, and $(D\sqcup D')_{w}$, the following relations hold:
\[\sum_{u>u'}\Psi_{u'v;w}f_{uu'}+\sum_{v>v'}(-1)^{|v-v'|_{1}(n-|u|_{1})}\Psi_{uv';w}f_{vv'}-\sum_{w>w'}f_{w'w}\Psi_{uv;w'}=0,\]
where $n$ is the number of crossings of the pseudo-diagram $D_{u}$.
\end{prop}
\begin{proof}

The proof is given by considering the counting of the oriented boundary points of the compactified one-dimensional moduli space ${M}^{+}_{uv;w}(\alpha, \beta;\gamma)_{1}$.

Now we define maps by
\begin{eqnarray*}
\bar{m}_{uv}(\alpha)&=&(-1)^{-\mathrm{dim}\breve{G}_{uv}}\sum_{\beta}\# \breve{M}_{uv}(\alpha, \beta)_{0}\cdot \beta \\
\bar{m}_{uv;w}(\alpha\otimes \beta)&=&\sum_{\gamma}\# {M}_{uv;w}(\alpha, \beta; \gamma)_{0}\cdot \gamma
\end{eqnarray*}
Since we assume that all links that appeared in the ends of the excision cobordism are pseudo-diagrams, there are only contributions from $M_{\partial G_{uv;w}}(\alpha, \beta;\gamma)_{0}$ to the boundary of 
$\breve{M}^{+}_{uv;w}(\alpha, \beta;\gamma)_{1}$.
Hence, \Cref{face of cpt moduli-1} implies the following:

  \begin{equation}\label{eqn1}
\begin{split}
        &\sum_{u > u'} (-1)^{(|u-u'|_{1}-1)(|u'v-w|_{1}+1)} \bar{m}_{u'v;w} \circ (\bar{m}_{uu'} \otimes 1) \\
        &+ \sum_{v > v'} (-1)^{|v-v'|_{1}(n-|u|_{1})+(|v-v'|_{1}-1)(|uv'-w|_{1}+1)} (\bar{m}_{uv';w} \circ 1) \otimes \bar{m}_{vv'} \\
        &- \sum_{w > w'} (-1)^{|uv-w'|_{1}|w'-w|_{1}+|w'-w|_{1}-1} \bar{m}_{w'w} \circ \bar{m}_{uv;w'} = 0.
\end{split}
\end{equation}
 Note that, we have $f_{uv}=(-1)^{s(u, v)}\bar{m}_{uv}$ and $\Psi_{uv;w}=(-1)^{s(u, v; w)}\bar{m}_{uv;w}$.
We can check that the above relation is equivalent to 
\begin{equation}\label{eqn2}
\begin{split}
        &\sum_{u > u'\geq w_1} (-1)^{s(u, u')+s(u', v;w)} \bar{m}_{u'v;w} \circ (\bar{m}_{uu'} \otimes 1) \\
        &+ \sum_{v > v'\geq w_2} (-1)^{|v-v'|_{1}(n-|u|_{1})+s(v,v')+s(u,v';w)} \bar{m}_{uv';w} \circ (1 \otimes \bar{m}_{vv'} )\\
        &- \sum_{uv\geq w' > w} (-1)^{s(u, v; w')+s(w', w)} \bar{m}_{w'w} \circ \bar{m}_{uv;w'} = 0,
\end{split}
\end{equation}
by a similar manner as in the proof of Proposition \ref{chain map eqn}.
\end{proof}
We define a map $\Psi:CKh^{\sharp}(D)\otimes \CKh^{\sharp}(D')\rightarrow \CKh^{\sharp}(D\sqcup D')$ by \[\Psi:=\sum_{1\geq uv\geq w\geq 0}\Psi_{uv;w}.\]
The following is an important property of the excision map $\Psi$: 
\begin{prop}\label{qfilE1ofpsi}
The map $\Psi$ is $q$-filtered map. 
Also, the map $E^{1}(\Psi)$ is a chain map over $\Z$. 
\end{prop}
\begin{proof}
First, we discuss the effect on the $q$-filtration. 
We can suppose $|uv-w|_{1}\geq 0$.
Recall that the map $\Psi_{uv;w}$ has the order 
\[\chi(S_{uv;w})+S_{uv;w}\cdot S_{uv;w}-4\left\lfloor \frac{S_{uv; w}\cdot S_{uv; w}}{8}\right\rfloor+|uv-w|_{1}\]
with respect to the filtration defined by $Q$ from \cref{quantum ex}. 
Since the cobordism $S_{uv;w}$ is type $0$ or $1$, we can check that $\chi(S_{uv;w})=S_{uv;w}\cdot S_{uv;w}=0$.
In particular, the map $\Psi_{uv;w}$ has the order
\[\ord_{q}(\Psi_{uv;w})\geq|uv-w|_{1}-1 \geq 0 .\]
for the $q$-filtration.
Hence, the map $\Psi$ is a $q$-filtered map.

For resolutions $u$, $v$, and $w$ with $|uv-w|_{1}=1$,   \cref{general rel} implies either 
\[\Psi_{u'v;w}f_{uu'}-f_{w'w}\Psi_{uv;w'}=0\]
or 
\[(-1)^{n-|u|_{1}}\Psi_{uv';w}f_{vv'}-f_{w'w}\Psi_{uv;w'}=0\]
holds,
where $u'v=w$, or $uv'=w$ respectively and $uv=w'$.
Hence, the map induced from $\Psi$ on the $E^1$-term is a chain map. This completes the proof. 
\end{proof}

Let $u\in \Z^{N_{1}}$, $v\in \Z^{N_{2}}$ and $w\in \Z^{N_{1}+N_{2}}$ be resolutions of pseudo-diagrams $D_1$, $D_2$, and $D_1 \sqcup D_2$ so that the resolutions are the unlink. 
Suppose $uv \geq w$.
Note that a resolution $w$ for the diagram $D_1\sqcup D_2$ can be written $w=(w_1, w_2 )$ using two resolutions $w_1$ and $w_2$ for diagrams $D_1$ and $D_2$ respectively.
We consider the cobordism: 
\[(W, S_{uv;w}): (S^3_+, (D_1)_{u}\sqcup H_{\omega})\sqcup (S^3_-, (D_2)_{v}\sqcup H_{\omega})\rightarrow (S^3, (D_1\sqcup D_2)_{w}\sqcup H_{\omega})\sqcup (S^3, H_{\omega}).
\]
Next, we take a link cobordism 
\[
T : D \to D_1, T \subset [0,1]\times S^3_+  
\]
such that $D$ and $D_1$ are the same diagram near the fixed crossings of the pseudo diagrams and the cobordism $T$ does not touch neighborhoods of the crossings and a small neighborhood of $\infty \in S^3$. 
Take small neighborhoods $\{B_i\}_{i \in I}$ of the crossings and a small neighborhood $B_\infty$ of $\infty \in S^3$. 
In other words, we suppose $T$ is product outside of  $[0,1]\times \overset{\circ}{S} \subset [0,1]\times S^3_+ $, where $\overset{\circ}{S}$ is given as 
\[
\overset{\circ}{S} := S^3 \setminus \left(\bigcup_{i \in I} \overset{\circ}{B}_i \cup \overset{\circ}{B}_\infty \right). 
\]
Then, we define 
\begin{align}\label{Tbordism}
(W, T_{uv;w} ) 
:= ([0,1] \times S^3_+, T \sqcup [0,1]\times H  ) \circ  (W, S_{uv;w}). 
\end{align}
Here, we take an identification $W \cong [0,1] \times S^3_+ \circ W$.
We will simply denote this cobordism by $T_{uv;w}$, and $\bar{T}_{uv;w}$ denotes the non-compact cobordism obtained by attaching a half cylinder on each boundary component.
From the assumption on the cobordism $T$, we see 
\[
T_{uv;w} \subset [0,1]\times Y_{\sqcup } \subset W \text{ and } \bar{T}_{uv;w} \subset \R\times Y_{\sqcup } \subset \bar{W}. 
\]
Moreover, one can extend the embedding $[0,1]\times \overset{\circ}{S} \hookrightarrow [0,1]\times S^3_+$ to 
\[
\iota: \mathbb{R} \times  \overset{\circ}{S} \hookrightarrow \R \times Y_{\sqcup } \subset \bar{W}. 
\]
We first suppose $|uv-w|_\infty \leq 1$. For such a tuple $(u,v,w)$, we introduce a family of orbifold metrics $G^T_{uv; w}$. The space is characterized as follows: 
\[
G^T_{uv; w} =  \R \times G_{uv; w}, 
\]
where this additional parameter $\R$ comes from implanting a fixed Riemann metric on $\R\times \overset{\circ}{S}$ translated by the $\R$-action. Here, we are assuming that the orbifold metrics near the neighborhood of $\R\times \partial (\overset{\circ}{S})$ are $\R$-invariant. 

As in the case of $G_{uv;w}$, we can define a natural compactification $G^{T, +}_{uv;w}$ of $G^{ +}_{uv;w}$. Since it is analogous to $G^+_{uv;w}$, we do not write the full description of it. Instead, we describe the codimension one face of $G^{T, +}_{uv;w}$: 
\[
\begin{array}{c}
    G_{uu'}\times G^{T}_{u'v;w}, \quad G_{vv'}\times G^{T}_{uv'; w}, \\
    G^{T}_{uv;w'}\times G_{w'w}, \quad G^T_{uu'}\times G_{u'v;w}, \\
    G_{uv;w'}\times G^T_{w'w}
\end{array}
\]

where the notations ${G}^T_{uu'}$ and ${G}^T_{w'w}$ denote the families of orbifold metrics for $(\R \times S^3_+, T\circ S_{uu'})$ and $(\R \times S^3, T\circ S_{w'w} )$ again described as follows: 
\begin{align*}
{G}^T_{uu'} &=\R \times  {G}_{uu'}  \\
{G}^T_{w'w}& = \R \times  {G}_{w'w}. 
\end{align*}
Here, the additional $\R$-parameter comes from implanting translated orbifold metrics on a nighborhood of $\bar{T\circ S_{uu'} }$ and $\bar{T\circ S_{w'w} }$ in $\R \times S^3_+$ and $\R \times S^3$ respectively.
Corresponding these boundaries, we have several components of codimension one face in the parametrized instanton moduli spaces. 
One can associate a compatible orientation on $G^{T}_{uv;w}$ with the fixed orientations on $G_{uv}$ given in Kronheimer--Mrowka's argument. See \cref{ori app} for its details. 
We write the associated compactified moduli space ${M}^{+}_{uv;w}(W, T; \alpha, \beta; \gamma) $ as certain union of ${M}^{+}_{uv;w}(W, T; \alpha, \beta; \gamma; \check{g})$ parametrized by $\check{g} \in {G}^{T, +} _{uv;w}$ with suitable topology as in the case of ${M}^{+}_{uv;w}(W, S; \alpha, \beta; \gamma)$.  
Here, $(W, T)$ is a cobordism of pairs arising from the excision of disjoint link diagrams.
Now, with such orientations, we define a $\Z$-module map 
\[
\Psi^T:  \CKh^\sharp(D)\otimes \CKh^\sharp(D') \to  \CKh^\sharp (D \sqcup D') 
\]
by 
\[
\Psi^T :=\sum_{1\geq uv\geq w\geq 0}\Psi^T_{uv;w}, 
\]
where $\Psi^T_{uv;w}$ is defined as 
\begin{eqnarray*}
\Psi^T_{uv;w}(\alpha \otimes \beta)&:=&(-1)^{s^{T}(u, v; w)}\sum_{\gamma\in \mathfrak{C}_{\pi}((D\sqcup D' )_{w})}\#{M}_{uv;w}(W, T; \alpha, \beta;\gamma)_{0}\cdot \gamma,\\
s^{T}(u, v;w)&:=&\frac{1}{2}|uv-w|_{1}(|uv-w|_{1}+1)+\sum_{i}u_{i}+\sum_{i}v_{i}-\sum_{i}w_{i}
\end{eqnarray*}

The proof of the following proposition is essentially the same as that of \cref{general rel}.

\begin{prop}\label{excision_T}
   Suppose $D $ and $D'$ are pseudo-diagrams. Then, we have 
   \[
\sum_{u>u'}( \Psi^T_{u'v;w}f_{uu'}+\epsilon_{1}\Psi_{u'v;w}f_{uu'}^T ) +\sum_{v>v'} (-1)^{|v-v'|_{1}(n-|u|_{1})} \Psi^T_{uv';w}f'_{vv'} -\sum_{w>w'} (f''_{w'w}\Psi^T_{uv;w'} + \epsilon_{2}(f''_{w'w})^T\Psi_{uv;w'})=0, 
   \]
   where $f_{uu'}^T $, $(f'_{vv'})^T$ and $(f''_{w'w})^T$ are the components of Kronheimer--Mrowka's cobordism map with respect to $T$ introduced in \cite{KM14} and 
   \begin{eqnarray*}
       \epsilon_{1}&=&\epsilon_{1}(u, u', v; w)\\
       \epsilon_{2}&=&\epsilon_{2}(u, v;w', w)
   \end{eqnarray*}
   are $\pm1$-valued functions depending on the displayed resolutions. 
   In particular, $\epsilon_{1}\equiv_{(2)}\epsilon_{2}$ for $|uv-w|_{1}=0$.
   \qed
\end{prop}


%% file: 5.tex
\section{Proof of key lemmas}\label{Disjoint formula}

Now, we shall prove three key lemmas \cref{excision lemma}, \cref{disjoint lemma}, and \cref{compatibility with Khovanov}  from instanton theories in order to complete the proof of \cref{thm:main}.  
We first prove \cref{excision lemma}, which is restared here: 
\begin{prop}\label{exci e1}
Suppose $D_1$ and $D_2$ are pseudo diagrams. 
We have the commutative diagrams
\[
    \begin{split}
    \xymatrix@C=44pt{
        E^1(\CKh^\sharp(D_1))\otimes E^1(\CKh^\sharp(D_2))
        \ar[r]^-{E^1(\Psi)}
        \ar[d]_-{\gamma \otimes \gamma}
        & E^1(\CKh^\sharp(D_1 \sqcup D_2))
        \ar[d]^-{\gamma}\\
        \CKh(D_1^*) \otimes \CKh(D_2^*)
        \ar[r]^-{\cong}
        & \CKh(D_1^* \sqcup D_2^*)
    }
    \end{split}
\]
and
\[
    \begin{split}
    \xymatrix@C=44pt{
        qE^0(\CKh^\sharp(D_1))\otimes qE^0(\CKh^\sharp(D_2))
        \ar[r]^-{qE^0(\Psi)}
        \ar[d]_-{\gamma \otimes \gamma}
        & qE^0(\CKh^\sharp(D_1 \sqcup D_2))
        \ar[d]^-{\gamma}\\
        \CKh(D_1^*) \otimes \CKh(D_2^*)
        \ar[r]^-{\cong}
        & \CKh(D_1^* \sqcup D_2^*), 
    }
    \end{split}
\]
where $\gamma$ is a chain map introduced in \cref{prop:equal_to_Kh}. 
\end{prop}

\begin{rem}
    The signing convention for tensor products of $\CKh$ is chosen so that the above two diagrams commute. Recall from \Cref{prop:equal_to_Kh} that we take sign assignments $s_1, s_2$ and $s_{12}$ for $\CKh(D^*_1), \CKh(D^*_2)$ and $\CKh(D^*_1 \sqcup D^*_2)$ respectively, so that the vertical $\gamma$'s are identities. The differential $d^\otimes$ on the tensor product $\CKh(D^*_1) \otimes \CKh(D^*_2)$ is defined as 
    \[
        d^\otimes(x \otimes y) = d_1 x \otimes y + (-1)^{|u|_1} x \otimes d_2 y
    \]
    for enhanced states $x \in V((D^*_1)_u)$ and $y \in V((D^*_2)_v)$. Note that 
    \[
        \gr_h(x) = |u|_1 - n_-(D^*_1) 
    \]
    so the above convention coincides with the standard one after modifying the differential $d_2$ of $\CKh(D^*_2)$ by multiplying $(-1)^{n_-(D^*_1)}$. 
\end{rem}

\begin{proof}

Let $K_1$ and $K_2$ be links in $\R^3_+$ and $\R^3_-$ respectively, where $\R^3_\pm$ is just a copy of $\R^3$. Let us denote by $S^3_+$ and $S^3_-$ their compactifications and regard $D_1$ and $D_2$ as diagrams of $K_1$ and $K_2$.  
With these labels, we regard the excision cobordism as 
\[
(W, S) :  (S^3_+, (D_1)_u \sqcup H )  \sqcup (S^3_+, (D_2)_v \sqcup H ) \to (S^3_+, H) \sqcup (S^3_-, (D_1\sqcup D_2)_w\sqcup H ),  
\]
where $H$ is the Hopf link for resolutions $u, v$ and $w$.

We first focus on the statement about homological grading. 
    For this first map, we need to calculate the 
    map 
    \[
    E^1 ( \Psi) = f_{uv,(u,v)}
    \]
   which is just the usual cobordism map 
    \[
    I ^\#_{(W, S)} : I^\sharp ((D_1)_u)  \otimes I^\sharp ((D_2)_v) \to I^\sharp((D_1)_u  \sqcup (D_2)_v). 
    \]
    In this case, the link cobordism $S_{uv;w}$ in $D^3_+\times I \sqcup D^3_-\times I$ is just the product cobordism from unlink $(D_1)_u \sqcup (D_2)_v$ in $D^3_+ \sqcup D^3_-$ to itself. 

   We consider two kinds of cobordism maps 
   \begin{align*}
      & I^\sharp([0,1] \times S^3, D^2): \Z \to I^\sharp_*(U_1)  \\ 
     &  I^\sharp([0,1]\times S^3, D^2, \cdot): \Z \to I^\sharp_*(U_1). 
   \end{align*}
   Then, if we write by $u_+, u_-$ the homogenous generators of $I^\sharp_*(U_1)$ whose degrees are $0$ and $2$, then 
   \begin{align}\label{gen rep}
   I^\sharp(D^2)(1) = u_+ \text{ and } 
   I^\sharp(D^2, \cdot) (1) = 2u_-
   \end{align}
   follow from \cite{KM11u}. 
    We first consider the trivial case $D_1=D_2 = \emptyset$. In this case, the excision cobordism $(W, S)$ has annuli connecting the Hopf links. Note that $I^\sharp(\emptyset) \cong \Z$ and the excision cobordism gives an isomorphism 
    \[
    I^\sharp (W, S) : \Z \cong I^\sharp(\emptyset) \otimes I^\sharp(\emptyset) \to I^\sharp(\emptyset) \cong \Z. 
    \]
    Therefore, the statement follows in this case.

    Next, we consider the general case. Suppose the numbers of components of $r((D_1)_u) = n_1 $ and $r((D_2)_v) = n_2$ for fixed resolutions $u$ and $v$ of $ D_1$ and $D_2$ respectively. 
We define the dotted cobordism 
\[
D(N) :=  \left(\bigcup_{i\in \{1, \cdots, n_1 \} } D^2_i \subset [0,1]\times S^3_{+}\right) \sqcup   \left(\bigcup_{j\in \{n_1+ 1, \cdots, n_1 + n_2 \} } D^2_j \subset [0,1]\times S^3_{-}\right)   
\]
for a given subset $N \subset \{1, \cdots, n_1+ n_2\}$ equipped with a dot on the components corresponding to $N  \subset \{1, \cdots, n_1+ n_2\}$, which is bounded by $(D_1)_u \sqcup (D_2)_v \subset (S^3_+ \sqcup S^3_-) \times \{1\}$. 
Let us also consider 
\[
D^*(N) := \left(\bigcup_{i\in \{1, \cdots, n_1+ n_2 \} } D^2_i \subset [0,1]\times S^3_- \right)
\]
again equipped with a dot on the components corresponding to $N  \subset \{1, \cdots, n_1+ n_2\}$, which is bounded by $(D_1)_u \sqcup (D_2)_v \subset  S^3_- \times \{1\}$. 
Then, we cap off the unknots by $D^2$:  
 \[
(W', S'):= D(N) \circ  (W, S) 
 \]
 which gives a cobordism from $(S^3_+, H)\sqcup (S^3_-,H)$ to $(S^3_+, H)\sqcup (S^3_-,(D_1)_u \sqcup (D_2)_v \sqcup H)$.
 From \eqref{gen rep}, we see that each generator in $I^\sharp( (D_1 )_u) \otimes I^\sharp ((D_2)_v )  \subset  \CKh^\sharp (D_1) \otimes \CKh^\sharp (D_2)$ over $\mathbb{Q}$ can be obtained 
\[
I^\sharp (D(N))  (1) \in I^\sharp( (D_1 )_u)\otimes I^\sharp ( (D_2)_v)  \cong \overbrace{H_*(S^2; \Z)\otimes \cdots \otimes H_*(S^2; \Z)}^{r(D_u)+ r(D_v) }
\]
for a choice of $N \subset \{ 1, \cdots, n_1 + n_2\}$. Again from \eqref{gen rep}, each generator in $I^\sharp( (D_1)_u \sqcup (D_2)_v)   \subset  \CKh^\sharp (D_1\sqcup D_2)$ can be obtained as 
\[
I^\sharp (D^*(N))  (1) \in I^\sharp((D_1)_u \sqcup (D_2)_v)  \cong \overbrace{H_*(S^2; \Z)\otimes \cdots \otimes H_*(S^2; \Z)}^{r(D_u)+ r(D_v) }
\]
for a suitable choice of $N \subset \{ 1, \cdots, n_1 + n_2\}$.
So, it is sufficient to see
\[
I^\sharp ((W, S))  \circ I^\sharp (D(N))(1) = I^\sharp (D^*(N))(1), 
\] 
which follows easily from isotopy invariance of $I^\sharp$ since $(W', S')$ and $(W, S^*) \circ D^*(N)$ are smoothly isotopic rel boundary under the identification $W \cong W'$, where $S^*$ is the components of $S_{uv,w}$ which connects the Hopf links. This ensures the $E^1 (\Psi )$ coincides with the disjoint sum formula in Khovanov theory described in \cref{prop:ckh-isoms}.
This completes to show that the first diagram is commutative.
For the second case of quantum grading, since $\Psi$ is quantum grading $\geq 0$, one can see 
\[
qE^0(\Psi) = f_{uv, (u,v)}. 
\]
Hence, the commutativity of the second diagram follows from the argument parallel to the first claim.
\end{proof}

\subsection{Proof of \cref{disjoint lemma}}
Next, we shall give a proof of \cref{disjoint lemma}, which reduces to prove the following two lemmas: 
\begin{lem} \label{012handle}
Let $D_1$, $D_2$, and $D_1'$ be pseudo-diagrams, and $S \colon D_1 \to D'_1$ be a diagrammatic $0, 1, 2$-handle  attaching or an isotopy supported on the complement of all the crossings. 
Then, for the disjoint union 
\[(S \sqcup [0,1]\times D_2) \colon D_1 \sqcup D_2 \to D'_1 \sqcup D_2
\]
the diagram
\[
    \xymatrix@C=40pt{
        E^{1}(\CKh^\sharp(D_1)) \otimes E^{1}( \CKh^\sharp(D_2))
        \ar[r]^-{E^{1}(\phi^{KM}_S )\otimes \id}
        \ar[d]_-{E^{1}(\Psi)}
        & E^1(\CKh^\sharp(D'_1)) \otimes E^{1}( \CKh^\sharp(D_2))
        \ar[d]^-{E^{1}(\Psi')}\\
       E^{1}( \CKh^\sharp(D_1 \sqcup D_2))
        \ar[r]^-{E^{1}(\phi^{KM}_{S \sqcup ([0,1]\times D_2)})}
        & E^1(\CKh^\sharp(D'_1 \sqcup D_2))
    }
\]
is commutative. 
\end{lem}
\begin{proof}
 First, we consider the case when $\phi^{KM}_{S}$ is a map associated to the 1-handle attachment.
In this situation, let us add crossing $c_{*}$ on the diagram $D_{1}$, and regard the $1$-handle attachment operation as a cobordism map induced from the change of  $1$-resolution to $0$-resolution for the crossing $c_{*}$.
More specifically, the surface $T$ inside the cobordism is chosen as a standard link cobordism $S_{uw}$ such that $|u-w_{1}|_{\infty}=1$, $u(c_{*})=1$, and $ w_{1}(c_{*})=0 $ for a specified crossing $c_{*}$.
For $i\in \{0, 1\}$ and a diagram $D$ inside $S^3$ with $N+1$ crossings,  we 
recall the notation:
\[CKh^{\sharp}_{i}(D):=\bigoplus_{\substack{u\in \{0, 1\}^{N+1}\\ u(c_{*})=i}}C^{\sharp}(D_{u})\]

Then, we will prove the commutativity of the diagram:
\[
    \xymatrix@C=40pt{
        E^{1}(\CKh^\sharp_{1}(D_1)) \otimes E^{1}( \CKh^\sharp(D_2))
        \ar[r]^-{E^{1}(\phi^{KM}_S )\otimes \id}
        \ar[d]_-{E^{1}(\Psi)}
        & E^1(\CKh^\sharp_{0}(D_1)) \otimes E^{1}( \CKh^\sharp(D_2))
        \ar[d]^-{E^{1}(\Psi')}\\
       E^{1}( \CKh^\sharp_{1}(D_1 \sqcup D_2))
        \ar[r]^-{E^{1}(\phi^{KM}_{S \sqcup ([0,1]\times D_2)})}
        & E^1(\CKh^\sharp_{0}(D'_1 \sqcup D_2)).
    }
\]

 To see this, consider the $1$-dimensional compactified moduli space $\breve{M}^{+}_{uv;w}(\alpha, \beta;\gamma)_1$ and count the oriented boundary.
 Here $u(c)v(c)=w(c)$ except for $c=c_{*}$.
The boundary consists of four types of faces as in Figure \eqref{fig:fourends}. 
In particular, two of these faces are empty since we assume that the links at the ends are given by pseudo-diagrams.
Hence, the counting of the reminded boundary faces gives
\[\Psi_{u'v;w}f_{uu'}-f_{w'w}\Psi_{uv;w'}=0\]
where $uv=w'$ and $u'v=w$.
The statement for $1$-handles follows from this relation.
 
Next, we discuss the effect of isotopy.
 Let $T \colon D_1 \to D'_1$ be a trace of isotopy disjoint from crossings. 
Then, for the disjoint union 
\[(T \sqcup [0,1]\times D_2) \colon D_1 \sqcup D_2 \to D'_1 \sqcup D_2
\]
we claim that the diagram
\[
    \xymatrix@C=40pt{
        E^{1}(\CKh^\sharp(D_1)) \otimes E^{1}(\CKh^\sharp(D_2))
        \ar[r]^-{E^{1}(\phi^{KM}_T) \otimes \id}
        \ar[d]_-{E^{1}(\Psi)}
        & E^{1}(\CKh^\sharp(D'_1)) \otimes E^{1}(\CKh^\sharp(D_2))
        \ar[d]^-{E^{1}(\Psi')}\\
        E^{1}(\CKh^\sharp(D_1 \sqcup D_2))
        \ar[r]^-{E^{1}(\phi^{KM}_{T \sqcup ([0,1]\times D_2)})}
        & E^{1}(\CKh^\sharp(D'_1 \sqcup D_2))
    }
\]
is commutative.
This follows from 
the formula in \Cref{excision_T} to the case $|uv-w|_{1}=0$.

The proofs for $0$ and $2$-handles attaching are similar.
\end{proof}

\begin{lem}\label{Add-Drop}
Let $D'_1$ be the pseudo diagram obtained from the dropping crossings of another pseudo diagram $D_1$. We denote by 
\[
\Phi^\sharp : \CKh^\sharp(D_1) \to \CKh^\sharp(D'_1) \text{ and } \Psi^\sharp : \CKh^\sharp(D'_1)  \to \CKh^\sharp(D_1)
\]
Kronheimer--Mrowka's dropping and adding maps, respectively. 
Then the diagrams
\begin{align*}
    \xymatrix@C=40pt{
        E^{1}(\CKh^\sharp(D_1)) \otimes E^{1}( \CKh^\sharp(D_2))
        \ar[r]^-{E^{1}(\Phi^\sharp )\otimes \id}
        \ar[d]_-{E^{1}(\Psi)}
        & E^{1}(\CKh^\sharp(D'_1)) \otimes E^{1}(\CKh^\sharp(D_2))
        \ar[d]^-{\Psi}\\
        E^{1}(\CKh^\sharp(D_1 \sqcup D_2))
        \ar[r]^-{E^{1}(\Phi^{\sharp})}
        & E^{1}(\CKh^\sharp(D'_1 \sqcup D_2))
    } \text{ and } \\
    \xymatrix@C=40pt{
        E^{1}(\CKh^\sharp(D_1')) \otimes E^{1}( \CKh^\sharp(D_2))
        \ar[r]^-{E^{1}(\Psi^\sharp )\otimes \id}
        \ar[d]_-{E^{1}(\Psi)}
        & E^{1}(\CKh^\sharp(D_1)) \otimes E^{1}(\CKh^\sharp(D_2))
        \ar[d]^-{\Psi}\\
        E^{1}(\CKh^\sharp(D'_1 \sqcup D_2))
        \ar[r]^-{E^{1}(\Psi^{\sharp})}
        & E^{1}(\CKh^\sharp(D_1 \sqcup D_2))
    }
\end{align*}
are commutative up to chain homotopy of order $(\ord_h, \ord_q) \geq (0,0)$.

\end{lem}

\begin{proof}

We first consider an operation of dropping a crossing.
Recall that the dropping map 
\[\Phi^{\sharp}:CKh_{1}^\sharp(D_1)\oplus \CKh_{0}^\sharp(D_1)\rightarrow \CKh_{-1}^\sharp(D_1)\] is defined by
$\Phi=\left[{F}_{1,-1}, {F}_{0,-1}\right]$
where ${F}_{ij}$ are maps of cube complexes induced from a standard cobordism $S_{uw}$ such that $u(c)=i$ and $w(c)=j$.
For simplicity, let us write: 

\[
\begin{aligned}
C^{\otimes}_{*} &:= \CKh_{*}^\sharp(D_1) \otimes \CKh^\sharp(D_2), \\
C^{\text{disj}}_{*} &:= \CKh_{*}^\sharp(D_1 \sqcup D_2).
\end{aligned}
\]
for $*=1,0, -1$.
 Let $\mathrm{Id}_{i}$ be the identity map on the cube complex $\CKh^{\sharp}(D_{i})$ where $i=1, 2$ and we write
 $\Phi^{\otimes}=\left[{F}^{\otimes}_{1,-1}, {F}^{\otimes}_{0,-1}\right]$
 where ${F}^{\otimes}_{ij}:={F}_{ij}\otimes \mathrm{Id}_{2}$.
 Furthermore, we write the associated dropping map 
 \[
 \Phi^{\textrm{disj}}: C^{\text{disj}}_{1}\oplus C^{\text{disj}}_{0}\rightarrow C^{\text{disj}}_{-1}
 \]
 by
 $\Phi^{\textrm{disj}}=\left[{F}^{\textrm{disj}}_{1,-1}, {F}^{\textrm{disj}}_{0,-1}\right]$
 where ${F}^{\textrm{disj}}_{ij}$ are maps on cube complexes induced from the standard cobordism
 \[S_{uw}\sqcup ([0,1]\times D_{2}) \subset [0, 1]\times S^3\]
where $u(c)=i$ and $w(c)=j$. 
The differential on the chain complex $C^{\otimes}_{1}\oplus C^{\otimes}_{0}$
has a form 

\[
 \begin{pmatrix}
    F^{\otimes}_{1,1}+\epsilon  d^{\otimes}_{2}& 0\\
    F^{\otimes}_{1,0}&F^{\otimes}_{0,0}+\epsilon d^{\otimes}_{2}
\end{pmatrix}
\]
where $d^{\otimes}_{2}:=\id_{1}\otimes d^\sharp$, and $\epsilon$ is the sign of graded tensor product with respect to the $\Z/4$-grading on cube complexes.
Note that the map
\[\Psi: C^{\otimes}_{1}\oplus C^{\otimes}_{0}\rightarrow C^{\textrm{disj}}_{1}\oplus C^{\textrm{disj}}_{0}\]
induced from the excision cobordism has the form 
\[
 \Psi=
 \begin{pmatrix}
    \Psi_{1,1}& 0\\
    \Psi_{1,0}& \Psi_{0,0}
\end{pmatrix}
\]
with respect to the decomposition.
For the convenience, let us write $\Psi_{-1,-1}$ for the corresponding map $C^{\otimes}_{-1}\rightarrow C^{\textrm{disj}}_{-1}$.
We will prove that the following diagram commutes up to chain homotopy.
\[
    \xymatrix@C=40pt{
        E^{1}(C^{\otimes}_{1}\oplus C^{\otimes}_{0})
        \ar[r]^-{E^{1}(\Phi^{\otimes})}
        \ar[d]_-{E^{1}(\Psi)}
        & E^{1}(C^{\otimes}_{-1})
        \ar[d]^-{E^{1}(\Psi_{-1,-1})}\\
        E^{1}(C^{\text{disj}}_{1}\oplus C^{\text{disj}}_{0})
        \ar[r]^-{E^{1}(\Phi^{\textrm{disj}})}
        & E^{1}(C^{\text{disj}}_{-1})
    }
\]

To introduce a desired chain homotopy map, we define a filtered map $H: C^{\otimes}_{1}\oplus C^{\otimes}_{0}\rightarrow C^{\textrm{disj}}_{-1}$.
Then, we define a chain homotopy \[H=\left[ \Psi_{1,-1}, \Psi_{0,-1}\right]: C^{\otimes}_{1}\oplus C^{\otimes}_{0}\rightarrow C^{\textrm{disj}}_{-1}\] by the following way:
\[\Psi_{*,-1}:=\sum_{\substack{\mathbf{1}\geq uv\geq w\geq \mathbf{0},\\u(c_{*})=*, w(c_{*})=-1}}\Psi_{uv;w}.\]
 The proof of the statement \Cref{Add-Drop} reduces to showing the pair of equalities: 
\begin{equation}
\begin{split}
    E^{1}(\Psi_{-1,-1}) \circ E^{1}({F}^{\otimes}_{1,-1}) 
    - E^{1}({F}^{\textrm{disj}}_{1,-1}) \circ E^{1}(\Psi_{1,1}) 
    - E^{1}({F}^{\textrm{disj}}_{0,-1}) \circ E^{1}(\Psi_{1,0})\\ 
    = E^{1}(\Psi_{0,-1}) \circ E^{1}({F}_{1,0})
\end{split}
\label{htpy rel (1)}
\end{equation}
\begin{equation}
\begin{split}
    E^{1}(\Psi_{-1,-1}) \circ E^{1}({F}^{\otimes}_{0,-1}) 
    - E^{1}({F}^{\textrm{disj}}_{0,-1}) \circ E^{1}(\Psi_{0,0}) = 0
\end{split}
\label{htpy rel (2)}
\end{equation}

 The equation \eqref{htpy rel (1)} follows from Proposition \ref{general rel} to the case $u(c_{*})=1$, $w(c_{*})=-1$, and $u(c)v(c)=w(c)$ for $c\neq c_{*}$.
The proof of the chain homotopy relation (\ref{htpy rel (2)}) is similar to that of Proposition \ref{012handle}.

\begin{figure}
    \centering
    \includegraphics[width=1.0\linewidth]{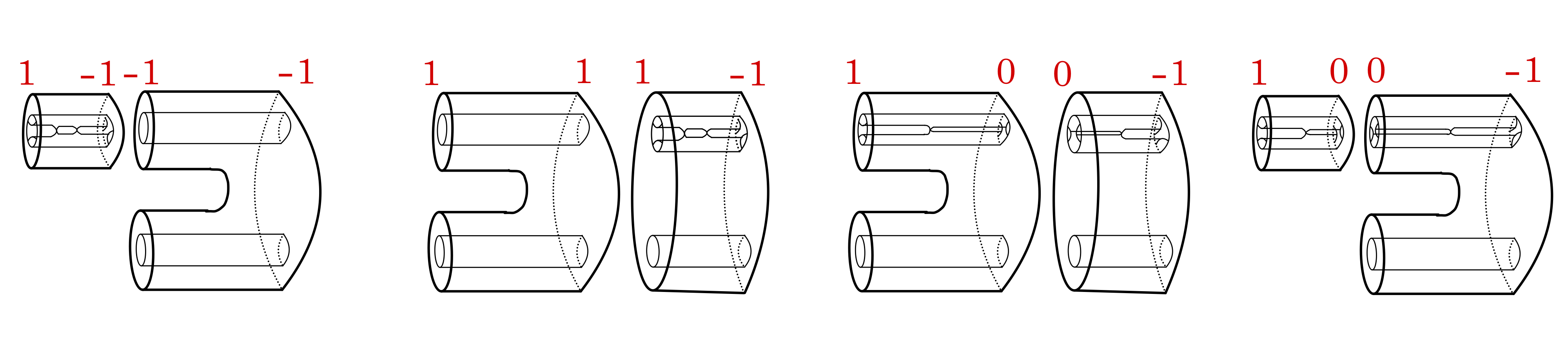}
    \caption{Ends of the moduli spaces}
    \label{fig:fourends}
\end{figure}

The proof for adding crossings is similar. 
\end{proof}

Now, we can give a proof of \cref{disjoint lemma}. 

\begin{cor}[\cref{disjoint lemma}] 
Let $S$ be a 
fundamental cobordism $S:D_{1}\rightarrow D'_{1}$ which is given by $RI, RI^{-1}, RII, RII^{-1} , RIII, h^0, h^1, h^2$, or a \text{planar isotopy}.
Let $D_2$ be any link diagram.
Consider the associated maps
\[
\phi^{KM}_{S}:CKh^{\sharp}(D_{1})\rightarrow \CKh^{\sharp}(D'_{1})
\]
and
\[\phi^{KM}_{S\sqcup [0,1]\times D_{2}}:CKh^{\sharp}(D_{1}\sqcup D_{2})\rightarrow \CKh^{\sharp}(D'_{1}\sqcup D_{2}).\]
Then we have a diagram 
\[
    \xymatrix@C=40pt{
        E^{1}(\CKh^\sharp(D_1)) \otimes E^{1}(\CKh^\sharp(D_2))
        \ar[r]^-{E^{1}(\phi^{KM}_{S})\otimes \id}
        \ar[d]_-{E^{1}(\Psi)}
        & E^{1}(\CKh^\sharp(D'_1)) \otimes E^{1}(\CKh^\sharp(D_2))
        \ar[d]^-{E^{1}(\Psi')}\\
        E^{1}(\CKh^\sharp(D_1 \sqcup D_2))
        \ar[r]^-{E^{1}(\phi^{KM}_{S \sqcup ([0,1]\times D_2)})}
        & E^{1}(\CKh^\sharp(D'_1 \sqcup D_2))
    }
\]
is commutative up to chain homotopy of order $(\ord_h, \ord_q) \geq (0,0)$.
\end{cor}
\begin{proof}
Let 
\[\phi:CKh^\sharp(D_1)\rightarrow \CKh^\sharp(D'_1)\]
and
\[\phi':CKh^\sharp(D'_1)\rightarrow \CKh^\sharp(D''_1)\]
be maps given by 0, 1, 2-handle attachment, planar isotopy, or add/drop operations.
In addition, let $\phi^{\textrm{disj}}$  and $\phi'^{\textrm{disj}}$ denote corresponding maps on the cube complexes for link diagrams by taking a disjoint union with $D_{2}$.  
We consider the composition of diagrams
\[
    \xymatrix@C=40pt{
        E^{1}(\CKh^\sharp(D_1)) \otimes E^{1}(\CKh^\sharp(D_2))
        \ar[r]^-{E^{1}(\phi)\otimes \id}
        \ar[d]_-{E^{1}(\Psi)}
        & E^{1}(\CKh^\sharp(D'_1)) \otimes E^{1}(\CKh^\sharp(D_2))
         \ar[r]^-{E^{1}(\phi')\otimes \id}
        \ar[d]^-{E^{1}(\Psi')}
        & E^{1}(\CKh^\sharp(D''_1)) \otimes E^{1}(\CKh^\sharp(D_2))
        \ar[d]^-{E^{1}(\Psi'')}
        \\
        E^{1}(\CKh^\sharp(D_1 \sqcup D_2))
        \ar[r]^-{E^{1}(\phi^{\textrm{disj}})}
        & E^{1}(\CKh^\sharp(D'_1 \sqcup D_2))
        \ar[r]^-{E^{1}(\phi'^{\textrm{disj}})}
        & E^{1}(\CKh^\sharp(D''_1 \sqcup D_2)).
    }
\]
 
\Cref{012handle} and \Cref{Add-Drop} state that each square commutes up to filtered chain homotopy.
Hence, the composed square diagram also commutes up to filtered chain homotopy.
The statement follows since the map $\phi^{KM}_{S}$ for any fundamental cobordism $S$ is defined as (a composition of) maps associated to 0, 1, 2-handle attachment, planar isotopy, add/drop-operation.

\end{proof}

\subsection{Proof of \cref{compatibility with Khovanov}}
Finally, we shall prove \cref{compatibility with Khovanov}, which was the final key lemma to see \cref{thm:main}. 
\begin{lem}[\cref{compatibility with Khovanov}]
If $S$ is either of $0,1,2$-handle attachments, then the diagrams (\ref{eq:h-filt}) and (\ref{eq:q-filt}) are commutative.
\end{lem}
\begin{proof}
The $0$ and $2$-handle cases follow from \cref{disjoint lemma} combined with \cref{exci e1}. 
Thus, we focus on the cobordism map for the $1$-handle case. 
Suppose $S: D \rightarrow D'$ is a diagrammatic $1$-handle attachment. Then, one can easily check $E^1(\phi^\sharp_S) = I^\sharp(S)$. It 
is showed in \cite[Proposition 8.11]{KM11u} that $I^\sharp(S)$ coincides with the corresponding map in Khovanov homology. This completes the proof. 

\end{proof}

%% file: A.tex
\appendix
\section{Orientations of moduli spaces}\label{ori app}
In this section, we briefly introduce several notations to describe the orientation of moduli spaces in our construction. Again, we follow Kronheimer--Mrowka's formulation \cite{KM11u}.

For a given finite-dimensional vector space $V$ over $\R$,  we can associate its \textit{determinant}  $\textrm{det}V$ by setting 
$\textrm{det}V:=\Lambda^{\max}(V)$.
For two vector spaces $V$ and $W$, there is a natural isomorphism
\begin{equation}\label{det-isom}
 \textrm{det}V\otimes \textrm{det}W\cong \textrm{det}(V\oplus  W)   
\end{equation}
which is given by $v\otimes w \mapsto v \wedge w$.
We associate the set of orientations $\Lambda_{V}:=\{o_{V}, -o_{V}\}$ with a vector space, where $o_{V}$ is a specified orientation of $V$ and $-o_{V}$ is another orientation on $V$.
Note that $\Lambda_{V}\otimes_{\Z/2} \Lambda_{W}$ forms a two-point set, and the natural isomorphism (\ref{det-isom}) induces 
\[\Lambda_{V\oplus W}\cong  \Lambda_{V}\otimes_{\Z/2} \Lambda_{W}.\]
Under this natural identification, we can regard the product $o_{V}\otimes_{\Z/2} o_{W}$ as an element of $\Lambda_{V\oplus W}$, where $o_{V}\in \Lambda_{V}$ and $o_{W}\in \Lambda_{W}$.
Moreover, there is a natural identification $\Lambda_{V\oplus W}=\Lambda_{W\oplus V}$, and hence we can compare two elements $o_{V}\otimes_{\Z/2}o_{W}$ and $o_{W}\otimes_{\Z/2}o_{V}$.
In particular, we have 
\begin{eqnarray}\label{commori}
    o_{V}\otimes_{\Z/2} o_{W}=(-1)^{\mathrm{dim}(V)\mathrm{dim}(W)}o_{W}\otimes_{\Z/2}o_{V}.
\end{eqnarray}

For a finite-dimensional orientable connected smooth manifold $M$, the determinant bundle $\textrm{det}(TM)$ is trivial, and the choice of the orientation of $\textrm{det}(TM)$ is called an orientation of $M$.
For a finite-dimensional, orientable, connected smooth manifold $M$, let us denote $\Lambda(M)$ by the two-point set of the orientations of $M$.

In the context of gauge theory,  the orientations of moduli spaces are given by index theory.
Let $(W, S, \bf{P})$ be a cobordism of pairs from $(Y_{0}, K_{0}, \bf{P}_{0})$ to $(Y_{1}, K_{1}, \bf{P}_{1})$ with a singular bundle data $\bf{P}$.
For $i=0, 1$, let $\beta_{i}$ be a critical point of the Chern--Simons functional on $\mathcal{B}(Y_{i}, K_{i}, \bf{P}_{i})$ respectively. 
Firstly, we consider the case when $W$ , $Y_{0}$, and $Y_{1}$ are connected.
A family of Fredholm operators $\{D_{A}\}$ parametrized by points of $\mathcal{B}_{z}(W, S,\mathbf{P};\beta_{0},\beta_{1})$ forms the determinant line bundle 
\[\lambda_{z}(W, S,\mathbf{P};\beta_{0},\beta_{1}):=\mathrm{det\ ind}\{D_{A}\}\]
and this formally orients the moduli space $M_{z}(W,S, \mathbf{P};\beta_{0}, \beta_{1})$.

The two point set of orientations of $\lambda_{z}(W,S, \mathbf{P};\beta_{0}, \beta_{1})$ is denoted by
\[\Lambda_{z}(W, S, \bf{P}; \beta_{0}, \beta_{1}).\]
Since there is a canonical identification $\Lambda_{z}(W, S, \mathbf{P}; \beta_{0}, \beta_{1})\cong \Lambda_{z'}(W, S, \bf{P}; \beta_{0}, \beta_{1})$ for any choice of homotopy classes $z$ and $z'$, we may drop $z$ from the notation.

In particular, if both $\beta_{0}$ and $\beta_{1}$ coincide with the associated base points $\theta$ and $\theta'$, we simply write
\[\Lambda(W, S, \bf{P}).\]

    We call the choice of elements in $\Lambda(W, S, \mathbf{P}; \beta_{0}, \beta_{1})$  \textit{a formal orientation} of the moduli space $M(W, S, \mathbf{P}, \beta_{0}, \beta_{1})$.
    If the moduli space $M(W, S, \mathbf{P}, \beta_{0}, \beta_{1})$ is regular, a formal orientation gives an orientation as a smooth manifold.
   
    In particular, the restriction of the formal orientation $o(W, S, \mathbf{P}, \beta_{0}, \beta_{1})$ to $\Lambda_{z}(W, S, \mathbf{P}; \beta_{0}, \beta_{1})$
    defines an element $o_{z}(W, S, \mathbf{P}, \beta_{0}, \beta_{1})$.
    For a regular moduli space, the element $o_{z}(W, S, \mathbf{P}, \beta_{0}, \beta_{1})$ gives an orientation of the moduli space contained in the homotopy class $z$.

Let us recall the following notion:
\begin{defn}[$I$-orientation, \cite{KM11u}]

   Let $(Y_{0}, K_{0}, \bf{P}_{0}, \bf{a}_{0} )$ and $(Y_{1}, K_{1}, \bf{P}_{1}, \bf{a}_{1} )$ are three-manifolds with singular bundle data and auxiliary data $\bf{a}_{i}$, $(i=0, 1)$ respectively.
   Let $(W, S, \bf{P})$ be a cobordism of pairs from $(Y_{0}, K_{0}, \bf{P}_{0})$ to $(Y_{1}, K_{1}, \bf{P}_{1})$.
   Then a choice of an element in $\Lambda(W, S, \bf{P})$ is called an \textit{$I$-orientation} of $(W, S, \bf{P})$.
 \end{defn}
 The reader can refer detailed definition of singular bundle data and auxiliary data in \cite{KM11u}.
Consider the composition of two cobordisms of pairs. 
Then we have a canonical isomorphism
\[\Lambda_{z}(W, S,\mathbf{P};\alpha, \beta)\otimes_{\Z/2}\Lambda_{z'}(W', S',\mathbf{P}';\beta, \gamma)\xrightarrow{\cong}\Lambda_{z'\circ z}(W'\circ W, S'\circ S, \mathbf{P'\circ P}; \alpha, \gamma)\]
that essentially follows from the additive formula of Fredholm indices.
This identification enables us to write a formal orientation  $o_{z\circ z'}(W'\circ W, S'\circ S, \mathbf{P'\circ P}; \alpha, \gamma)$ as a product of two formal orientations:
\[\pm o_{z}(W, S,\mathbf{P};\alpha, \beta)\otimes_{\mathbb{Z}/2}o_{z'}(W', S',\mathbf{P}';\beta, \gamma)\]
up to sign.
This extends to the canonical identification:
\[\Lambda(W, S,\mathbf{P};\alpha, \beta)\otimes_{\Z/2}\Lambda(W', S',\mathbf{P}';\beta, \gamma)\xrightarrow{\cong}\Lambda(W'\circ W, S'\circ S, \mathbf{P'\circ P}; \alpha, \gamma).\]
For a cylindrical cobordism $(W, S)=[0, 1]\times(Y, K)$, we  simply write
\[\Lambda_{z}(\alpha, \beta):=\Lambda_{z}([0, 1]\times Y, [0, 1]\times K;\alpha, \beta),\]
and 
\[\Lambda_{z}(\beta):=\Lambda_{z}(\beta, \theta)\]
for a fixed base critical point $\theta$. 
In particular, we have the following canonical identification:
\begin{eqnarray}\label{ori-rel}
    \Lambda(\beta)\otimes_{\Z/2}\Lambda(W, S,\mathbf{P};\beta, \beta')\cong \Lambda(W, S, \mathbf{P})\otimes_{\Z/2}\Lambda(\beta')
\end{eqnarray}
The above identification gives us the following consequence:
For a given $I$-orientation for $(W, S, \mathbf{P})$, and fixed orientations for $\beta$ and $\beta'$, the moduli space $M(W, S, \mathbf{P}; \beta, \beta')$ is formally oriented in a canonical way.

Next, we consider a cobordism $(W,S):(Y, K)\rightarrow (Y', K')$ such that $Y$ and $Y'$ are possibly disconnected.
For such a cobordism, we assume that the connected components $Y:=\sqcup_{i}^{m}Y_{i}$ and $Y':=\sqcup_{i}^{n}Y'_{i}$ are ordered by the indices.
Let $\beta$ be a critical point on $Y$, and we write $\beta_{i}$ for its restriction on the component $Y_{i}$.
We write similarly $\beta'_{i}=\beta'|_{Y'_{i}}$ for a critical point $\beta'$ on $Y'$.
Then we define 
\begin{eqnarray*}
    \Lambda(\beta)&:=&\Lambda(\beta_{1})\otimes_{\Z/2}\cdots\otimes_{\Z/2} \Lambda(\beta_{m}),\\
    \Lambda(\beta')&:=&\Lambda(\beta'_{1})\otimes_{\Z/2} \cdots \otimes_{\Z/2} \Lambda(\beta'_{n}).
\end{eqnarray*}
We fix a formal orientation of the moduli space $M(W,S; \mathbf{P};\beta,\beta')$ by the relation (\ref{ori-rel}) replacing $\Lambda(\beta)$ and $\Lambda(\beta')$ as above.

An instanton moduli space parametrized by $G$ is defined as a zero-set of a Fredholm section $\phi$ of the Banach bundle \[\mathbf{P}_{G}\rightarrow \mathcal{B}(W,S,\mathbf{P};\alpha, \beta)\times G\]
which is locally defined by $\phi([A], g):=([A], g,  F^{+_{g}}_{A})$.
We write $\phi^{-1}(0)=:M_{G}(W,S,\mathbf{P};\alpha, \beta)$.
In particular, for a fixed homotopy class $z$ over the cobordism $(W,S,\mathbf{P};\alpha, \beta)$, we define
\[
M_{z,G}(W,S,\mathbf{P};\alpha, \beta):=\phi^{-1}(0)\cap \mathcal{B}_{z}(W,S,\mathbf{P};\alpha, \beta)\times G.
\]
Assume $G$ is a smooth connected orientable manifold. 
Then there is a usual orientation bundle for $G$, and whose set of orientations are denoted by $\Lambda(G)$.
In case of $G$ is a cornered smooth manifold, then we define $\Lambda(G)$ as the orientation of its interior $\mathrm{int}(G)$ instead.
Assume we fix an orientation 
$o(G)\in \Lambda(G)$.
We define formal orientation of $M_{z, G}(W,S,\mathbf{P};\alpha, \beta)$ by  
\[o_{z}(W,S,\mathbf{P};\alpha, \beta)\otimes_{\Z/2}o(G)\]
as in an element of $\Lambda_{z}(W,S,\mathbf{P};\alpha, \beta)\otimes_{\Z/2}\Lambda(G)$.

For a formal orientation $o_{z}(W, S, \mathbf{P};\alpha, \beta)\in \Lambda_{z}(W, S, \mathbf{P};\alpha, \beta)$, we define an integer 
\[
\mathrm{ind}(o_{z}(W, S, \mathbf{P};\alpha, \beta)):=\mathrm{ind}D_{A}
\]
where $[A]\in \mathcal{B}_{z}(W, S, \mathbf{P};\alpha, \beta)$. 
Note that the integer $\mathrm{ind}(o_{z}(W, S, \mathbf{P};\alpha, \beta))$ depends on the choice of the path $z$ up to homotopy.
For a finite-dimensional orientable manifold $M$, we define $\mathrm{ind}(o_{M}):=\mathrm{dim}M$ for $o_{M}\in \Lambda(M)$.
Let $o_{1}$ and $o_{2}$ be given orientations of finite-dimensional manifold or formal orientations of moduli spaces.
Then the similar consideration as (\ref{commori}) implies that
\[o_{1}\otimes_{\Z/2}o_{2}=(-1)^{\mathrm{ind}(o_{1})\mathrm{ind}(o_{2})}o_{2}\otimes_{\Z/2}o_{1}.\]

Finally, we discuss the orientation induced on boundary faces of compactified moduli spaces.
Assume that a family of metric $G$ has a boundary face of the form of product $G_1 \times G_2$, and the instanton moduli space $M_{G}$ parameterized by $G$ has a compactification $M^{+}_G$ whose boundary faces are diffeomorphic to the form $M_{G_{1}}\times M_{G_{2}}$.
Then, gluing theory \cite[Section 5]{Do02} gives a local diffeomorphism 
\[[T, \infty)\times M_{G_{1}}\times M_{G_{2}}\rightarrow M_{G}\]
under the assumption that the regularity of moduli spaces holds.
The half interval $[T, \infty)$ models an outward normal vector of the boundary of the compactified moduli space. 
This local diffeomorphism gives an identification between the set of orientations:
\[\Lambda_{[T, \infty)}\otimes_{\Z/2}\Lambda_{M_{G_{1}}}\otimes_{\Z/2}\Lambda_{M_{G_{2}}}\cong \Lambda_{M_{G}}.\]
In particular, we can identify the orientation of the boundary face $o(M_{\partial G})$ as an element in $\Lambda_{M_{G_{1}}}\otimes_{\Z/2}\Lambda_{M_{G_{2}}}$.
We orient the boundary face $M_{\partial G}$ by the convention:
\[o_{+}\otimes_{\Z/2}o(M_{\partial G})=o(M_{G}) \]
where $o_{+}$ is the orientation of $[T, \infty)$ of the positive direction.

In particular, the orientation of $M_{G_{1}}\times M_{G_{2}}$ as a boundary face is differ by the factor $(-1)^{\mathrm{dim}G_{1}\mathrm{dim}G_{2}}$ from its product orientation.

\subsection{Trace of diagrammatic deformations.}\label{ori_of_KMmap}

Let $T_{uv}: K_{u}\rightarrow K'_{v}$ be a link cobordism introduced in \Cref{isotopy_trace}.
We discuss the boundary orientation of the space of metric $G^{T}_{uv}$ and the instanton moduli space $M^{T}_{uv}(\alpha, \beta)$.
Firstly, note that there are orientation-preserving identifications:
\[G^{T}_{uv}\cong G^{T}_{uv'}\times G_{v'v},\ G^{T}_{uv}\cong G_{uu'}\times G^{T}_{u'v}.\]
This implies the following lemma.
\begin{lem}
    The boundary face $G^{T}_{uv'}\times \breve{G}_{v'v}$ has an orientation differ by the factor $(-1)^{\mathrm{dim}{G}^{T}_{uv'}}$ from the product orientation. On the other hand, the boundary face $\breve{G}_{uu'}\times G^{T}_{u'v}$ has an orientation differ by the factor $-1$ from the product orientation.
\end{lem}
\begin{proof}
    The first half statement follows from the computation:
    \[o_{\nu}\otimes_{\Z/2}o_{G^{T}}(u, v')\otimes_{\Z/2}o_{\breve{G}}(v',v)=(-1)^{\mathrm{dim}G^{T}_{uv'}}o_{G^{T}}(u,v')\otimes_{\Z/2}o_{+}\otimes_{\Z/2}o_{\breve{G}}(v', v),\]
    and the identifications
    \[o_{G}(v', v)=o_{+}\otimes_{\Z/2}o_{\breve{G}}(v', v), \ o_{G^{T}}(u, v)=o_{G^{T}}(u, v')\otimes_{\Z/2}o_{G}(v', v).\]
    For the second half statement, note that $o_{\nu}=-o_{+}$ and the rest of the argument is similar.
\end{proof}
\begin{lem}
    There is a choice of the orientations $o_{z}(T_{uv};\alpha, \beta)$ such that
    \begin{itemize}
        \item $o_{z}(T_{uw};\alpha, \gamma)=o_{z'}(T_{uv};\alpha, \beta)\otimes_{\Z/2}o_{z''}(S_{vw};\beta,\gamma)$
        \item $o_{z}(T_{uw};\alpha, \gamma)=o_{z'}(S_{uv};\alpha, \beta)\otimes_{\Z/2}o_{z''}(T_{vw};\beta,\gamma) $
    \end{itemize}
    where the path $z$ is the concatenation of the paths $z'$ and $z''$.
\end{lem}
\begin{proof}

Choose paths 
   $z_0: \gamma \rightarrow \theta'$
over the cylinder and write the concatenations $z_1=z_0 \circ z$ and $z_2=z_0 \circ z''$.
Note that $z_1=z_2 \circ z'$ holds.
Since we have natural identifications of $I$-orientations:
\begin{eqnarray*}
    o_{z}(T_{uw};\alpha, \gamma)\otimes_{\Z/2}o_{z_{0}}(\gamma)&=&o_{z_{1}}(\alpha)\otimes_{\Z/2}o_{I}(T; u, w)\\
    &=&o_{z_{1}}(\alpha)\otimes_{\Z/2}o_{I}(T;u, v)\otimes_{\Z/2}o_{I}(v,w)\\
    &=&o_{z'}(T_{uv};\alpha,\beta)\otimes_{\Z/2}o_{z_{2}}(\beta)\otimes_{\Z/2}o_{I}(v,w)\\
    &=&o_{z'}(T_{uv};\alpha, \beta)\otimes_{\Z/2}o_{z''}(S_{vw};\beta,\gamma)\otimes_{\Z/2}o_{z_{0}}(\gamma).
\end{eqnarray*}
The first item follows from this computation, and the proof of the second item is similar.
\end{proof}
\begin{prop}\label{KMT_orientation}
    The oriented components of $M_{\partial G^{T}_{uv}}(\alpha,\beta)$ consists of 
    \begin{itemize}
        \item[(i)] $(-1)^{1+\mathrm{dim}\breve{G}_{uu'}\mathrm{dim}G^{T}_{u'v}}\breve{M}_{uu'}(\alpha, \beta)_{0}\times M^{T}_{u'v}(\beta, \gamma)_{0}$
        \item[(ii)] $(-1)^{\mathrm{dim}G^{T}_{uv'}(\mathrm{dim}\breve{G}_{v'v}+1)}M^{T}_{uv'}(\alpha, \beta)_{0}\times\breve{M}_{v'v}(\beta, \gamma)_{0}$ 
    \end{itemize}
\end{prop}
\begin{proof}
    We prove the case (i). The proofs of other cases are similar.
    By the convention of the boundary orientations, the codimension-1 face $\breve{M}_{uu'}(\alpha, \alpha')_{0}\times M_{u'v;w}(\alpha', \beta;\gamma)_{0}$ is oriented by the local identification of ${M}_{uv; w}(\alpha, \beta;\gamma)_{1}$ and 
    \begin{eqnarray}\label{bdly1}
  \mathbb{R}_{\leq 0}\times \breve{M}_{uu'}(\alpha, \beta)_{0}\times {M}^{T}_{u'v}(\beta, \gamma)_{0}. 
    \end{eqnarray}
    Here $\mathbb{R}_{\leq 0}$ models outward normal vectors and is assumed to be oriented in the positive direction.
    To distinguish orientations, we write (\ref{bdly1}) as 
    \begin{eqnarray}\label{bdly2}
  \mathbb{R}_{\leq 0}\times [\breve{M}_{uu'}(\alpha, \alpha')_{0}\times {M}_{u'v;w}(\alpha', \beta;\gamma)_{0}]_{\partial}. 
    \end{eqnarray}
    instead.
    On the other hand, we consider the product space of the same underlying space as (\ref{bdly1}) but equipped with a product orientation on the factor $\breve{M}_{uu'}(\alpha, \alpha')_{0}\times {M}_{u'v;w}(\alpha', \beta;\gamma)_{0}$. 
    We write this space as
    \begin{eqnarray}\label{bdly3}
  \mathbb{R}_{\leq 0}\times [\breve{M}_{uu'}(\alpha, \alpha')_{0}\times {M}^{T}_{u'v}(\beta,\gamma)_{0}]_{\mathrm{prod}}. 
    \end{eqnarray}
    Since the gluing map : \[\mathrm{Gl}: (-\infty, -T)\times \breve{M}_{uu'}(\alpha, \beta)_{0}\times {M}^{T}_{u'v}( \beta,\gamma)_{0}\rightarrow {M}^{T}_{uv}(\alpha, \gamma)_{1}\]
     gives a local model of the ends of the moduli space ${M}^{T}_{uv}(\alpha, \gamma)_{1}$, we can compare the orientation of (\ref{bdly3}) and ${M}^{T}_{uv}(\alpha, \gamma)_{1}$.
     By our convention, the orientation of ${M}^{T}_{uv}(\alpha, \gamma)_{1}$ is identified with
     $o_{+}\otimes_{\mathbb{Z}/2}o_{\partial M_{G}}=(-1)^{1-\mathrm{dim}{G}^{T}_{uv}}o_{+}\otimes_{\mathbb{Z}/2}o_{ M_{\partial G}}$.
     Moreover, the second factor is identified with
     $o_{z}(S_{uv};\alpha, \gamma)\otimes_{\mathbb{Z}/2}o(\partial_{u'}{G}_{uv})$ where $\partial_{u'}G^{T}_{uv}=\breve{G}_{uu'}\times G^{T}_{u'v}$
      Since we only consider pseudo-diagrams, there is a natural identification between orientations
     $o_{z}(S_{uv};\alpha, \gamma)=o_{z'}(S_{uu'};\alpha, \beta)\otimes _{\mathbb{Z}/{2}}o_{z''}(S_{u'v}; \beta, \gamma)$.
     Moreover, the boundary orientation of the face $\breve{G}_{uu'}\times {G}^{T}_{u'v}$ in ${{G}^{T}}^{+}_{uv}$ is differ by $(-1)$.
     Hence, the orientation of (\ref{bdly3}) is differ by the factor $(-1)^{1+\mathrm{dim}\breve{G}_{uu'}\mathrm{dim}{G}^{T}_{u'v}+(1-\mathrm{dim}{G}^{T}_{uv})}$.
     Since the orientation of $\partial M_{G}$ and $M_{\partial G}$
     are differ by the factor $(-1)^{1-\mathrm{dim}{G}^{T}_{uv}}$,
     the orientation of a component $\breve{M}_{uu'}(\alpha,\beta)_{0}\times {M}^{T}_{u'v}(\beta, \gamma)_{0}$
     in $M_{\partial G}$ is differ by the factor $(-1)^{1+\mathrm{dim}\breve{G}_{uu'}\mathrm{dim}{G}^{T}_{u'v}}$
     from the boundary orientation. 
\end{proof}

\subsection{Orientations for excision cobordisms}
Next, we discuss the orientation of moduli spaces associated to excision cobordisms.
For given resolution $u$, $v$ and $w$, we write the excision cobordism as $S_{uv;w}$ for short.
Firstly, note that compatible relations exist in the choice of $I$-orientations as follows.

\begin{lem}\label{compatible i-ori}
There exists a choice of $I$-orientations such that
    \begin{itemize}
    \item $o_{I}(u, u')\otimes_{\Z/2} o_{I}(v, v')\otimes_{\Z/2} o_{I}(u', v'; w)=o_{I}(u, v; w) $,
    \item $o_{I}(u, v;w')\otimes_{\Z/2} o_{I} (w', w)=o_{I} (u, v;w)$.
\end{itemize}
\end{lem}
\begin{proof}
    Considering as a pair of resolutions $(u,v)$ is a vertex of $\Z^{N_1}\times \Z^{N_{2}}\cong \Z^{N_{1}+N_{2}}$, the statements essentially follows from the argument of the proof of \cite[]{KM11u}. Therefore we omit the detail. 
\end{proof}
The next lemma also follows from the original argument given in \cite{KM11u}: 
\begin{lem}\label{compatible ori}
    There exists a choice of orientations of the space of metrics such that
    \begin{itemize}
         \item $o_{G}(u, u')\otimes_{\Z/2} o_{G}(v, v')\otimes_{\Z/2} o_{G}(u', v'; w)=o_{G}(u, v; w) $,
    \item $o_{G}(u, v;w')\otimes_{\Z/2} o_{G} (w', w)=o_{G} (u, v;w)$.
    \end{itemize}
\end{lem}
Assume that the resolutions $u,v$, and $w$ satisfy $uv\geq w\geq -1$.
We define an orientation $o_{G}(u, v; w)$ on $G_{uv;w}$ by the following identification:
\begin{align}\label{iden_ex}
o_{G}(u, \mathbf{-1})\otimes_{\Z/2}o_{G}(v, \mathbf{-1})\otimes_{\Z/2}o_{G}(\mathbf{-1}\mathbf{-1};\mathbf{-1})=o_{G}(u, v; w)\otimes_{\Z/2}o_{G}(w, \mathbf{-1})
\end{align}
where $\mathbf{-1}$ denotes the resolution $(-1, \cdots ,-1)$.

\begin{prop}
    Let $n$ be the number of crossings for a pseudo-diagram $D$.
    The oriented boundary faces in $\partial G^{+}_{uv;w}$ consists of the following types:
    \begin{itemize}
        \item $-\breve{G}_{uu'}\times G_{u'v;w}$
        \item $(-1)^{1+|v-v'|_{1}(n-|u|_{1})}\breve{G}_{vv'}\times G_{uv';w}$
        \item $(-1)^{|uv-w'|_{1}}G_{uv;w'}\times \breve{G}_{w'w}$
    \end{itemize}  
    
\end{prop}

\begin{proof}
The identification \eqref{iden_ex} 
implies 
    \[o_{G}(u,u')\otimes_{\Z/2}o_{G}(u', v;w)=o_{G}(u,v;w).\]
    Since $o_{G}(u,u')=o_{+}\otimes_{\Z/2}\breve{o}_{G}(u,u')=-o_{\nu}\otimes_{\Z/2}\breve{o}_{G}(u,u')$, the boundary orientation of the face $\breve{G}_{uu'}\times G_{u'v;w}$ is differ by the factor $(-1)$ from the product orientation.
    
    For the second item, note that we have a natural identification:
    \[(-1)^{\mathrm{dim}G_{vv'}\cdot \mathrm{dim}{G}_{u-1}}o_{G}(v,v')\otimes_{\Z/2}o_{G}(u, v';w)=o_{G}(u,v;w).\]
    Here, the factor $(-1)^{\mathrm{dim}G_{vv'}\cdot \mathrm{dim}{G}_{u-1}}=(-1)^{|v-v'|_{1}(n-|u|_{1})}$ arises from the operation switching $o_{G}(v,v')$ and $o_{G}(u, \mathbf{-1})$, and the rest of the argument is similar.

    For the third item, note that we have a natural identification:
    \[o_{G}(w', w)=o_{+}\otimes_{\Z/2}\breve{o}_{G}(w',w)=o_{\nu}\otimes_{\Z/2}\breve{o}_{G}(w',w).\]
    Applying the above identification to \ref{compatible ori} and switching $o_{G}(u, v;w')$ and $o_{\nu}$ give rise the factor $(-1)^{\mathrm{dim}G_{uv;w'}}=(-1)^{|uv-w'|_{1}}$.
\end{proof}
Now, we can compare the original orientation of moduli space and the composed orientation.
\begin{lem} \label{excision_orientation}
For cobordisms between pseudo-diagrams, the following holds: 
    \begin{itemize}
        \item $o_{z_{13}}(S_{uv;w};\alpha, \beta;\gamma)=
        o_{z_{12}}(S_{uu'};\alpha, \alpha')\otimes_{\Z/2}o_{z_{23}}(S_{u'v;w};\alpha', \beta;\gamma)$
        \item $o_{z_{13}}(S_{uv;w};\alpha, \beta;\gamma)=
        o_{z_{12}}(S_{vv'};\beta, \beta')\otimes_{\Z/2}o_{z_{23}}(S_{uv';w};\alpha, \beta';\gamma)$
        \item $o_{z_{13}}(S_{uv;w};\alpha, \beta;\gamma)=o_{z_{12}}(S_{uv;w'};\alpha, \beta;\gamma')\otimes_{\Z/2}o_{z_{23}}(S_{w'w}; \gamma', \gamma)$
    \end{itemize}
\end{lem}
\begin{proof}
We give the proof for the first item.
    Our convention and Lemma \ref{excision_orientation} implies
    \[o(\alpha)\otimes_{\Z/2}o(\beta)\otimes_{\Z/2}o_{I}(u, u')\otimes_{\Z/2} o_{I}(u', v; w)=o_{z_{13}}(S_{uv;w};\alpha, \beta;\gamma)\otimes_{\Z/2}o(\gamma).\]
    Note that we have the relation
    \[o(\beta)\otimes_{\Z/2}o_{I}(u, u')=(-1)^{\mathrm{gr}(\beta)\cdot \mathrm{ind}(S_{uu'})}o_{I}(u, u')\otimes_{\Z/2}o(\beta)\]
    where $\mathrm{ind}(S_{uu'})$ denotes the index of the deformation complex associated with an instanton in the moduli space $M_{z}(S_{uu'};\theta, \theta')$ for some homotopy class $z$. 
    Since we assume that links appear in the boundaries of the cobordisms are pseudo-diagrams, $\mathrm{gr}(\beta)$ is even. Finally, the orientation convention for $o_{z_{12}}(S_{uu'};\alpha, \alpha')$ and $o_{z_{23}}(S_{u'v;w};\alpha', \beta;\gamma)$ implies the desired relation.
   
\end{proof}

\begin{prop}\label{face of cpt moduli-1}
    
Assume that all resolutions are pseudo-diagrams.
    Then the associated boundary face of $\breve{M}_{uv;w}(\alpha, \beta;\gamma)_{1}$
    over (i) $\breve{G}_{uu'}\times G_{uu';w}$, (ii) $\breve{G}_{vv'}\times G_{uv':w}$, and (iii) $G_{uv;w'}\times \breve{G}_{w'w}$
    are diffeomorphic to the following oriented component in the family of moduli space $M_{\partial {G}_{uv;w}}(\alpha, \beta;\gamma)_{0}$.

\begin{itemize}
    \item[(i)]  $(-1)^{1+\mathrm{dim}\breve{G}_{uu'}\mathrm{dim}{G}_{u'v;w}}  \breve{M}_{uu'}(\alpha, \alpha')_{0} \times {M}_{u'v;w}(\alpha', \beta; \gamma)_{0}$,
    \item[(ii)] $(-1)^{1+\mathrm{dim}G_{vv'}\mathrm{dim}G_{u-1}+\mathrm{dim}\breve{G}_{vv'}\mathrm{dim}{G}_{uv';w}}  \breve{M}_{vv'}(\beta, \beta')_{0} \times {M}_{uv'; w}(\alpha, \beta'; \gamma)_{0}$,
    \item[(iii)] $(-1)^{\mathrm{dim}{G}_{uv;w'}(\mathrm{dim}\breve{G}_{w'w}+1)} {M}_{uv;w'}(\alpha, \beta; \gamma')_{0} \times \breve{M}_{w'w}(\gamma', \gamma)_{0}$.
\end{itemize}

\end{prop}

\begin{proof}
The proof is parallel to that of Proposition \ref{KMT_orientation}.
\end{proof}

A similar argument implies the following proposition.

\begin{prop}
    Assume that all resolutions are pseudo-diagrams.
    Then the associated boundary face of ${M}_{uv;w}(T_{uv;w}; \alpha, \beta;\gamma)_{1}$
    over (i) $\breve{G}_{uu'}\times G^{T}_{uu';w}$, (ii) ${G}^{T}_{uu'}\times G_{uu';w}$
    (iii) $\breve{G}_{vv'}\times G^{T}_{uv':w}$, (iv) $G^{T}_{uv;w'}\times \breve{G}_{w'w}$, and
    (v) $G_{uv;w'}\times {G}^{T}_{w'w}$
    are diffeomorphic to the following oriented component in the family of moduli space $M_{\partial {G}_{uv;w}}(\alpha, \beta;\gamma)_{0}$.

\begin{itemize}
    \item[(i)]  $(-1)^{1+\mathrm{dim}\breve{G}_{uu'}\mathrm{dim}{G}^{T}_{u'v;w}}  \breve{M}_{uu'}(\alpha, \alpha')_{0} \times {M}^{T}_{u'v;w}(\alpha', \beta; \gamma)_{0}$,
    \item[(ii)] $(-1)^{1+\mathrm{dim}{G}^{T}_{uu'}\mathrm{dim}{G}_{u'v;w}}  {M}^{T}_{uu'}(\alpha, \alpha')_{0} \times {M}_{u'v;w}(\alpha', \beta; \gamma)_{0}$,
    \item[(iii)] $(-1)^{1+\mathrm{dim}G_{vv'}\mathrm{dim}G_{u-1}+\mathrm{dim}\breve{G}_{vv'}\mathrm{dim}{G}^{T}_{uv';w}}  \breve{M}_{vv'}(\beta, \beta')_{0} \times {M}^{T}_{uv'; w}(\alpha, \beta'; \gamma)_{0}$,
    \item[(iv)] $(-1)^{\mathrm{dim}{G}^{T}_{uv;w'}(\mathrm{dim}\breve{G}_{w'w}+1)} {M}^{T}_{uv;w'}(\alpha, \beta; \gamma')_{0} \times \breve{M}_{w'w}(\gamma', \gamma)_{0}$,
    \item[(v)] $(-1)^{\mathrm{dim}{G}_{uv;w'}(\mathrm{dim}{G}^{T}_{w'w}+1)} {M}_{uv;w'}(\alpha, \beta; \gamma')_{0} \times {M}^{T}_{w'w}(\gamma', \gamma)_{0}$. \qed
\end{itemize}
\end{prop}